\documentclass{article}
\usepackage[utf8]{inputenc}
\usepackage[left=3cm,right=3cm,top=3cm,bottom=3cm]{geometry}
\usepackage{graphicx}
\usepackage{enumitem}
\usepackage{color}
\usepackage{amsmath,amsfonts,amssymb,amsthm,epsfig,epstopdf,titling,url,array}
\usepackage{mathrsfs}
\usepackage{caption}
\usepackage{subcaption}
\usepackage{tikz-cd}
\usepackage{MnSymbol}

\DeclareMathAlphabet{\mathpzc}{OT1}{pzc}{m}{it}
\DeclareMathOperator{\lcm}{lcm}
\DeclareMathOperator{\orb}{Orb}
\DeclareMathOperator{\supp}{Supp}

\newcommand{\Mane}{Mañé }

\def\XXint#1#2#3{{\setbox0=\hbox{$#1{#2#3}{\int}$ }
\vcenter{\hbox{$#2#3$ }}\kern-.6\wd0}}

\newtheorem{theor}{Theorem}
\newtheorem{theo}{Theorem}[section]
\newtheorem{prop}[theo]{Proposition}
\newtheorem{lem}[theo]{Lemma}
\newtheorem{cor}[theo]{Corollary}

\theoremstyle{definition}
\newtheorem{defi}[theo]{Definition}

\newtheorem{rem}[theo]{Remark}

\theoremstyle{remark}

\numberwithin{equation}{section}

\title{A Smooth, Recurrent, Non-Periodic Viscosity Solution of the Hamilton-Jacobi Equation}
\author{CHARFI Skander}
\date{}

\begin{document}
\maketitle

\begin{abstract}
	Viscosity solutions of the Hamilton-Jacobi equation were introduced by Lions and Crandall. For Tonelli Hamiltonians, these solutions are generated by the Lax-Oleinik operator. It is known that this operator converges in the autonomous framework, but this convergence fails in the general cases.

	In this paper, we introduce a method to construct smooth, recurrent, non-periodic viscosity solutions on fixed compact manifolds $M$ of dimension 2 or higher. Additionally, we provide a detailed description of the non-wandering set of the Lax-Oleinik operator and identify its action on various omega-limit sets.
\end{abstract}

\tableofcontents

\newpage

\addcontentsline{toc}{section}{Introduction}
\section*{Introduction}

\subsection{Overview}

\par Let $M$ be a connected closed manifold. Denote by $TM$ its tangent bundle and by $T^*M$ its cotangent bundle with their respective projections $\pi_{TM} : TM \to M$ and $\pi_{T^*M} : T^*M \to M$. Unless it is ambiguous, these projections will be simply denoted by $\pi$. Denote by $\mathcal{C}(M,\mathbb{R})$ the set of continuous scalar maps on $M$ endowed with its usual infinite norm $\Vert.\Vert_\infty$ given by $\Vert u\Vert_\infty = \sup_{x \in M} |u(x)|$. The set $\mathbb{T}^1 = \mathbb{R} / \mathbb{Z}$ refers to the circle.\\

Let $H(t,x,p) : \mathbb{T}^1 \times T^*M \to \mathbb{R}$ be a $1$-time periodic Tonelli Hamiltonian (See Definition \ref{TonelliDef}). We consider the \textit{Hamilton-Jacobi equation}
\begin{equation} \label{HJ}
	\partial_tu + H(t,x,\partial_xu) =0
\end{equation}

\par This equation has a well posed Cauchy problem in the viscosity sense (see \cite{MR667669}), \cite{MR0690039}). And in the Tonelli framework, these solutions are generated by the Lax-Oleinik operator $\mathcal{T}_0^{s,t}$ defined on $\mathcal{C}(M,\mathbb{R})$ (see Definition \ref{LODef}). More specifically, for any real time $s \in \mathbb{R}$ and any $u_0 \in \mathcal{C}(M,\mathbb{R})$, there exists a unique viscosity solution $u(t,x):[s,+\infty) \times M \to \mathbb{R}$ of (\ref{HJ}) with $u(s,\cdot) = u_0$ defined by $u(t,x) = \mathcal{T}_0^{s,t}u_0(x)$.\\

It is known that there exists a constant $\alpha_0$ referred to as the \textit{Mañé critical value}, such that for all $(s,u_0) \in \mathbb{R} \times \mathcal{C}(M,\mathbb{R})$, the family of maps $(\mathcal{T}_0^{s,t}u_0 + \alpha_0.(t-s))_{t \geq s}$ is uniformly bounded in $\mathcal{C}(M,\mathbb{R})$ (see Proposition \ref{ViscBounded}, or \cite{Representation} for a proof). Without loss of generality, it is possible to assume that this constant $\alpha_0$ is null. In fact, it suffices to consider the new Hamiltonian $H_0 = H - \alpha_0$ with its associate Lax-Oleinik operator $\mathcal{T}^{s,t} u = \mathcal{T}_0^{s,t}u + \alpha_0.(t-s)$. Thus, this allows to assume that all viscosity solutions are bounded in positive times, making it meaningful to study the fixed, periodic and recurrent elements of the operator $\mathcal{T}:= \mathcal{T}^{0,1}$. The latter is the main focus of this paper.\\

In the autonomous Hamiltonian case, the asymptotic behaviour of the Lax-Oleinik operator $\mathcal{T}$ has been extensively studied in \cite{MR1646936}, \cite{MR1457088}, \cite{MR1485736} and \cite{MR1680905}. This behavior is well understood, primarily due to Albert Fathi’s convergence theorem \cite{MR1650261}. 

\begin{theo} \label{FathiTh} (Fathi \cite{MR1650261})
	For all scalar maps $u \in \mathcal{C}(M, \mathbb{R})$, the sequence $\mathcal{T}^nu$ converges in the $C^0$-topology to a weak-KAM solution $u_\infty$. Moreover, for all time $t >0$, $\mathcal{T}^t u_\infty = u_\infty$.
\end{theo}

Hence, there is always convergence of viscosity solutions to steady states, namely the weak-KAM solutions. We refer to \cite{MR1752423} for an analytical proof. Various extensions of this result have also been explored in \cite{MR1457088}, \cite{MR1485736}, \cite{MR1680905}.\\

However, addressing the time-dependent framework reveals more challenging. According to A. Fathi and J.N. Mather \cite{MR1792479}, such convergence does not hold in the non-autonomous setting. Furthermore, in the one-dimensional case ($M = \mathbb{T}^1$), P. Bernard  and J.-M. Roquejoffre \cite{MR1646936}, \cite{MR2041603} demonstrated that viscosity solutions converge, up to a linear time dependence, to periodic solutions, thereby constraining the possible behaviors in dimension one. Their work elucidates that the periodicity of asymptotic viscosity solutions is intricately linked to the periodicity of the orbits within the Mather set (defined in Subsection \ref{MatherSection}). This set is a dynamical subset of the phase space emerging from Aubry-Mather theory, and plays a central role in weak-KAM theory. \\

Less is known about higher dimensions. This work presents the construction of a smooth, recurrent, non-periodic viscosity solution of the Hamilton-Jacobi equation \eqref{HJ} on any manifold $M$ of dimension $2$ or higher.

\subsection{Main Result}

In this work, we present a construction of a recurrent, non-periodic viscosity solution to the Hamilton-Jacobi equation (\ref{HJ}), applicable to compact manifolds of dimension 2 or higher. The main idea is to make advantage of the correlation between periodic orbits within the Mather set and periodic viscosity solutions. Specifically, for each $\rho_n$-periodic orbit $\{x_n^i\}_i$ within the Mather set, we consider the modified Peierls barriers $h^{\rho_n \infty}(x^0_n,\cdot)$ introduced in \cite{MR1792479}, which is a $\rho_n$-periodic viscosity solution. The recurrent viscosity solution is of the form $u = \inf_{n \geq 0} \{ h ^{\rho_n \infty}(x^0_n, \cdot)\}$. We carefully control the parameters of the constructed Hamiltonian to make sure that $u$ is $\rho_n$-periodic near every point $x^0_n$. The recurrence, yet non-periodicity, of $u$ is achieved by considering an infinite set of orbits with increasingly diverging periodicities $\rho_n$.\\

This idea allows the construction of numerous examples. However, achieving more than Lipschitz regularity for the recurrent solution reveals more intricate. To address this, adjustments are made to the Hamiltonian and the viscosity solution $u$ to achieve $C^\infty$ regularity with a new initial data of the form $u = \inf_{n \geq 0} \{c_n + h ^{\rho_n}(x_n, \cdot)\}$ where $c_n$ are constants to be chosen carefully.\\

More precisely, the weak-KAM theory demonstrates that $u_c$ is already differentiable on the Mather set. Building on this result, we ensure that the boundary between realizations of the infimum defining $u_c$ occurs within the Mather set. Additionally, symmetry conditions on the Hamiltonian are imposed to guarantee the $C^\infty$-regularity of the barriers $h^{\rho_n}(x_n, \cdot)$ outside of the Mather set.\\

Unfortunately, attaining higher regularity on the Mather set is significantly more challenging, as it often acts as an obstruction to achieving regularity higher than $C^{1,1}$. This is exemplified by the double covering of the pendulum (See \cite{charf2025}), where Peierls barriers are smooth everywhere except at one of the two points that shape the Mather set, at which they are never of $C^2$ regularity (due to the hyperbolic nature of these points). To address this issue, we construct the Hamiltonian to induce slow parabolic negative time convergence in the unstable manifolds of points in the Mather set. This approach can result in a null higher derivatives of $u_c$ at such points. This yields the following theorem.

\begin{theo} \label{MainC1}
	For any closed manifold $M$ of dimension $d \geq 2$, there exists a $C^\infty$ Tonelli Hamiltonian $H : \mathbb{T}^1 \times T^*M \to \mathbb{R}$ such that the Lax-Oleinik operator $\mathcal{T}$ admits a $C^\infty$ recurrent, non-periodic viscosity solution.
\end{theo}

After the construction, our focus shifts to the asymptotic behaviour of the Lax-Oleinik operator for the obtained Tonelli Hamiltonian. We determine the omega-limit set $\omega(u)$ of the constructed viscosity solution $u$ and, more broadly, describe the whole non-wandering set $\Omega(\mathcal{T})$ following a generalized representation formula derived from \cite{Representation}. It is observed that when the Mather set consists of periodic orbits with integer periods,the Lax-Oleinik operator behaves like an odometer on its non-wandering set (see \cite{MR0551496}, \cite{MR2180227}, \cite{MR4492832} for surveys on odometers, also known as adding machines). This leads to the following theorem.

\begin{theor} \label{MainOmega}
	Assume that the Mather set $\mathcal{M}$ is the union of periodic orbits with integer periods. Then for all $v \in \Omega(\mathcal{T})$, the restriction of the Lax-Oleinik operator $\mathcal{T}$ to the set $\omega(v)$ is a factor of an odometer.
\end{theor}

It is known that the Lax-Oleinik semigroup $\mathcal{T}$ acts by isometries on its non-wandering set (See \cite{Representation}). This imposes limitations on the possible dynamics that can occur on $\omega(u)$ for any recurrent viscosity solution $u$. A natural question that arises is whether it is possible to construct an explicit example where $\mathcal{T}$ exhibits more intricate dynamics than that of an odometer.

\subsection{Structure of the Article}

Section \ref{SectionTools} introduces preliminary tools from weak-KAM theory and Aubry-Mather theory. In Section \ref{SectionManeHamiltonian}, following ideas from Mañé, we construct a Tonelli Hamiltonian for which $\Omega(\mathcal{T})$ contains a smooth, non-wandering, non-periodic viscosity solution. Section \ref{SectionRecurrence} is dedicated to the main construction of a recurrent, non-periodic, yet non-regular viscosity solution $u$, along with a description of its $\omega$-limit set. We then extend this description to the non-wandering set of the Lax-Oleinik semigroup $\mathcal{T}$. Finally, in Section \ref{RegSection}, we adjust the initial condition to obtain a smooth, recurrent, non-periodic viscosity solution.\\

\section{Preliminary Tools and Properties} \label{SectionTools}

This section compiles the primary tools and propositions used in the construction process, drawing from the Weak-KAM and Aubry-Mather theories. Most of these propositions were comprehensively presented in \cite{Representation}. We will restate, without proofs, the properties necessary for constructing more regular recurrent viscosity solutions. Furthermore, we will introduce a more detailed subsection devoted to the \Mane set.

\subsection{Tonelli Hamiltonians and the Lax-Oleinik Operator} \label{SubsectionTonelli}

We define the notion of Tonelli Hamiltonians, initially introduced in \cite{MR1109661}. These are Hamiltonians that are convex and superlinear in the fibres, serving as the tame framework for the weak-KAM theory.

\begin{defi} \label{TonelliDef}
	A $1$-time-periodic Hamiltonian $H(t,x,p) : \mathbb{T}^1 \times T^*M \to \mathbb{R}$ is said \textit{Tonelli} if it satisfies the following classical hypotheses :
	\begin{itemize}
		\item Regularity: $H$ is $\mathcal{C}^2$.
		\item Strict convexity: $\partial_{pp}H(t,x,p)>0$ for all $(t,x,p) \in \mathbb{T}^1 \times T^*M$.
		\item Superlinearity: $H(t,x,p)/|p| \to \infty$ as $|p| \to \infty$ for each $(t,x) \in \mathbb{T}^1 \times M$.
		\item Completeness: The Hamiltonian vector field $X_H(t,x,p) = (\partial_pH(t,x,p),-\partial_xH(t,x,p))$ and hence its flow $\phi_H^{t,s}$ is complete in the sense that the flow curves are defined for all times $t \in \mathbb{R}$.
	\end{itemize}
\end{defi}
 
Under these assumptions, one can associate to $H(t,x,p)$ a time-periodic Tonelli \textit{Lagrangian} $L(t,x,v): \mathbb{T}^1 \times TM \to \mathbb{R}$ given by
\begin{equation} \label{Lag}
	L(t,x,v) = \max_{p \in T_x^*M} \{p(v) - H(t,x,p)\}
\end{equation}
which symmetrically gives 
\begin{equation}\label{Ham} 
	H(t,x,p) = \max_{v \in T_xM} \{p(v) - L(t,x,v)\}
\end{equation}

\par The \textit{Euler-Lagrange flow} also named \textit{Lagrangian flow} $\phi_L^{s,t}$ is the conjugate to the Hamiltonian flow $\phi_H^{s,t}$ by the \textit{Legendre map}\footnote{The Tonelli assumptions imply that for all time $t \in \mathbb{R}$, the Legendre map $\mathcal{L}$ is a diffeomorphism between $\{t\} \times TM$ and $\{t\} \times T^*M$ (see \cite{fathi2008weak} for details).} $\mathcal{L}(t,x,v) = \big(t,x,\partial_vL(t,x,v)\big)$.  We adopt the notation $\phi_L^t$ for $\phi_L^{0,t}$.\\

If $0\leq s \leq t$ are two real times and $x$ and $y$ are two points of $M$, we define the following quantities
\begin{itemize}
	\item For all absolutely continuous curve $\gamma :[s,t] \to M$, the \textit{action} of $\gamma$ is
	\begin{equation}
		A_L(\gamma) = \int_s^t L(\tau,\gamma(\tau), \dot{\gamma}(\tau)) \; d\tau
	\end{equation}
	\item the \textit{potential} between $(s,x)$ and $(t,y)$ is
	\begin{equation} \label{Potential0}
		h_0^{s,t} (x,y) = \inf \left\{ A_L(\gamma) \; \left| \;
		\begin{matrix}
			\gamma : & [s,t] \to M \\
			& s \mapsto x \\
			& t \mapsto y
		\end{matrix} \right.	\right\}
	\end{equation}
	where the infimum is taken over such absolutely continuous curves $\gamma$.
\end{itemize}

\begin{rem}
	The demanded absolute continuity of the curves enables to define this integral. All the curves and all the infimum over the curves that may be considered in this paper will be in the family of absolutely continuous curves. We will refrain from recalling it every time.
\end{rem}

We now introduce the Lax-Oleinik operator that generates the viscosity solutions of the Hamilton-Jacobi equation
\begin{equation} \label{HJAlpha}
	\partial_tu + H(t,x,\partial_xu) = \alpha_0
\end{equation}

\begin{defi} \label{LODef}
	Fix two times $s<t$.
	\begin{enumerate}
		\item The \textit{Lax-Oleinik operator} $\mathcal{T}_0^{s,t}: \mathcal{C}(M,\mathbb{R}) \to \mathcal{C}(M,\mathbb{R})$ is defined as
		\begin{equation} \label{LO}
			\begin{split}
				\mathcal{T}_0^{s,t} u_0(x) & 
				= \inf_{
				\begin{matrix}
					\gamma : [s,t] \rightarrow M \\
					t \mapsto x 
				\end{matrix}
				} \left\{  u_0(\gamma(s)) + \int_s^t L(\tau,\gamma(\tau), \dot{\gamma}(\tau)) \;d\tau \right\} \\
				& = \inf_{y \in M} \left\{  u_0(y) + h_0^{s,t}(y,x) \right\}\\
			\end{split}
		\end{equation}
		\item The \textit{Mañé critical value} $\alpha_0$ is defined as 
		\begin{equation}\label{ManeCritValue}
			\alpha_0 = - \inf_\mu \left\{\int_{\mathbb{T}^1 \times TM}L \;d\mu \right\}
		\end{equation}
		where the infimum is taken over compact supported Borel probability measures $\mu$ invariant by the Euler-Lagrangian flow corresponding to $L$.\\
		\item The \textit{full Lax-Oleinik operator} $\mathcal{T}^{s,t}: \mathcal{C}(M,\mathbb{R}) \to \mathcal{C}(M,\mathbb{R})$ is defined as
		\begin{equation}
			\mathcal{T}^{s,t} u_0(x) = \mathcal{T}_0^{s,t} u_0(x) + \alpha_0.(t-s)
		\end{equation}
		We adopt the notations $\mathcal{T}^t$ for $\mathcal{T}^{0,t}$ and $\mathcal{T}$ for $\mathcal{T}^{0,1}$.
	\end{enumerate}
\end{defi}

One can verify that for all $s<t<\tau$, $\mathcal{T}^{t,\tau} \circ \mathcal{T}^{s,t} = \mathcal{T}^{s, \tau} $. Additionally, Since $L$ is time-periodic, $\mathcal{T}^{t+1} = \mathcal{T}^t \circ \mathcal{T}$. Hence $\{\mathcal{T}^n \}_{n\in \mathbb{N}}$ is a discrete semi-group acting on $\mathcal{C}(M,\mathbb{R})$, called the Lax-Oleinik semi-group. The main focus of this paper is the asymptotic behaviour of $\mathcal{T}$.\\

\begin{rem} \label{RemarqueIntro}
	\begin{enumerate}
		\item By abuse of language, we will often refer to the initial condition $u_0$ as a viscosity solution while referring in reality to the corresponding viscosity solution $u(t,x) = \mathcal{T}^{t}u_0(x)$.
		
		\item We will prove that in the case of the Mañé Lagrangians, heavily utilized in this article (see Subsection \ref{SectionConstruction}), the Mañé critical value is null. In this case, the two operators $\mathcal{T}$ and $\mathcal{T}_0$ do coincide.
	\end{enumerate}
\end{rem}

\begin{lem} \label{ViscosityInf}
	Let $v_n \in \mathcal{C}(\mathbb{R} \times M, \mathbb{R})$ be a sequence of viscosity solutions and let $u(t,x) = \inf_n \{v_n(t,x) \}$. Then $u$ is a viscosity solution.
\end{lem}

To a scalar map $u$ of $\mathcal{C}(M , \mathbb{R})$, we associate its \textit{$\omega$-limit set} $\omega(u):=\omega_{\mathcal{T}}(u)$ by the operator $\mathcal{T}$ as the set of limit points in the $C^0$-topology of the sequence $(\mathcal{T}^n(u))_{n \in \mathbb{N}}$. More precisely
\begin{equation}
	\omega(u) = \{ v \in \mathcal{C}(M , \mathbb{R}) \; | \; \exists (k_n)_n \in \mathbb{N}^\mathbb{N} \text{ increasing sequence s.t } \Vert \mathcal{T}^{k_n}u - v \Vert_\infty \to 0 \text{ as } n \to \infty\}
\end{equation}

\begin{defi}
	\begin{enumerate}
		\item If $\mathcal{T}u = u$, or equivalently if $u(t,\cdot) = \mathcal{T}^tu$ is a $1$-periodic viscosity solution, then $u$ is called a \textit{weak-KAM solution} of the Hamilton-Jacobi equation (\ref{HJAlpha}).
		\item The \textit{recurrent set} $\mathcal{R}(\mathcal{T})$ of $\mathcal{T}$ is the subset of recurrent elements of $\mathcal{C}(M,\mathbb{R})$ i.e. $u \in \mathcal{C}(M,\mathbb{R})$ such that $u$ belongs to $\omega(u)$.
		\item The \textit{non-wandering set} $\Omega( \mathcal{T})$ of $\mathcal{T}$ is the subset of elements $u$ of $\mathcal{C}(M,\mathbb{R}) $ such that for all neighbourhood $U$ of $u$ in the $C^0$ topology, there are infinitely many positive integers $k$ such that $\mathcal{T}^k(U) \cap U \neq \emptyset$.
	\end{enumerate}
\end{defi}

The following proposition justifies the non-emptiness of these sets
\begin{prop} \label{ViscBounded}
	For all $(s,u) \in \mathbb{R} \times \mathcal{C}(M, \mathbb{R})$, the family $\big( \mathcal{T}^{s,t}u = \mathcal{T}_0^{s,t}u + \alpha_0.(t-s) \big)_{t \geq s}$ is uniformly bounded in $\mathcal{C}(M, \mathbb{R})$.
\end{prop}

\subsection{Existence of Minimizing Curves}

We give classical results on minimizing curves that stem directly from the theory of variational calculus. We will avoid providing tedious proofs that can be found in the autonomous framework in \cite{fathi2008weak} and \cite{MR1720372} and in the non-autonomous framework in \Mane's lecture notes \cite{mane1990global}.

\begin{defi} \label{MinimizingDef}
	A \textit{minimizing curve} $\gamma : I \to M$ defined on an interval $I$ is a curve such that for all $s<t$ in $I$,
	\begin{equation} \label{MinimizingFormula}
		\int_s^t L(\tau,\gamma(\tau), \dot{\gamma}(\tau)) \; d\tau = h_0^{s,t}( \gamma(s), \gamma(t))
	\end{equation}
\end{defi}

\begin{prop}\label{MinimizingProp}
	\begin{enumerate}
		\item If for some times $s<t$, the curve $\gamma$ verifies (\ref{MinimizingFormula}), then it is minimizing on $[s,t]$.
		\item A minimizing curve $\gamma$ is as regular as the Lagrangian $L$ and it follows the Lagrangian flow $\phi_L$ i.e for all time $\tau \in [s,t]$, $(\gamma(\tau), \dot{\gamma}(\tau)) = \phi_L^{s,\tau}(\gamma(0), \dot{\gamma}(0))$. Moreover, it verifies the Euler-Lagrange equation
		\begin{equation}
			\partial_x L(\tau,\gamma(\tau), \dot{\gamma}(\tau)) = \frac{d}{d\tau} \big(\partial_v L(\tau,\gamma(\tau), \dot{\gamma}(\tau))\big)
		\end{equation}
	\end{enumerate}
\end{prop}

\begin{rem} \label{MinimizingRem}
	The first property is due to the fact that the quantity
	\begin{equation*}
		\int_{s'}^{t'} L(\tau,\gamma(\tau), \dot{\gamma}(\tau)) - h_0^{s',t'}( \gamma(s'), \gamma(t'))  \; d\tau \geq 0
	\end{equation*}
	is non-negative and is increasing with the inclusion of the interval $[s',t'] \subset [s,t]$.
\end{rem}

A main brick in the study of Tonelli Lagrangians is the existence of minimizing curves.

\begin{theo}(Tonelli's Theorem) \label{TonelliTheorem}
	Let $L : \mathbb{T}^1 \times TM \to \mathbb{R}$ be a Tonelli Lagrangian. Let $s<t$ be two real times and $x$ and $y$ be two points of $M$. Then, there exists a minimizing curve $\gamma : [s,t] \to M$ linking $x$ to $y$.
\end{theo}

A first application of this theorem to the Lax-Oleinik operator $\mathcal{T}$ gives the following 
\begin{cor} \label{TonelliLaxOleinik}
	Let $u$ be a scalar map in $\mathcal{C}(M,\mathbb{R})$. For all times $s<t$ and for all point $x$ of $M$, there exists a minimizing curve $\gamma : [s,t] \to M$ such that 
	\begin{equation}
		\gamma(t) = x \quad \text{and} \quad \mathcal{T}^{s,t}u(x) = u(\gamma(0)) + A_L(\gamma)
	\end{equation}
\end{cor}

Another fundamental theorem is the A priori compactness property of minimizing curves for Tonelli Lagrangians. It has been first proven by John N. Mather in \cite{MR1109661}.

\begin{theo}(A Priori Compactness) \label{APrioriCompactness}
	Let $L : \mathbb{T}^1 \times TM \to \mathbb{R}$ be a Tonelli Lagrangian and fix a small positive $\varepsilon >0$. Then, there exists a compact subset $K_\varepsilon$ of $TM$ such that every minimizing curve $\gamma : [s,t] \to M$ with $t-s \geq \varepsilon$ verifies $( \gamma(\tau) , \dot{\gamma}(\tau)) \in K_\varepsilon$.
\end{theo}

\begin{cor}
	For fixed times $s<t$, if $\gamma_n : [s,t] \to M$ is a sequence of minimizing curves, then it admits a subsequence that $C^1$-converges to a minimizing curve $\gamma:[s,t] \to M$.
\end{cor}

\subsection{Non-Expansiveness and Regularizing Properties of $\mathcal{T}$}

The regularizing properties of the Lax-Oleinik semi-group $\mathcal{T}$ has various implications on the asymptotic behaviour of viscosity solutions. These will be briefly explored in this paper, particularly in Section \ref{OmegaUSection} where we focus on the understanding of the limit behaviour of the constructed recurrent viscosity solution. A detailed investigation of these questions is done in \cite{Representation}.\\

\begin{prop} \label{Contracting}
	For all time $t>0$, The time $t$ Lax-Oleinik operators $\mathcal{T}^t$ is non-expanding, \textit{i.e}
	\begin{equation} \label{NonExpansivenessFormula}
		\forall u,v \in \mathcal{C}(M,\mathbb{R}), \quad \Vert\mathcal{T}^tu - \mathcal{T}^tv\Vert_\infty = \Vert \mathcal{T}^tu - \mathcal{T}^tv\Vert_\infty \leq  \Vert u-v\Vert_\infty
	\end{equation} 
	In particular, the Lax-Oleinik operators $\mathcal{T}$ is non-expanding.
\end{prop}

\begin{cor} \label{Minimality}
	\begin{enumerate}
		\item The non-wandering set $\Omega( \mathcal{T})$ is equal to the recurrent set $\mathcal{R}(\mathcal{T})$.
		\item Let $u \in \mathcal{C}(M,\mathbb{R})$. The restriction of $\mathcal{T}$ to the $\omega$-limit set of $u$ is minimal i.e. for all $v \in \omega(u)$, $\omega(u)= \overline{\{\mathcal{T}^nv \; | \; n \in \mathbb{N} \}}$.
	\end{enumerate}
\end{cor}

\begin{rem} \label{rect}
	\begin{enumerate}
		\item Due to the equality $\Omega( \mathcal{T}) = \mathcal{R}(\mathcal{T})$, we will often confuse non-wandering and recurrent viscosity solutions of the Hamilton-Jacobi equation.
		\item Some implications of this corollary on $\omega(u)$ include :
		\begin{enumerate}
			\item Every viscosity solution of $\omega(u)$ is recurrent.
			\item If $\omega(u)$ contains a periodic orbit, then it is equal to that orbit itself.
		\end{enumerate}
	\end{enumerate}
\end{rem}

We also present a more intricate results concerning the regularizing property of $\mathcal{T}$, rooted primarily in the A Priori Compactness Theorem \ref{APrioriCompactness}. Proofs can be found in \cite{MR1014729} or \cite{MR0235269}.

\begin{prop} \label{Regularity}
	For all positive $\varepsilon >0$, there exists a positive constant $\kappa_\varepsilon>0$ such that for all times $s < t$ with $t-s \geq \varepsilon$, we have
	\begin{enumerate}
		\item The potential $h_0^{s,t} : M \times M \to \mathbb{R}$ is $\kappa_\varepsilon$-Lipschitz. Moreover, we can take  $\kappa_\varepsilon$ so that the time dependent potential $h_0 : \{0 \leq s \leq t- \varepsilon\}\times M \times M \to \mathbb{R}$ is still $\kappa_\varepsilon$-Lipschitz. 
		\item For all initial data $u \in \mathcal{C}(M,\mathbb{R})$, the map $\mathcal{T}^{s,t}u : M \to \mathbb{R}$ is $\kappa_\varepsilon$-Lipschitz.
	\end{enumerate}
\end{prop}

Consequently, we get a regularity result on viscosity solutions.
 
\begin{cor} \label{Equicontinuity}
	For all viscosity solution $u: [s, + \infty) \times M \to \mathbb{R}$ of the Hamilton-Jacobi equation (\ref{HJAlpha}), the family $(u(t,\cdot))_{t \geq s+1}$ is equilipschitz.
\end{cor}

\subsection{Calibrated Curves}

Calibrated curves represent a type of minimizing curves that are well adapted to a given viscosity solution in the following sense.

\begin{defi} 
	Let $u(t,x)$ be a viscosity solution of (\ref{HJAlpha}). A curve $\gamma : I \subset \mathbb{R} \to M$ defined on a real interval $I$ is said \textit{calibrated by $u$} or \textit{$u$-calibrated} if for all times $s<t$ of $I$, we have
	\begin{equation} \label{CalibrationEquation}
		\begin{split}
			u(t, \gamma(t)) &= u(s, \gamma(s)) + \int_s^t \Big( L(\tau, \gamma(\tau), \dot{\gamma}(\tau)) + \alpha_0 \Big) \; d\tau \\
			&= u(s, \gamma(s)) + h_0^{s,t}(\gamma(s), \gamma(t)) + \alpha_0.(t-s)
		\end{split}
	\end{equation}	 
\end{defi}

\begin{rem}\label{CalibEL}
	\begin{enumerate}
		\item One observes from (\ref{CalibrationEquation}) that calibrated curves $\gamma$ realize the infimum in the definition (\ref{LO}) of the Lax-Oleinik operator. This means that all calibrated curves are minimizing and do follow the Lagrangian flow $\phi_L$.
		\item Same as for minimizing curves in Remark \ref{MinimizingRem}, if a curve $\gamma$ verifies (\ref{CalibrationEquation}) for some times $s<t$, then it verifies it for all $s \leq s'<t' \leq t$ and $\gamma$ is calibrated by $u$ on the interval $[s,t]$.
	\end{enumerate}
\end{rem}

\begin{prop} \label{CalibExist}
	Let $u(t,x) : \mathbb{R} \times M \to \mathbb{R}$ be a viscosity solution of (\ref{HJAlpha}). For all point $x$ of $M$ and for all real time $t \in \mathbb{R}$, $u$ admits a calibrated curve $\gamma_x : (- \infty , t] \to M$ with $\gamma(t) = x$.
\end{prop}

One of the powerful results of the weak-KAM theory is the theorem of regularity on calibrated curves proved by A.Fathi in the autonomous case (see \cite{fathi2008weak}).

\begin{theo} \label{CalibRegularity}
	Let $u$ be a viscosity solution of (\ref{HJAlpha}). If the curve $\gamma : I \to M$ is calibrated by $u$, then for all time $t$ in the interior of $I$, $u$ is differentiable at $(t, \gamma(t))$ with differential
	\begin{equation}
		\partial_tu(t,\gamma(t))= -H(t,\gamma(t),d_xu(t,\gamma(t))) \quad \text{and} \quad  d_xu(t,\gamma(t)) = \partial_vL \big(t,\gamma(t), \dot{\gamma}(t) \big)
	\end{equation}
\end{theo}

\subsection{Peierls Barrier}  \label{PeierlsSection}

The Peierls barrier introduced by Mather in \cite{MR1275203} enables the creation of particular (periodic) viscosity solutions that are central in the main constructions of this paper. In this section, we introduce a time-dependent Peierls Barrier following \cite{CIS} and compile its properties. Proofs can be found in \cite{Representation}.\\

For times $s<t$ we define the map $h^{s,t}: M \times M \to \mathbb{R}$ by
\begin{equation} \label{Potential}
	\begin{split}
		h^{s,t}(x,y)& = (t-s).\alpha_0 + h_0^{s,t}(x,y)\\  
		&= (t-s).\alpha_0 + \inf \left\{ A_L(\gamma) \; \left| \;
			\begin{matrix}
				\gamma : & [s,t] \to M \\
				& s \mapsto x \\
				& t \mapsto y
			\end{matrix} \right.	\right\}
	\end{split}
\end{equation}
And we adopt the notation $h^t$ for $h^{0,t}$.

\begin{prop} \label{hprop}
	\begin{enumerate}
		\item (Triangular Inequality) For all real time $s<\tau<t$, and for all points $x$, $y$ and $z$ in $M$, we have the triangular inequality
			\begin{equation}\label{TriangIneg}
				h^{s,t}(x,z) \leq h^{s,\tau}(x,y)+h^{\tau,t}(y,z)
			\end{equation}
	\end{enumerate}
\end{prop}

\begin{defi} \label{PeierlsDefi}
	The \textit{Peierls barrier} $h^\infty : M^2 \to \mathbb{R}$ is defined as 
	\begin{equation} \label{Peierls} 
		h^\infty(x,y) = \liminf_{n \to \infty} h^n(x,y)
	\end{equation}

	Its time dependence is defined as follows 
	\begin{equation}
		h^\infty(t,x,y) = h^{\infty+t}(x,y)= \liminf_{n \to \infty} h^{n+t}(x,y)
	\end{equation}

	More generally, for every two times $s$ and $t$, we define
	\begin{equation} \label{PeierlsGeneral}
		h^{s,\infty+t}(x,y) = \liminf_{n \to \infty} h^{s,n+t}(x,y)
	\end{equation}
	
	We define the \textit{Peierls set} $\mathcal{A}_0$ in $M$ as follows
	\begin{equation}
		\mathcal{A}_0 = \{ x \in M \; | \; h^\infty(x,x) = 0 \}
	\end{equation}
\end{defi}

\begin{prop} \label{PeierlsProp}
	\begin{enumerate}
		\item (Finiteness) For all $(t,x,y) \in \mathbb{R} \times M \times M$, the Peierls barrier $h^\infty(t,x,y)$ is finite.
		\item (Regularity) \label{PeierlsRegularity} The Peierls barrier $h^\infty$ is $\kappa_\varepsilon$-Lipschitz for all $\varepsilon >0$ with $\kappa_\varepsilon$ being the same Lipschitz value introduced in Proposition \ref{Regularity}.
		\item (Liminf Property) For all points $x$ and $y$ in $M$ and all sequences of points $(x_n)_n$ and $(y_n)_n$ respectively converging to $x$ and $y$, we have
		\begin{equation} \label{PeierlsLiminf}
			h^\infty(x,y) = \liminf_n h^{n}(x_n,y_n)
		\end{equation}
		\item (Triangular Inequality) For all point $x$, $y$ and $z$ in $M$, we have the triangular inequality
		\begin{equation}\label{TriangInegPeierls}
			h^\infty(x,z) \leq h^\infty(x,y) + h^\infty(y,z)
		\end{equation}
		\item And for all times $s< \tau < t$, we have the triangular inequalities
		\begin{equation}\label{TriangInegPeierls2}
			\begin{split}
				h^{s,\infty+t}(x,z) &\leq h^{s,\infty+\tau}(x,y) + h^{\tau,t}(y,z)\\
				h^{s,\infty+t}(x,z) &\leq h^{s,\tau}(x,y) + h^{\tau,\infty + t}(y,z)
			\end{split}
		\end{equation}
		\item (Weak-KAM Solution) \label{pvisc} For all $x \in M$, $h^\infty(t,x,y) - \alpha_0t$ is a weak-KAM solution of the Hamilton-Jacobi equation (\ref{HJAlpha}).
		\item (Non-negativity) \label{PeierlsPositivity} For every point $x$ in $M$, we have $h^\infty(x,x) \geq 0$.
	\end{enumerate}
\end{prop}

Following \cite{MR1792479}, we define another family of periodic barriers. 

\begin{defi} \label{nPeierlsDefi}
	For all integer $n \geq 1$ and for $k \in \{0,..,n-1\}$, we define 
	\begin{itemize}
		\item the \textit{$n$-Peierls Barrier} by
			\begin{equation}
				h^{n \infty}(x,y) = \liminf_{p \to \infty} h^{np}(x,y)
			\end{equation}
		\item the \textit{$(n,k)$-Peierls Barrier} by
			\begin{equation}
				h^{n \infty + k}(x,y) = \liminf_{p \to \infty} h^{np+k}(x,y)
			\end{equation}
	\end{itemize}
\end{defi}

The properties of Proposition \ref{PeierlsProp} do adapt to these more general periodic barriers. In addition, we get the following.
\begin{prop} \label{kvisc}
	For all integers $n \geq 1$ and $k \in \{0,..,n-1\}$,
	\begin{enumerate}
		\item $\mathcal{T} h^{n \infty + k} = h^{n \infty + k+1}$
		\item For all $x \in M$, $h^{n \infty + k}(x,\cdot)$ is a $n$-periodic viscosity solution of the Hamilton-Jacobi equation \eqref{HJAlpha}.
	\end{enumerate}
\end{prop}

\begin{rem}
	\begin{enumerate}
		\item These viscosity solutions can be seen as weak-KAM solutions for the modified Hamiltonian $nH(nt, x , p)$.
		\item If $\alpha_0 =0$, then these are $n$-time-periodic viscosity solutions of the Hamilton-Jacobi equation. 
	\end{enumerate}
\end{rem}

\subsection{The Mather Set} \label{MatherSection}

The Mather set, introduced by John N. Mather, is associated with the study of action-minimizing measures in Lagrangian and Hamiltonian systems.

\begin{defi}
	\begin{enumerate}
		\item A measure $\mu$ on $\mathbb{T}^1 \times TM$ is \textit{a minimizing measure} if it is a Borel probability measure, invariant by the Euler-Lagrange flow $\phi_L$ and it satisfies
		\begin{equation}
			\int_{\mathbb{T}^1 \times TM} L \; d\mu = - \alpha_0
		\end{equation} 
		where $\alpha_0$ is the Mañé critical value defined in (\ref{ManeCritValue}) (and assumed to be null).
		\item The \textit{Mather set} $\tilde{\mathcal{M}}$ is defined by
		\begin{equation}
			\tilde{\mathcal{M}} = \bigcup_\mu \supp(\mu) \subset \mathbb{T}^1 \times TM
		\end{equation}
		where the union is on minimizing measures $\mu$.
		\item The \textit{projected Mather set} $\mathcal{M}$ is the projection of $\tilde{\mathcal{M}}$ to $\mathbb{T}^1 \times M$.
		\item The \textit{time-zero Mather set} $\tilde{\mathcal{M}}_0$ and its projected counterpart $\mathcal{M}_0$ are the intersections
		\begin{equation}
			\tilde{\mathcal{M}}_0 := \tilde{\mathcal{M}} \cap \big( \{0\} \times TM \big) \quad \text{and} \quad \mathcal{M}_0 := \mathcal{M} \cap \big( \{0\} \times M \big)
		\end{equation}
		seen respectively as subsets of $TM$ and $M$.
	\end{enumerate}
\end{defi}

\begin{rem} \label{MatherInv}
	\begin{enumerate}
		\item More explicitly, the invariant measures $\mu$ featured in the definition are invariant by the maps $\Phi_L^\tau$, for all time $\tau>0$, given by
		\begin{equation} \label{InvarianceFlow}
			\begin{split}
				\Phi_L^\tau : \mathbb{T}^1 \times TM &\longrightarrow \mathbb{T}^1 \times TM \\
				(t,x,v) & \longmapsto (t + \tau, \phi_L^{t,t+\tau}(x,v))
			\end{split}
		\end{equation}   
		\item We easily see from these definitions that the different Mather sets are invariant by either the Euler-Lagrange flow or its time-one map. 
	\end{enumerate}
\end{rem}

The non-emptiness of this set has been proved by John N.Mather in the non-autonomous case. See proposition 4 of \cite{MR1109661}
\begin{prop} \label{MatherNonempty}
	The Mather set $\tilde{\mathcal{M}}$ is compact and non-empty.
\end{prop}

The following proposition has been proved by R.\Mane in \cite{MR1479499} and generalized by J-M.Roquejoffre and P.Bernard in \cite{MR2041603} to non-wandering viscosity solutions. However, for the purpose of this paper, we only require the result for periodic viscosity solutions.

\begin{prop} \label{CalibMather}
	Let $\tilde{x}$ be an element of the Mather set $\tilde{\mathcal{M}}_0$ and $x = \pi(\tilde{x})$ be its projection in $\mathcal{M}_0$. Let $\gamma : \mathbb{R} \to M$ be the projection on $M$ of the Lagrangian flow at $\tilde{x}$ i.e. $\gamma(t) = \pi \circ \phi_L^t(\tilde{x})$. Then,
	\begin{enumerate}
		\item Every periodic viscosity solution $u$ is calibrated by $\gamma$ and differentiable at $(t,\gamma(t))$.
		\item In particular, every weak-KAM solution $u$ is calibrated by $\gamma$.
	\end{enumerate}
\end{prop}

\begin{rem} \label{MatherRegularity}
	The application of Theorem \ref{CalibRegularity} to the curves $\gamma$ of $\mathcal{M}$ reveals that the Mather set (and more generally, the Peierls set $\mathcal{A}_0$ defined below) is a set of differentiability for all weak-KAM or periodic viscosity solutions. This fact can be extended to all non-wandering viscosity solutions of the Hamilton-Jacobi equation (see \cite{MR2041603}), but we will not need such generality in this paper.
\end{rem}

\begin{prop} \label{MatherPeierls}
	The following inclusion holds
	\begin{equation}
		\mathcal{M}_0 \subset \mathcal{A}_0
	\end{equation}
\end{prop}

\begin{proof} 
	Let $\mu$ be a minimizing measure on $\mathbb{T}^1 \times TM$. It is invariant by the map $\Phi_L^1$ defined in \eqref{InvarianceFlow}. Then the Poincaré recurrence theorem shows that $\Phi_L^1$-recurrent points are dense in $\supp(\mu)$. If $(t,x,v)$ is such a point, then so is $(\Phi_L^t)^{-1}(t,x,v) = (0, (\phi^{t}_L)^{-1}(x,v))$. Hence, $\Phi_L^1$-recurrent points are dense in $\supp_0(\mu) := \supp(\mu) \cap \{t = 0 \}$. Let $(x,v)$ be a $\Phi_L^1$-recurrent element of $\supp_0(\mu)$ where we omit the time variable $t=0$. Note that with this omission, the map $\Phi_L^1$ coincides with $\phi_L^1$ on $\supp_0(\mu)$.
	Consider an increasing sequence of positive integers $k_n$ such that $\phi_L^{k_n}(x,v)$ converges to $(x,v)$. Let $u$ be any weak-KAM solution, for instance $h^\infty(x, .)$ (see Property \ref{pvisc} of Proposition \ref{PeierlsProp}). From Proposition \ref{CalibMather}, we know that the curve $x(t) = \pi \circ \phi_L^t(x,v)$ is calibrated by $u$ and  
	\begin{equation*}
		u(x(k_n)) - u(x) = u(k_n,x(k_n)) - u(x) = k_n.\alpha_0 + \int_0^{k_n} L(s,x(s),\dot{x}(s)) \; ds \longrightarrow 0 \quad \text{as } n \to \infty
	\end{equation*}
	where the limit holds due to the continuity of $u$. Thus, for any small $\varepsilon >0$ and any large integer time $n_0>0$, closing the orbit $x(t)$ enables to construct a loop $\gamma : [0,k_n] \to M$ with $k_n > n_0$, $\gamma(0) = \gamma(k_n) = x$ and such that
	\begin{equation*}
		k_n.\alpha_0 + \int_0^{k_n} L(s,\gamma(s), \dot{\gamma}(s)) \; ds \leq \varepsilon
	\end{equation*}
	Therefore, by the definition of the Peierls barrier and taking the liming on $k_n$, we get $h^\infty(x,x) \leq 0$. However, we already have from the non-negativity Property \ref{PeierlsPositivity} of Proposition \ref{PeierlsProp} the inverse inequality  $h^\infty(x,x) \geq 0$. We conclude that $h^\infty(x,x)=0$.
\end{proof}

\subsection{The Mañé Set} \label{ManeSection}

In this subsection, we introduce the final concept derived from the Aubry-Mather theory that holds significance for this article: the Mané set. This set will be used in the proof of the regularity of the recurrent viscosity solution of Section \ref{RegSection}.\\

The definition of the Mañé Set requires the introduction of a potential.

\begin{defi}
	\begin{enumerate}
		\item The \textit{Mañé potential} $m : M \times M \to \mathbb{R}$ is defined by
		\begin{equation}
			m(x,y) = \inf_{n \geq 0} h^n(x,y)
		\end{equation}
		with a time evolution counterpart $m: \{(s,t) \in \mathbb{T}^1 \times \mathbb{T}^1 \; | \; s\leq t \} \times M \times M \to \mathbb{R}$ defined by
		\begin{equation}
			m^{s,t}(x,y) = \inf_{n\geq 0} h^{s,t+n}(x,y)
		\end{equation}
		We use the notation $m^t$ for $m^{0,t}$. 
		
		\item A curve $\gamma : I \to \mathbb{R}$ is said \textit{semi-static} if for all times $s<t$ in $I$, we have
		\begin{equation} \label{SemiStaticFormula}
			m^{s,t}(\gamma(s), \gamma(t)) = \alpha_0.(t-s) + \int_s^t L(\tau, \gamma(\tau), \dot{\gamma}(\tau)) \; d\tau  
		\end{equation}		 
		
		\item The \textit{Mañé set} $\tilde{\mathcal{N}}$ is the subset of $\mathbb{T}^1 \times TM$ defined as
		\begin{equation}
			\tilde{\mathcal{N}} = \{ (t,\gamma(t), \dot{\gamma}(t)) \; | \; \gamma : \mathbb{R} \to M \text{ is a semi-static curve } \}
		\end{equation}	
	\end{enumerate}
\end{defi}

\begin{rem} \label{SemiStaticRemark}
	\begin{enumerate}
		\item The first term of the infimum is the quantity $h^0$ defined as
		\begin{equation}
			h^0(x,y) = 
			\begin{cases}
				0 & \text{if } x=y \\
				+ \infty & \text{if }  x \neq y
			\end{cases}
		\end{equation}
		It follows that for any point $x$ of $M$, $m(x,x) =0$. This also implies that the map $m$ can exhibit discontinuity at the diagonal of $M \times M$. 
		\item Away from the diagonal of $M \times M$, the Mañé potential $m^{s,t}$ is continuous on its time and space variables.
		\item Note that a semi-static curve is minimizing as for any times $s<t$,
		\begin{align*}
			A_{L+\alpha_0}(\gamma)  = m^{s,t}(\gamma(s), \gamma(t)) \leq h^{s,t}(\gamma(s), \gamma(t)) \leq A_{L+\alpha_0}(\gamma)
		\end{align*}
		and we get equality everywhere.
		\item Analogously to Remarks \ref{MinimizingRem} and \ref{CalibEL}, if for fixed times $s<t$ the identity (\ref{SemiStaticFormula}) holds, then the curve $\gamma$ is semi-static on $[s,t]$.
	\end{enumerate}
\end{rem}

The regularity Proposition \ref{Regularity} for the finite-time potential $h^t$ is reflected on the regularity of the Mañé potential $m^{s,t}$ as stated below.

\begin{prop} \label{RegularityManePotential}
	For all positive real number $\varepsilon >0$ and all times $s< t+ \varepsilon$, the Mañé potential $m^{s,t}$ is $\kappa_\varepsilon$-Lipschitz with $\kappa_\varepsilon$ being the same Lipschitz value figuring in Proposition \ref{Regularity}.
\end{prop}
	
Below is another characterization of semi-static curves. Only an implication is stated here, but the inverse statement will be shown in Proposition \ref{CalibMane}.

\begin{prop} \label{SemiStaticCalib}
	Let $u$ be a weak-KAM solution of (\ref{HJAlpha}) and let $\gamma : I \to M$ be a calibrated curve by $u$. Then the curve $\gamma$ is semi-static.
\end{prop}

\begin{proof}
	Let $s<t$ be two times in $I$. We have by definition of weak-KAM solutions that for all integer $k$ such that $s \leq t+k$
	\begin{align*}
		u(t,\gamma(t)) & = u(t+k, \gamma(t)) + k.\alpha_0 = \mathcal{T}^{s,t+k} u(s, \gamma(t)) + k.\alpha_0\\
		&\leq u(s, \gamma(s)) + h_0^{s,t+k}(\gamma(s),\gamma(t)) + k.\alpha_0 \\
		&= u(s, \gamma(s)) + h^{s,t+k}(\gamma(s),\gamma(t)) - \alpha_0.(t-s)
	\end{align*}
	so that
	\begin{align*}
		A_{L+\alpha_0}(\gamma_{|[s,t]}) =  A_L(\gamma_{|[s,t]})+  \alpha_0.(t-s) = u(t,\gamma(t)) - u(s, \gamma(t)) + \alpha_0.(t-s) \leq h^{s,t+k}(\gamma(s),\gamma(t))
	\end{align*}
	And taking the infimum over the integers $k$, we get that $A_{L+\alpha_0}(\gamma_{|[s,t]}) = m^{s,t}(\gamma(s),\gamma(t))$ meaning that $\gamma$ is a semi-static curve.
\end{proof}

\begin{prop} \label{ManeMatherInclusion}
	The following inclusion holds
	\begin{equation}
		\tilde{\mathcal{M}} \subset \tilde{\mathcal{N}}
	\end{equation}
\end{prop}

\begin{proof}
	Let $(s,\tilde{x})$ be an element of $\tilde{\mathcal{M}}$ and $\gamma : \mathbb{R} \to M$ be the curve of $M$ defines as $\gamma(t) =\pi\circ \phi_L^{s,t}(\tilde{x})$. We need to show that $\gamma$ is semi-static.
	
	Fix a weak-KAM solution $u$. We know from Proposition \ref{CalibMather} that $\gamma$ is calibrated by $u$. Thus, for all integers $q<p$ and $k \geq 1$, and for all curves $\sigma : [q, q+k] \to M$ linking $\gamma(q)$ to $\gamma(p)$, we have
	\begin{align*}
		A_{L+\alpha_0}(\gamma_{|[q,p]}) &= A_L(\gamma_{|[q,p]}) + \alpha_0.(p-q) \\
		&= u(p, \gamma(p)) - u(q, \gamma(q)) + \alpha_0.(p-q) \\
		&= u(q+k, \gamma(p)) + \alpha_0.(q+k-p) - u(q, \gamma(q)) + \alpha_0.(p-q) \\
		&= u(q+k, \gamma(p)) - u(q, \gamma(q)) + \alpha_0.k \\
		&\leq A_L(\sigma) + \alpha_0.k = A_{L+\alpha_0}(\sigma)
	\end{align*}
	where we used the definition of weak-KAM solutions in the second line. Taking the infimum over such curves $\sigma$ while varying $k$, this implies that
	\begin{align*}
		A_L(\gamma_{|[q,p]}) + \alpha_0.(p-q)\leq m(\gamma(q), \gamma(p)) \leq A_L(\gamma_{|[q,p]}) + \alpha_0.(p-q)
	\end{align*}
	We deduce equality everywhere. And since the integers $p$ and $q$ are arbitrary, we conclude from the last point of Remark \ref{SemiStaticRemark} that the curve $\gamma$ is semi-static and $(s,\gamma(s), \dot{\gamma}(s)) = (s, \tilde{x})$ belongs to the Mañé set $\tilde{\mathcal{N}}$.
\end{proof}

We devote the rest of this subsection to understanding the dynamical behaviour of the semi-static curves in the Mañé set $\tilde{\mathcal{N}}$. To achieve this, we need to introduce additional concepts following \cite{MR1479499} and \cite{MR1720372}.

\begin{defi} \label{DefinitionStaticClasses}
	\begin{enumerate} 
		\item We define the semi-distance $d_1: \mathcal{M}_0 \times \mathcal{M}_0 \to \mathbb{R}_{\geq 0}$ as
		\begin{equation}
		d_1(x,y) = h^{\infty}(x, y) + h^{\infty}(y, x)
		\end{equation}
		And set $\sim$ to be the equivalence relation on $\mathcal{M}_0$ given by
		\begin{equation}
		x \sim y \Longleftrightarrow d_1(x,y) = 0
		\end{equation}
		The \textit{static classes} are the equivalence classes of the equivalence relation $\sim$. We denote by $\mathbb{M}$ the set of static classes.\\
		
		\item For a curve $\gamma : \mathbb{R} \to M$, we define its integer-time $\alpha$-limit set $\alpha_k(\gamma)$ and its integer-time $\omega$-limit set $\omega_k(\gamma)$ by 
		\begin{equation} \label{AlphaOmega1}
			 	\begin{split}
					\alpha_k(\gamma) &= \{ z \in M \; | \; \exists (q_n)_n \in \mathbb{N}^\mathbb{N}; \; \lim_n q_n = +\infty, \; \lim_n \gamma(-k.q _n) = z \} \\
					\omega_k(\gamma) &= \{ z \in M \; | \; \exists (q_n)_n \in \mathbb{N}^\mathbb{N}; \; \lim_n q_n = +\infty, \; \lim_n \gamma(k.q_n) = z \}
				\end{split}
		\end{equation}
		
		 Then, we set a partial order $\preceq$ on $\mathbb{M}$ given by
		\begin{enumerate}[label=\roman*.]
			\item $\preceq$ is reflexive and transitive.
			\item For all $x$ and $y$ in $\mathbb{M}$, if there exists a curve $\gamma$ with lift $(\gamma, \dot{\gamma})$ in the Mañé set $\tilde{\mathcal{N}}$ such that its integer $\alpha$-limit set $\alpha_1(\gamma) \subset x$ and its integer $\omega$-limit set $\omega_1(\gamma) \subset y$, then $x \preceq y$.
		\end{enumerate}
	\end{enumerate}	
\end{defi}

\begin{rem} \label{PositivityPseudoDistance}
	Note that the semi-distance $d_1$ is non-negative thanks to the triangular inequality property (\ref{TriangInegPeierls}) and the non-negativity Property \ref{PeierlsPositivity} of Proposition \ref{PeierlsProp}. Indeed, for all $x$ and $y$ in $\mathcal{M}_0$,
	\begin{align*}
		d_1(x,y) = h^{\infty}(x, y) + h^{\infty}(y, x) \geq h^{\infty}(x, x) \geq 0
	\end{align*}
	
	It is a distance when restricted to the set $\mathbb{M}$.
\end{rem}

\begin{prop} \label{StaticClassesProperties}
	If $x$ and $y$ are two point of the Peierls set $\mathcal{A}_0$ with the same static class in $\mathbb{M}$, then for all point $z \in M$, 
	\begin{equation}
		h^\infty(x,z) = h^\infty(x,y) + h^\infty(y,z)
	\end{equation}
\end{prop}
\begin{proof}
	This is no more than an application of the triangular inequality (\ref{TriangInegPeierls}) as follows
	\begin{align*}
		h^\infty(x,z) &\leq h^\infty(x,y) + h^\infty(y,z) \\
		& \leq h^\infty(x,y) + h^\infty(y,x) + h^\infty(x,z)\\
		& = d_1(x,y) + h^\infty(x,z) = h^\infty(x,z)
	\end{align*}
	So, there is equality everywhere.
\end{proof}

The following proposition motivates the definition of the partial order $\preceq$.

\begin{prop} \label{SemiStaticLimitInclusions}
	Let $\gamma : I \to M$ be a semi-static curve. Then, whenever the $\alpha$ and $\omega$-limit sets are well-defined for the interval $I$, there exist static classes $\bar{x}$ and $\bar{y}$ in $\mathbb{M}$ such that 
	\begin{equation}
		\alpha_1(\gamma) \subset \bar{x} \subset \mathcal{A}_0 \quad \text{and} \quad \omega_1(\gamma) \subset \bar{y} \subset \mathcal{A}_0 
	\end{equation}
\end{prop}

\begin{proof}
	we only prove the first inclusion. Assume that $I$ is unbounded in the negative direction. Let $z$ be an element of $\alpha_1(\gamma)$ and let $q_n$ be an increasing sequence of positive integers such that $\lim\limits_n q_{n+1}-q_n = +\infty$ and $\gamma(-q_n)$ converges to $z$. We first prove that $z$ belongs to $\mathcal{A}_0$. Fix a small time $s>0$. By the A priori compactness Theorem \ref{APrioriCompactness}, we know that the curves $\gamma_n(t) = \gamma(-q_n +t)$ are bounded in the $C^1$-topology. Thus, we can assume, up to extraction, that they $C^1$-converge on the interval $[0,s]$ to a curve $\sigma_z : [0,s] \to M$ with $\sigma_z(0) = z$. Let us evaluate $h^{\infty+t}(z, \sigma_z(t))$. The regularity Proposition \ref{Regularity} shows that for all integers $n \geq 0$
	\begin{equation*}
		|h^{q_{n+1}-q_n+t}(z,\sigma_z(t)) - h^{q_{n+1}-q_n+t}(\gamma(t-q_{n+1}), \gamma(-q_{n}))|  \leq \kappa_1. [d(z, \gamma(t-q_{n+1})) + d(\sigma_z(t), \gamma(-q_{n}))]
	\end{equation*}
	which implies that
	\begin{align*}
		\liminf_n h^{q_{n+1}-q_n+t}(z,\sigma_z(t)) = \liminf_n h^{q_{n+1}-q_n+t}(\gamma(-q_{n+1}), \gamma(t-q_{n}))
	\end{align*}
	Using the semi-static behaviour of the curve $\gamma$, we get the following
	\begin{align*}
		h^{\infty+t} (z,\sigma_z(t)) & \leq \liminf_n h^{q_{n+1}-q_n+t}(z,\sigma_z(t)) = \liminf_n h^{q_{n+1}-q_n+t}(\gamma(-q_{n+1}), \gamma(t-q_{n})) \\
		&= \liminf_n A_L( \gamma_{|[-q_{n+1},t-q_n]}) = \liminf_n m^t(\gamma(-q_{n+1}), \gamma(t-q_{n})) = m^t(z,\sigma_z(t)) 
	\end{align*}
	where the last equality is due to the continuity of the Mañé potential $m^t$ stated in Proposition \ref{RegularityManePotential}. And since $m^t(z,\sigma_z(t)) \leq h^{\infty+t} (z,\sigma_z(t))$, we deduce the equality $m^t(z,\sigma_z(t)) = h^{\infty+t} (z,\sigma_z(t))$ for all times $t \in (0,s]$. Taking the time $t$ to zero, we get
	\begin{equation} 
		h^{\infty} (z,z) = \lim_{t \to 0} h^{\infty + t} (z, \sigma_z(t) ) = \lim_{t \to 0} m^t(z,\sigma_z(t)) \leq  \lim_{t \to 0} A_L(\sigma_{z|[0,t]}) =0
	\end{equation}
	where the first limit is due to the uniform continuity of $h^\infty$ on the graph of $(t,\sigma_z(t))$ for $t \in [0,s]$. Hence, we deduce from the non-negativity Property \ref{PeierlsPositivity} of Proposition \ref{PeierlsProp} that $h^{\infty} (z,z)=0$ meaning that $z$ belongs to the Peierls set $\mathcal{A}_0$.\\
	
	Now we show that the $\alpha$-limit set $\alpha_1(\gamma)$ belongs to a unique static class. Let $x$ and $z$ be two points of $\alpha_1(\gamma)$. Let $q_n$ and $p_n$ be two intertwined increasing sequences of integers such that $\gamma(-q_n)$ converges to $x$, $\gamma(-p_n)$ converges to $z$, for all $n \geq 0$, $q_n < p_n < q_{n+1}$ and $\lim\limits_n q_{n+1}-p_n = \lim\limits_n p_n- q_n = +\infty$. As done above, we construct two curves $\sigma_z$ and $\sigma_x : [0,s] \to M$ and we prove that
	\begin{align*}
		h^{\infty + t}(x, \sigma_z(t)) = m^t(x, \sigma_z(t)) \quad \text{and} \quad  h^{t,\infty+2t}(\sigma_z(t), \sigma_x(2t)) = m^{t,2t}(\sigma_z(t), \sigma_x(2t))
	\end{align*}		
	Thus, we get
	\begin{align*}
		h^{\infty + t}(x, \sigma_z(t)) + h^{t,\infty+2t}(\sigma_z(t), \sigma_x(2t)) &= m^t(x, \sigma_z(t)) + m^{t,2t}(\sigma_z(t), \sigma_x(2t)) \\
		&= \lim_n m^t\big(\gamma(-q_{n+1}), \gamma(t-p_n)\big) + m^{t,2t}\big(\gamma(t-p_n), \gamma(2t-q_n)\big) \\
		&= \lim_n A_{L+\alpha_0}(\gamma_{|[-q_{n+1}, t-p_n]}) + \lim_n A_{L+\alpha_0}(\gamma_{|[t-p_n, 2t -q_n]}) \\
		&= \lim_n A_{L+\alpha_0}(\gamma_{|[-q_{n+1}, 2t-q_n]}) =\lim_n m^{2t}\big(\gamma(-q_{n+1}), \gamma(2t-q_n)\big) \\
		&= m^{2t}(x, \sigma_x(2t))
	\end{align*}
	And taking $t$ to $0$, we finally get
	\begin{align*}
		d_1(x,z) = h^\infty(x,z) + h^\infty(z,x) & = \lim_{t \to 0} h^{\infty + t}(x, \sigma_z(t)) + h^{t,\infty+2t}(\sigma_z(t), \sigma_x(2t)) \\
		&=\lim_{t \to 0} m^{2t}(x, \sigma_x(2t)) \leq \lim_{t \to 0} A_{L+\alpha_0}(\sigma_{x|[0,2t]}) =0
	\end{align*}
	However, we already know from Remark \ref{PositivityPseudoDistance} that $d_1(x,z) \geq 0$. We obtain $d_1(x,z) =0$ and the two points $x$ and $z$ belong to the same static class $\bar{x}$.
\end{proof}

\begin{prop} \label{CalibMane}
	Let $\gamma : \mathbb{R} \to M$ be a semi-static curve on $M$ and let $\bar{x}$ be the static class of $\alpha_1(\gamma)$. Then for all point $x$ of $\bar{x}$, the curve $\gamma$ is calibrated by the weak-KAM solution $h^\infty(x, \cdot)$.
\end{prop}

\begin{proof} 
	Fix two real times $s<t$. Let us first assume that $x$ belongs to $\alpha_1(\gamma)$. Then there exists an increasing sequence of integers $q_n$ such that $\gamma(-q_n)$ converges to $x$. We have from the definition of semi-static curves that
	\begin{align*}
		m^t(\gamma(-q_n), \gamma(t)) &= A_{L+\alpha_0}(\gamma_{|[-q_n,t]})= A_{L+\alpha_0}(\gamma_{|[-q_n,s]}) + A_{L+\alpha_0}(\gamma_{|[s,t]})  \\
		&= m^s(\gamma(-q_n), \gamma(s)) + A_{L+\alpha_0}(\gamma_{|[s,t]})
	\end{align*}
	In order to take the limit, we need to make sure that we work in the domain of continuity of the Mañé potential $m^{s,t}$. It suffices to assume that $s$ and $t \neq 0$ in $\mathbb{T}^1 = \mathbb{R} / \mathbb{Z}$. Then we can take the limit on $n$ to get
	\begin{align*}
		m^t(x,\gamma(t)) = m^s(x, \gamma(s)) + A_{L+\alpha_0}(\gamma_{|[s,t]})
	\end{align*}

	We need to replace $m^t$ by $h^{\infty+t}$. To do so, remark that for all integer $n \geq 0$, we have from the triangular inequality (\ref{TriangInegPeierls2}) of the Peierls Barrier that
	\begin{align*}
		h^{\infty+t}(x,\gamma(t)) = h^{\infty+t+n}(x,\gamma(t)) \leq h^\infty(x,x) + h^{t+n}(x,\gamma(t))=h^{t+n}(x,\gamma(t))
	\end{align*}
	where we used the inclusion $x \in \mathcal{A}_0$ derived from Proposition \ref{SemiStaticLimitInclusions}. And taking the infimum over $n$, we get the inequality $h^{\infty+t}(x,\gamma(t)) \leq m^t(x, \gamma(t))$. And since the inverse inequality is immediate, we deduce the equality $h^{\infty+t}(x,\gamma(t)) = m^t(x, \gamma(t))$. Doing the same for $s$, we finally get
	\begin{align*}
		h^{\infty+t}(x,\gamma(t)) = h^{\infty+s}(x, \gamma(s)) + A_{L+\alpha_0}(\gamma_{|[s,t]})
	\end{align*}
	Now if for example $s=0$ in $\mathbb{T}^1 = \mathbb{R} / \mathbb{Z}$ and $\gamma(s) = x$, we take $s'>s$ and we get $h^{\infty+s'}(x,\gamma(t)) = m^{s'}(x, \gamma(s'))$. The regularity of the Peierls barrier $h^\infty$ allows to take the limit $s' \to s$ and conclude.
	
	The extension of the result to the static class $\bar{x}$ is an immediate application of Proposition \ref{StaticClassesProperties}.
\end{proof}

\begin{rem}
	Gathering Propositions \ref{SemiStaticCalib} and \ref{CalibMane}, we infer the following characterization of semi-static curves : a curve is semi-static if and only if it is calibrated by a weak-KAM solution.
\end{rem}

The following theorem is an adaptation in the non-autonomous case of a result proved in Section 3.11 of \cite{MR1720372}. It can also be derived from \cite{MR2393423} which covers the time-dependant framework.

\begin{theo} \label{ManeChainTransitivity}
	If the set of static classes $\mathbb{M}$ is finite, then for all $x$ and $y$ in $\mathbb{M}$ we have $x \preceq y$.
\end{theo}

\begin{rem}
	The general result true even in the case where $\mathbb{M}$ is infinite is the chain-transitivity of the Mañé set $\tilde{\mathcal{N}}$. However, we will need in Section \ref{RegSection} the specific behaviour of the semi-static curves between the static classes.	
\end{rem}

\section{Construction of the \Mane Hamiltonian} \label{SectionManeHamiltonian}

In this section, we construct a Tonelli Hamiltonian $H : \mathbb{T}^1 \times T^*M \to \mathbb{R}$ whose Hamilton-Jacobi equation admits a smooth, recurrent, non-periodic viscosity solution. This construction is inspired by Mañé's approach in the appendix of \cite{MR1166538}.\\

The idea of the construction involves creating regions in $M$ where a selected viscosity solution is periodic with a prescribed minimal period. Taking these periods to infinity will make the solution non-periodic, and controlling the size of the regions facilitates the proof of the recurrence.\\

We will also highlight certain symmetries imposed on the Hamiltonian vector field, which are crucial for ensuring the existence of regular, recurrent, non-periodic viscosity solutions, as will be shown in Section \ref{RegSection}.\\

In the first subsection, we present the construction of the Hamiltonian, and in the second, we analyze the behavior of the $n$-Peierls barriers $h^{n\infty}$ (see Definition \ref{nPeierlsDefi}), which will be essential for building periodic viscosity solutions with prescribed minimal periods.\\

\subsection{Construction of the Lagrangian $L$} \label{SectionConstruction}

We construct the Lagrangian $L$ following the examples constructed by Mañé in the appendix of \cite{MR1166538}. We will designate these Lagrangians as Mañé Lagrangians.\\

The idea is to select the desired dynamics of $f_t = \phi^t_L$ on a specific submanifold of $TM$ and subsequently extending it into a Lagrangian flow across the entire tangent bundle. From the Hamiltonian perspective, the initial submanifold corresponds to the zero-section of $T^*M$ which yields the desired Hamiltonian dynamics.\\

\subsubsection{Subdivision of the Manifold $M$} \label{SectionSubdivision}

We first need to fix a region of work in the manifold $M$ where we will define the isotopy $f_t$ as desired. We work on a fixed chart $U \to \mathbb{R}^d$ of $M$. This chart will be identified to an open $0$-centered ball $B$ of $\mathbb{R}^d$. We use the coordinates $(r,\theta,x_3,...,x_d)$ in $\mathbb{R}^d$ where $(r, \theta)\in \mathbb{R}_{\geq 0} \times \mathbb{T}^1$ are the polar coordinates on the plane $\mathbb{R}^2 \times \{0\}$ and $(x_3,...,x_d)$ are the remaining canonical coordinates of $\{0\} \times \mathbb{R}^{d-2}$. The space $\mathbb{R}^d$ is endowed with its canonical $L^2$-metric $\Vert v\Vert = \sqrt{\sum_{k=1}^d v_k^2}$ and its relative distance $d$. The manifold $M$ is endowed with a Riemannian metric so that the chart $U \to \mathbb{R}^d$ is a Riemannian isometry onto its image.\\

We aim for the map $f=f_1$ to possess a family of attractive orbits $\{x^i_n\}_{0 \leq i \leq \rho_n-1}$ with increasing periodicities $\rho_n$, each having distinct attraction basins in the sense that no $\varepsilon$-orbit can transition from one basin to another.\\

Let $(\rho_n)_{n\geq 0}$ be an increasing sequence of positive integers with $\rho_0 \geq 2$. Let $(r_n)_{n\geq 0}$ be a decreasing sequence of positive radii, converging to $0$. We take $r_0$ smaller than the radius of the ball $B$ defining the chart of $M$. Consider the circles $O_n = \{ (r_n, \theta,0,..,0) \; | \; \theta \in \mathbb{T}^1\} \subset \mathbb{R}^2 \times \{0\}$ and set the $n$-th orbit to be $x^i_n = (r_n, \frac{i}{\rho_n}, 0, .. ,0) \in O_n$ for $i=0,..,\rho_n-1$.\\

Let $(\delta_n)_{n\geq 0}$ be a decreasing sequence of sufficiently small positive real numbers. And consider the following different sets
\begin{equation} \label{Sets}
	\begin{split}
		O_n &= \{ (r_n, \theta,0,..,0) \; | \; \theta \in \mathbb{T}^1\}\\
		B^i_n & = \{ x \in \mathbb{R}^d \; | \; d(x,x^i_n) < \delta_n \} \\
		B_n &= \{ x \in \mathbb{R}^d \; | \; d(x,O_n) < \delta_n \} \\
		A_n &= \{ x \in \mathbb{R}^d \; | \; \delta_n < d(x,O_n) < 2\delta_n \} \\
		C_n &= A_n \cup \overline{B_n}= \{ x \in \mathbb{R}^d \; | \; d(x,O_n) < 2\delta_n \}\\
		D_n &= \{ x \in \mathbb{R}^d \; | \; 2\delta_n < d(x,O_n) < 3\delta_n \} \\
		D &= M \setminus \big(\bigcup_{n=0}^\infty \overline{C_n}\big) \\
		D' &= M \setminus \big(\bigcup_{n=0}^\infty \overline{C_n \cup D_n}\big)
	\end{split}
\end{equation}
The radii $\delta_n$ are taken small enough so that all of these defined sets are in the chart $B$, all the closed balls $(\overline{B^i_n})_i$ are disjoint and all the closed sets $(\overline{C_n \cup D_n})_{n \geq 0}$ are disjoint.\\

\begin{rem} \label{RemarkDimension}
	Note that in dimension $2$, the sets $A_n$ have two connected components contrasting with the higher dimensional case. We denote these connected components by $A^\pm_n$ where $A^+_n$ corresponds to the larger $r$-coordinate, and $A^-_n$ to the lower. We also define $D^\pm_n$ in the same way.
	
	This distinction will take more importance in the study of the regularity in Section \ref{RegSection} where we will need to separate the two cases.
\end{rem}

We consider the following points
\begin{equation} \label{Points}
	\begin{split}
		x^i_n & = (r_n, \frac{i}{\rho_n}, 0, .. ,0) \in O_n \\
		y_n &= (r_n+ \delta_n, 0, 0, .. ,0) \in \partial B_n \\
		z_\infty &= (0,0,..,0) \in \overline{D} \\
		z^\pm_n &= (r_n \pm 2\delta_n, 0, 0, .. ,0) \in \partial A^\pm_n \cap \partial C_n \text{ in the 2D case}
	\end{split}
\end{equation}

The different sets of (\ref{Sets}) and points of (\ref{Points}) are represented in the Figure \ref{FigureConstruction} down below.\\

\begin{figure}[!htbp]
	\centering
	\begin{subfigure}[b]{0.55\linewidth}
		\includegraphics[width=\linewidth]{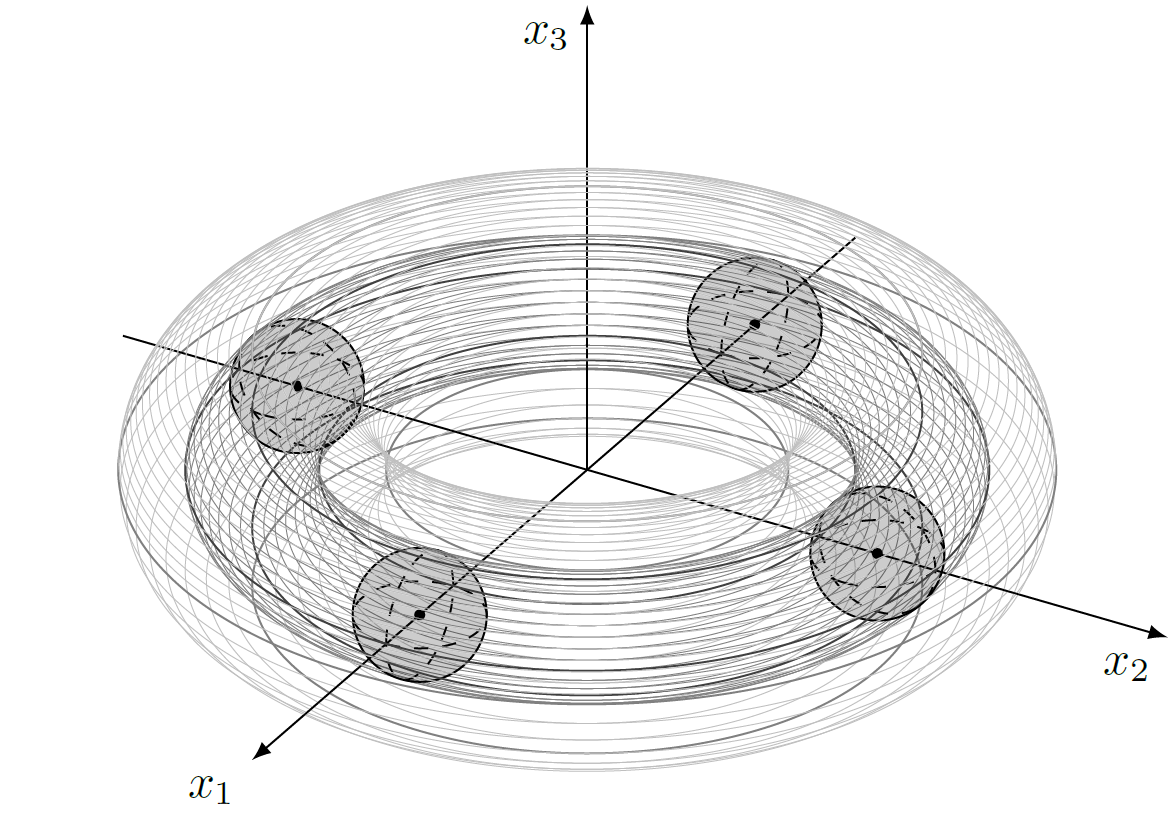}
		\caption{3D View of $C_n$.}
		\label{3DModel}
	\end{subfigure}
	\begin{subfigure}[b]{0.6\linewidth}
		\includegraphics[width=\linewidth]{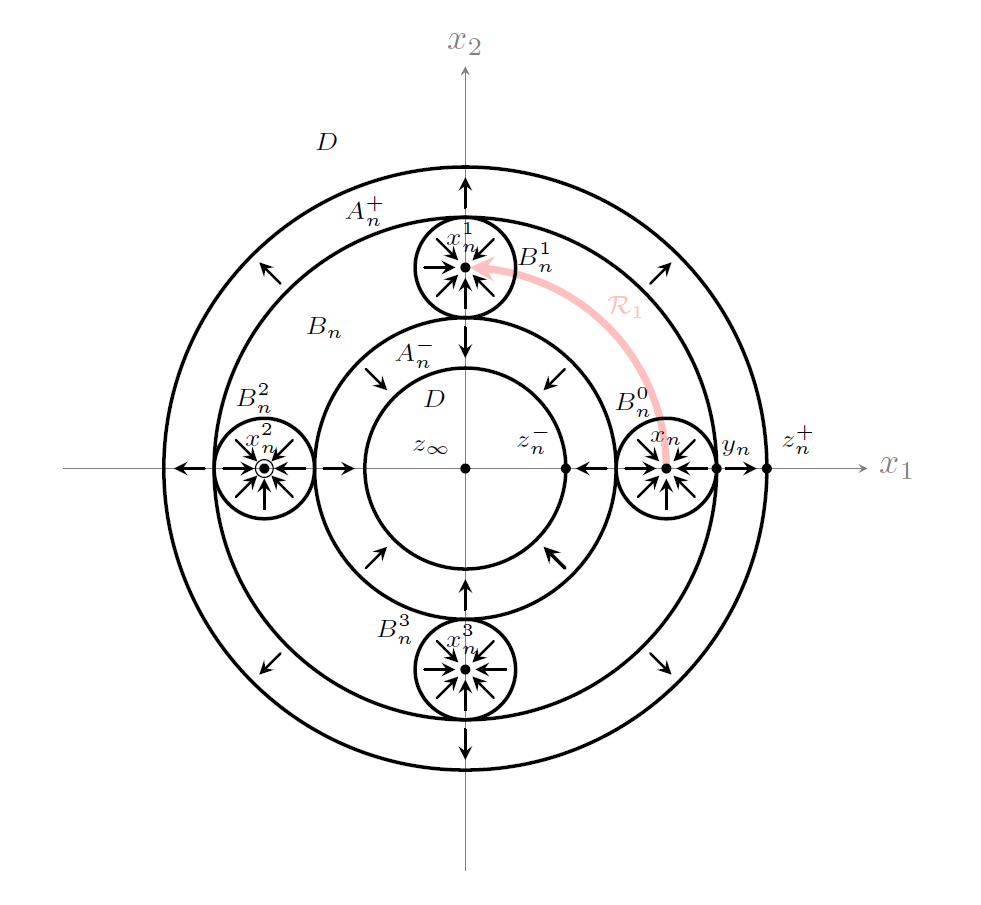}
		\caption{Top View of $C_n$ / The 2D Case.}
		\label{TopView}
	\end{subfigure}
	\begin{subfigure}[b]{0.57\linewidth}
		\includegraphics[width=\linewidth]{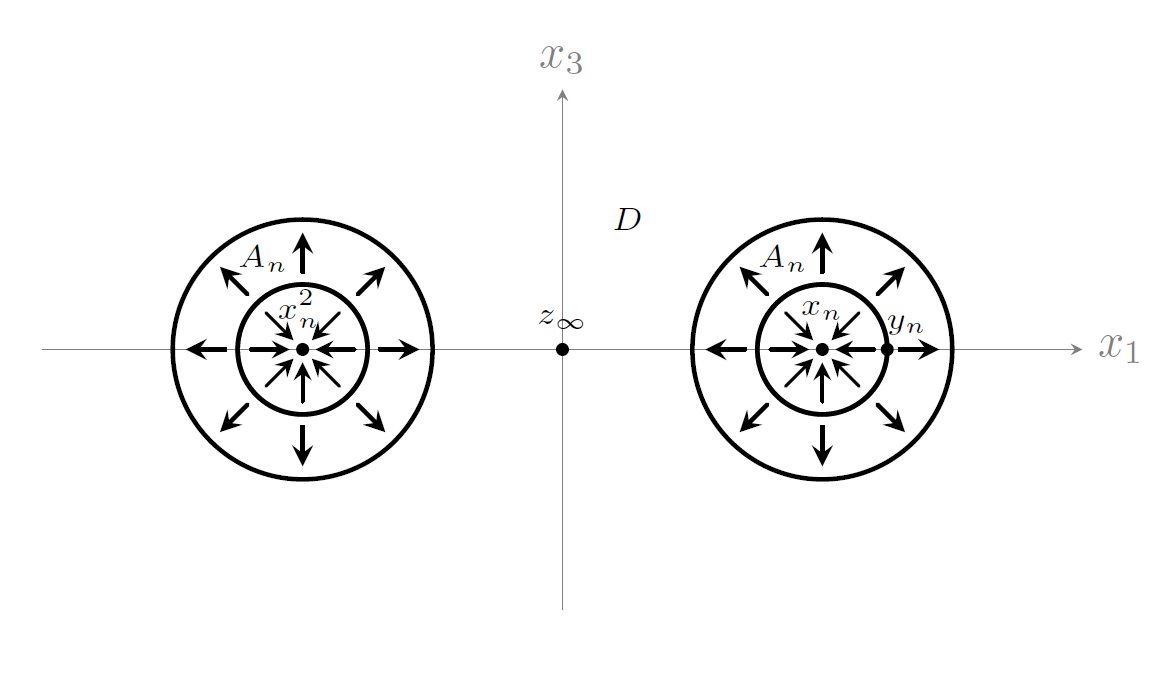}
		\caption{Side View of $C_n$.}
		\label{SideView}
	\end{subfigure}
	\caption{The dynamics of the flow $g_t$ on the set $C_n$}
	\label{FigureConstruction}
\end{figure}

\subsubsection{The Isotopy $f_t$} \label{tomf}

Now we define the isotopy $f_t$ as the composition of a "radial" autonomous Lagrangian flow $g_t$ with a rotational isotopy $\mathcal{R}_t$.

\paragraph{The Isotopy $g_t$.} Let ($\varsigma_n)_n$ be a sequence of small real numbers, let $\chi : \mathbb{R} \to \mathbb{R}_{\geq 0}$ be a smooth bump function with support in $[0,1]$ and positive on $(0,1)$. And let $p_{C_n}: A_n \to  \partial C_n$ be the projection on $\partial C_n$ defined as follows
\begin{equation}
	p_{C_n}(x) = p_{O_n}(x) + 2\delta_n.\frac{x- p_{O_n}(x) }{\Vert x- p_{O_n}(x)\Vert} \quad \text{where } p_{O_n}(x) = (\rho_n, \theta_x, 0,..,0)
\end{equation}
We define the autonomous vector field $Z : M \to TM$ by

\begin{equation} \label{Mapg}
	Z(x) = 
	\begin{cases}
		-\varsigma_n. \chi ( \Vert x- x^i_n\Vert^2 / \delta_n^2 ).(x- x^i_n) & \text{if } x \in B^i_n \\
		-\varsigma_n. \chi ( \Vert x- p_{C_n}(x)\Vert^2 / \delta_n^2 ).(x- p_{C_n}(x)) & \text{if } x \in A_n \\
		0  & \text{elsewhere} \\
	\end{cases}
\end{equation}
and let $(g_t)_{t \in [0,1]}$ be the associated flow.\\

The restriction of $g:=g_1$ to $B^i_n$ is radially symmetric with respect to the center $x^i_n$ which is attractive. Similarly, the restriction of $g$ to $A_n$ has a radial symmetry with respect to the circle $O_n$ meaning that for a fixed angle $\theta$, the restriction of $g$ to the "hollowed" ball $A_n \cap \{\theta_x = \theta\}$ is symmetric with respect to its center $O_n \cap \{\theta_x = \theta\}$. Moreover, the set $\partial C_n$ is attractive for its dynamics. This can be observed in Figure \ref{FigureConstruction}.\\

Computation shows that the vector field $Z$ is $C^k$-regular if  $\varsigma_n \underset{n \to \infty}{=} o( \delta_n^k )$.

\paragraph{The Isotopy $\mathcal{R}_t$.} Let $ \eta :\mathbb{R} \to \mathbb{R}$ be a smooth map, null on $( - \infty,0]$, constant equal to $1$ on $[1, +\infty)$ and increasing on $[0,1]$. And consider for all integer $n \geq 0$ the map $\eta_n : D_n \to [0,1]$ defined by
\begin{equation}
	\eta_n (x) = 1-\eta(\Vert x- p_{C_n}(x)\Vert^2/\delta_n^2 )
\end{equation}
For all angles $ \alpha \in \mathbb{T}^1$, we define the rotation $R_\alpha : \mathbb{R}^d \to \mathbb{R}^d$ by 
\begin{equation}
	R_\alpha (r, \theta, x_3, .. ,x_d) = (r, \theta + \alpha, x_3, .. ,x_d)
\end{equation}

The isotopy $(\mathcal{R}_t)_{t \in [0,1]}$ is taken as 
\begin{equation} \label{MapR_t}
	\mathcal{R}_t(x) = 
	\begin{cases}
		R_{\frac{1}{\rho_n} \eta(t)}(x) & \text{if } x \in \overline{C_n} \\
		\left(r, \theta +  \frac{\eta(t)}{\rho_n}\eta_n(x),x_3,..,x_d \right) & \text{if } x \in D_n \\
		x & \text{elsewhere}
	\end{cases}
\end{equation}
It is a rotation of constant angle $\rho_n$ on the sets $C_n$, and these angles decrease radially and progressively on $D_n$ until they reach zero on $D'$.\\

Computation shows that the map $\mathcal{R}_t$ is $C^k$-regular if $\frac{1}{\rho_n} \underset{n \to \infty}{=} o( \delta_n^k )$.\\

\paragraph{The Isotopy $f_t$.} The desired isotopy $(f_t)_{t \in [0,1]}$ is the composition
\begin{equation}
	f_t = \mathcal{R}_t \circ g_t
\end{equation} 

\begin{rem} \label{Symmetriesf}
\textbf{Symmetries of $\mathcal{R}_t$ and $g_t$.} Notice that 
	\begin{itemize}
		\item In a neighbourhood of the sets $\overline{C_n}$, the isotopies $\mathcal{R}_t$, $g_t$ and hence $f_t$ commutes with the rotation $R_{\frac{i}{\rho_n}}$ for every $0 \leq i < \rho_n$.
		\item In the balls $B^i_n$, the isotopy $g_t$ is symmetric with respect to the center $x^i_n$ of $B^i_n$ meaning that it commutes with all affine isometries of the form $x^i_n +U$ where $U \in O(n)$.
		\item In the sets $A_n$, the isotopies $\mathcal{R}_t$, $g_t$ and hence $f_t$ are symmetric with respect to the circles $O_n$ meaning that for any fixed angle $\theta$, they commute with all the affine isometries defined on the set $A_n \cap \{\theta_x = \theta \}$ and of the form $x_n^\theta + U$ where $x_n^\theta = (\rho_n, \theta , 0,..,0) \in O_n$ and $U \in O(n-1)$. Additionally, in the sets $A_n$, they commute with the translations in the $\theta$-coordinate.
	\end{itemize}
\end{rem}

\begin{rem} \label{RegularityConditions1}
	\begin{enumerate}
		\item Computation shows that the smaller the quantities $\frac{1}{\rho_n}$ and $\varsigma_n$ are, the closer is the map $f=f_1$ to the identity $Id_M$ in the $C^\infty$-topology. 
		\item As mentioned for the maps $g_t$ and $\mathcal{R}_t$, the map $f_t$ is $C^{k+1}$-regular provided that
		\begin{equation}
			\varsigma_n \underset{n \to \infty}{=} o( \delta_n^k ) \quad \text{and} \quad \frac{1}{\rho_n} \underset{n \to \infty}{=} o( \delta_n^k )
		\end{equation}
		Assuming that these identities hold for all integers $k\geq 0$, we obtain a smooth isotopy $f_t$.
	\end{enumerate}
\end{rem}

\subsubsection{The Mañé Lagrangian $L$} \label{fcme}

The Mañé Lagrangian $L$ is defined using the vector field $X_t$ of the constructed isotopy $f_t$. Additionally, to ensure periodicity in time for $L$, it is essential to confirm that the vector field $X_t$ is itself time-periodic.

\paragraph{The Vector Field $X_t$.} The first point of Remark \ref{RegularityConditions1} tells us that if $\frac{1}{\rho_n}$ and $\delta_n$ are small enough, we get for all times $t \in [0,1]$
\begin{equation}
	\Vert f_t - Id_M\Vert_{C^1} < 1
\end{equation}
and in particular, the map $f_t$ is invertible for all times $t$.\\ 

This allows us to define the corresponding vector field $X(t,x)= X_t(x) : \mathbb{R} \times M \to TM$ by
\begin{equation}
	X_t = \frac{df_t}{dt} \circ f_t^{-1}
\end{equation}

We know that for all positive integer $k \geq 1$, $\eta^{(k)}(0) = \eta^{(k)}(1)=0$. Thus, $X_0 = X_1 =0$ and so are all of its time derivatives. This implies that the vector field $X_t$ reduces to be a smooth time periodic vector field $X_t : \mathbb{T}^1 \times M \to TM$.

Consider the time periodic vector fields $Y_t$ relative to the isotopy $\mathcal{R}_t$ and recall the relation between $Z$ and $g_t$
	\begin{equation} \label{VectorFields}
		Y_t = \frac{d\mathcal{R}_t}{dt} \circ \mathcal{R}_t^{-1} \quad \text{and} \quad Z = \frac{dg_t}{dt} \circ g_t^{-1}
	\end{equation}

\begin{prop} \label{XP}
	We have
	\begin{equation} \label{XFormula}
		X_t = Y_t + d\mathcal{R}_t. Z \circ \mathcal{R}^{-1}_t
	\end{equation}
\end{prop}
\begin{proof}
	We have 
	\begin{align*}
		\frac{df_t}{dt} = \frac{d(\mathcal{R}_t \circ g_t)}{dt} =  \frac{d\mathcal{R}_t}{dt} \circ g_t + d\mathcal{R}_t . \frac{dg_t}{dt}
	\end{align*}
	Then
	\begin{align*}
		X_t & = \frac{df_t}{dt} \circ f_t^{-1} = \frac{d\mathcal{R}_t}{dt} \circ g_t \circ g^{-1}_t \circ \mathcal{R}^{-1}_t + d\mathcal{R}_t . \frac{dg_t}{dt} \circ g^{-1}_t \circ \mathcal{R}^{-1}_t  \\
		&=  Y_t + d\mathcal{R}_t. Z \circ \mathcal{R}^{-1}_t
	\end{align*}
\end{proof}

\begin{rem} \textbf{Symmetries of $X_t$.} \label{SymmetryX}
	We can deduce from the symmetries of $\mathcal{R}_t$ and $g_t$ stated in Remark \ref{Symmetriesf} that
	\begin{itemize}
		\item In a neighbourhood of the sets $\overline{C_n}$, the vector field $X_t$ is $R_{\frac{i}{\rho_n}}$-invariant for every $0 \leq i < \rho_n$.
		\item In the balls $B^i_n$, the vector field $Z$ is symmetric with respect to the center $x^i_n$ of $B^i_n$.
		\item In the sets $A_n$, the vector field $X_t$ is symmetric with respect to the circles $O_n$ and with respect to the $\theta$-coordinate.
		\item In the sets $D_n$, we have equality $X_t = Y_t$.
		\item In the set $D'$, the vector field $X_t$ is null.
	\end{itemize}	 
	These symmetries will show crucial in the study of regularity in Section \ref{RegSection}.
\end{rem}

\paragraph{The Mañé Lagrangian $L$.} Following the appendix of \cite{MR1166538}, we define a smooth time-periodic Mañé Lagrangian $L: \mathbb{T}^1 \times TM \to \mathbb{R}$ by
\begin{equation} \label{LagMane}
	L(t,x,v) = \frac{1}{2} \Vert v-X_t(x)\Vert^2
\end{equation}
where $\Vert.\Vert$ is the norm coming from the metric of $M$. It reduces to the $L^2$-norm when restricted to the chart where the dynamics are non trivial.

\begin{prop} \label{ManeHamiltonian}
	The \Mane Hamiltonian $H: \mathbb{T}^1 \times T^*M \to \mathbb{R}$ associated to the \Mane Lagrangian $L$ is given by
	\begin{equation} \label{ManeHamiltonianFormula}
		H(t,x,p) = \frac{1}{2} \Vert p + X_t(x) \Vert^2 - \frac{1}{2} \Vert X_t(x) \Vert^2
	\end{equation}
	with vector field $X_H= (X_H^x,X_H^p)$ given by
	\begin{equation} \label{ManeHamiltonianVF}
		X_H^x(t,x,p) = p + X_t(x) \quad \text{and} \quad X_H^p(t,x,p) =  -p.dX_t(x)
	\end{equation}
\end{prop}

\begin{proof}
	Using the metric $\langle \cdot , \cdot \rangle$ induced by the norm $\Vert \cdot \Vert$, we identify elements of $v$ of$T_xM$ with elements $p$ of $T^*_xM$ as follows : $p = \langle v, \cdot \rangle $. Then, we have for all $(t,x) \in \mathbb{R} \times M$ and $(v,p) \in T_xM \times T^*_xM$
	\begin{align*}
		2.\big[ H(t,x,p) + L(t,x,v) - p.v \big] &= \Vert p + X_t(x) \Vert^2  - \Vert X_t(x) \Vert^2 + \Vert v-X_t(x)\Vert^2 - 2p.v \\
		&= \Vert p \Vert^2 + 2p.X_t(x) + \Vert v \Vert^2 - 2v.X_t(x) + \Vert X_t(x) \Vert^2 - 2p.v \\
		&=  \Vert p - v + X_t(x) \Vert^2 \geq 0
	\end{align*}
	with equality if and only if $p = v-X_t(x) = \partial_v L(t,x,v)$. 
	
	Direct computation yields
	\begin{equation*}
		X_H^x(t,x,p) = \partial_p H(t,x,p) = p + X_t(x) \quad \text{and} \quad X_H^p(t,x,p) = - \partial_xH(t,x,p) =  -p.dX_t(x)
	\end{equation*}
\end{proof}

\begin{prop}
	The \Mane Lagrangian $L$ defined in \eqref{LagMane} is Tonelli.
\end{prop}   
\begin{proof}
	\textit{Regularity.} is verified by the assumptions of Remark \ref{RegularityConditions1}.\\
	
	\textit{Strict Convexity.} We denote by $\langle \cdot , \cdot \rangle$ the scalar product associated to the norm $\Vert.\Vert$ on $TM$. Then, for all $(t,x,v) \in \mathbb{T}^1 \times TM$, we have
	\begin{equation*}
		\partial_v L(t,x,v) = \langle v-X_t(x), \cdot \rangle \quad \text{and} \quad \partial^2_{vv} L(t,x,v) = \langle \cdot ,  \cdot \rangle
	\end{equation*}
	It follows that $\partial^2_{vv} L(t,x,v) >0$ for all $(t,x,v) \in \mathbb{T}^1 \times TM$.\\
	
	\textit{Superlinearity.} Since the norm is squared in the definition \eqref{LagMane} of $L$, we get $\frac{L(t,x,v)}{\Vert v \Vert} \to \infty$ as $\Vert v \Vert \to \infty$.\\
	
	\textit{Completeness.} We will prove the completeness for the Hamiltonian flow $\phi_H$ which is conjugated to the Lagrangian flow $\phi_L$ by the Legendre map $\mathcal{L}$ introduced in Subsection \ref{SubsectionTonelli}. Note from the formula of the Hamiltonian vector field $X_H$ expressed in \eqref{ManeHamiltonianVF}, its time-periodicity and from the compactness of the manifold $M$ that there exist two positive real numbers $C_1$ and $C_2 >0$ such that for all $(s,x,p) \in \mathbb{T}^1 \times T^*M$, we have
	\begin{equation*}
		\Vert X_H(s,x,p) \Vert \leq C_1\Vert p \Vert + C_2
	\end{equation*}
	We work locally on charts in $\mathbb{R}^d$. Fix a time $T>0$. For all time $t$ such that $|t-s| \leq T$ and $\phi_H^{s,t}(x,p)$ is well defined and belongs to the chart, we have
	\begin{equation*}
		\left\Vert \frac{d}{dt} \phi_H^{s,t}(x,p) \right\Vert = \Vert X_H(t,\phi_H^{s,t}(x,p)) \Vert \leq C_1 \Vert \phi_H^{s,t}(x,p)\Vert + C_2
	\end{equation*}
	Hence, we get
	\begin{align*}
		\Vert \phi_H^{s,t}(x,p) \Vert & \leq \Vert \phi_H^{s,t}(x,p) - (x,p) \Vert + \Vert (x,p) \Vert \\
		& = \left\Vert \int_s^t X_H(\tau, \phi_H^{s,\tau}(x,p)) \; d\tau \right\Vert + \Vert (x,p) \Vert \\
		& \leq \int_s^t \Vert X_H(\tau, \phi_H^{s,\tau}(x,p)) \Vert d\tau + \Vert (x,p) \Vert \\
		& \leq C_1 \int_s^t \Vert \phi_H^{s,\tau}(x,p) \Vert d\tau + C_2.(t-s)+ \Vert (x,p) \Vert \\
		& \leq C_1 \int_s^t \Vert \phi_H^{s,\tau}(x,p) \Vert d\tau + C_2T+ \Vert (x,p) \Vert
	\end{align*}
	Applying Gr\"onwall lemma yields
	\begin{equation*}
		\left\Vert \phi_H^{s,t}(x,p) \right\Vert \leq (C_2T+ \Vert (x,p) \Vert). e^{C_1(t-s)} \leq (C_2T+ \Vert (x,p) \Vert). e^{C_1T}
	\end{equation*}
	This compactness result for a fixed $T>0$ allow to extend $\phi_H^{s,t}(x,p)$ to all times $t$ such that $|t-s| \leq T$. Since this holds for all $T > 0$, we deduce the completeness of the Hamiltonian vector field $X_H$, and by conjugacy, the completeness of the Lagrangian vector field $X_L$.
\end{proof}

\subsection{Evaluation of the Peierls Barriers of Mañé Lagrangians} \label{SectionManeExample}

This subsection examines the behaviour of the Peierls barriers $h^\infty$ and $h^{k\infty}$, that are specific to Mañé Lagrangians. These will provide a straightforward way to estimate the values of the barriers between two points in $M$. \\

\subsubsection{The $0$-set of the Peierls Barriers on Pseudo-Orbits of Mañé Lagrangians}

Prior to selecting a viscosity solution that addresses the main Theorem \ref{MainC1} of this paper, one needs to be able to identify the Mather set and the Peierls set associated to the constructed Mañé Lagrangian $L$. This constitutes the objective of the current subsection, where our focus will be on understanding the $0$-set of the Peierls Barriers.\\

We recall some properties of the Lagrangian $L(t,x,v) = \frac{1}{2} \Vert v-X_t(x)\Vert^2$ mentioned in the previous subsection.
\begin{itemize}
	\item The projection on $M$ of the Euler-Lagrange flow $\phi_L^t$ restricted to the subset $\mathcal{L}_L := \{v = X_t(x)\} \subset TM$ corresponds precisely to the flow $f_t$ of $X_t$.
	\item The subset $\mathcal{L}_L$ is a graph over the zero-section $0_{TM} \simeq M$ and its Hamiltonian counterpart $\mathcal{L}_H \subset T^*M$ is the zero-section $0_{T*M} \simeq M$ of the cotangent bundle. Furthermore, the Hamiltonian flow $\phi_H^t$ restricted to the zero-section is precisely the flow $f_t$.
\end{itemize}

The following proposition justifies the choice of $\alpha_0 =0$ and of considering the Hamilton-Jacobi equation \eqref{HJ}.

\begin{prop} \label{ManeCVnull}
	The Mañé critical value $\alpha_0$ of \Mane Lagrangians of the form $L(t,x,v) = \frac{1}{2} \Vert v - X_t(x) \Vert^2 $ is null.
\end{prop}

\begin{proof}
	Recall that the Mañé critical value is the real
	\begin{equation*}
		\alpha_0 = - \inf_\mu \left\{\int_{\mathbb{T}^1 \times TM}L \;d\mu \right\} \leq 0
	\end{equation*} 
	where the infimum is taken over compact supported probability measures $\mu$ invariant by the Euler-Lagrangian flow corresponding to $L$.\\
	Following the proof of the Krylov–Bogolyubov theorem (see \cite{MR1109661}), there exists a $f_t$-invariant measure $\mu$ supported on the submanifold $\{(t,x,v) \; | \; v=X_t(x)\} \subset TM$. Integrating $L$ with respect to $\mu$ gives
	\begin{equation*}
		-\alpha_0 \leq \int_{\mathbb{T}^1 \times T^*M}L \;d\mu  =\frac{1}{2} \int_{\supp(\mu)} \Vert v-X_t(x)\Vert^2 \;d\mu =0 
	\end{equation*}
	Therefore $\alpha_0 =0$.
\end{proof}

We get the identifications $h=h_0$ and $\mathcal{T} = \mathcal{T}_0$. And since the \Mane Lagrangian $L$ is non-negative, we get the following.
\begin{cor} \label{hpos}
	For all $x,y \in M$, all times $s<t$ and all integers $n,k \geq 1$, $h^{s,t}(x,y) \geq 0$ and $h^{n\infty}(x,y) \geq 0$.
\end{cor}

\begin{rem}
	In a general framework, we only have $h^\infty(x,x) \geq 0$, however for two distinct points $x$ and $y$, the value of $h^\infty(x,y)$ can potentially be negative. The exceptional non-negativity of the barriers highly simplifies the Aubry-Mather theory of Mañé Lagrangians. 
\end{rem}

We will discern the conditions under which the Peierls Barrier vanishes and when it remains positive. But first, we introduce the following definition.

\begin{defi}
	For any $\varepsilon >0$, any fixed time $\tau >0$, and any two points $x$ and $y$ of $M$,  we call a \textit{$(\varepsilon,\tau)$-pseudo-orbit} of the flow of $X_t$ between $x$ and $y$ a finite family of curves $(\gamma_k : [S_k,T_k] \to M)_{0 \leq k \leq m}$ such that 
	\begin{enumerate}[label=\roman*.]
		\item $S_0 = 0$ and $\gamma_0(0) = x$. 
		\item For all $0 \leq k \leq m$, $\dot{\gamma}_k(t)= X_t(\gamma_k(t))$.
		\item The real times $S_k$ and $T_k$ verify $T_k-S_k \geq \tau$ and $S_{k+1}=T_k \; mod \; 1$.
		\item For all $0 \leq k \leq m-1$, $d\big(\gamma_k(T_k),\gamma_{k+1}(S_{k+1})\big) < \varepsilon$ and $d\big(\gamma_{m}(T_{m}),y\big)< \varepsilon$.
	\end{enumerate}
\end{defi}

\begin{prop} \label{chrec0}
	Let $x$ and $y$ be two points of $M$ such that there exists an increasing real sequence of positive times $(t_n)_{n \geq 0}$ with $\lim_n t_n = + \infty$ and $\lim_n h^{t_n}(x,y)=0$. Then for every $\varepsilon >0$ and $\tau>0$, there exists an $(\varepsilon, \tau)$-pseudo-orbit of the flow of $X_t$ between $x$ and $y$.
	
	By contraposition, if for some fixed $\tau >0$ there exists a constant $\varepsilon>0$ such that no $(\varepsilon,\tau)$-pseudo-orbit of the flow of $X_t$ links $x$ to $y$, then $ \liminf\limits_{t\to + \infty} h^{t}(x, y) > 0$.
\end{prop}

\begin{lem}
	Let $x$ and $y$ be two points of $M$ such that $\lim_n  h^{t_n}(x,y)=0$ with $t_n \to + \infty$ i.e. such that there exists a sequence of minimizing curves $(\sigma_n : [0,t_n] \to M)_{n\in \mathbb{N}}$ between $x$ and $y$ with $\lim\limits_n \int_0^{t_n} L(t,\sigma_n(t), \dot{\sigma}_n(t)) \; dt=0$. Then
	\begin{equation*}
		\sup_{t\in [0,t_n]} \Vert \dot{\sigma}_n(t) - X_t(\sigma_n(t)) \Vert \underset{n\to\infty}{\to} 0
	\end{equation*}
\end{lem}

The lemma shows that these curves $\sigma_n$ get closer to the set $\{(t,x,v) \; | \; v=X_t(x)\}$ as $n$ grows. This enables to deduce that the Mather set (and the Aubry set) are contained in the zero-level of $L$.

\begin{proof}
	Arguing by contradiction, suppose that there exists $\delta >0$ such that we can find an increasing sequence $k_n$ of integers verifying for all $n\in \mathbb{N}$, $\sup_{t\in [0,t_{k_n}]} \Vert \dot{\sigma}_{k_n}(t) - X_t(\sigma_{k_n}(t)) \Vert > \delta$.\\
	Let $n \in \mathbb{N}$. There exists $s_n \in [0,t_{k_n}]$ such that $\Vert\dot{\sigma}_{k_n}(s_n) - X_{s_n}(\sigma_{k_n}(s_n)) \Vert > \delta$. Thus $(s_n,\sigma_{k_n}(s_n), \dot{\sigma}_{k_n}(s_n))$ belongs to the open set $U$ of $\mathbb{T}^1 \times TM$ defined as
	\begin{equation}
		U:= \left\{(t,x,v) \in \mathbb{T}^1 \times TM \; \left| \; L(t,x,v) > \frac{\delta^2}{2} \right.\right\}
	\end{equation}
	We also define the set $F$ as
	\begin{equation}
		F:= \left\{(t,x,v) \in \mathbb{T}^1 \times TM \; | \; L(t,x,v) = 0 \right\}
	\end{equation}
	It is easy to see that $F$ and its complement are invariant under the map $\Phi_L^\tau$ defined in \eqref{InvarianceFlow}. And knowing that $\bar{U} \subset F^c$, we get the inclusion
	\begin{equation*}
		\Phi_L^{[0,1]}(\bar{U}) := \{ (t + \tau, \phi_L^{t, t+ \tau}(x,v) \; | \; (t,x,v) \in \bar{U}, \; \tau\in [0,1] \} \subset F^c
	\end{equation*}
	Since $L$ is continuous and since the sets $\Phi_L^{[0,1]}(\bar{U})$ and $F$ are disjoint with one of them being closed and the other compact in $\mathbb{T}^1 \times TM$, there exists $\nu>0$ such that
	\begin{equation}\label{inclusion}
		\Phi_L^{[0,1]}(\bar{U}) \subset \left\{(t,x,v) \in \mathbb{T}^1 \times TM \; | \; L(t,x,v) > \nu \right\}
	\end{equation}
	
	Now, we study the action along any of the curves $\sigma_n$ 
	\begin{align*}
		A_L(\sigma_n) & = \int_0^{t_n} L(t,\sigma_n(t), \dot{\sigma}_n(t)) \; dt = \int_0^{t_n} \frac{1}{2} \Vert\dot{\sigma}_n(t) - X_t(\sigma_n(t))\Vert^2 \; dt \underset{n\rightarrow \infty}{\longrightarrow} 0
	\end{align*}
	Thus there exists $N>0$ such that for all $n \geq N$, $A_L(\sigma_n) < \nu$.\\
	Let $n\in \mathbb{N}$ be such that $k_n > N$, and suppose that $s_n < t_{k_n}-1$, otherwise use $\Phi_L^{[-1,0]}(\bar{U})$ instead of $\Phi_L^{[0,1]}(\bar{U})$.
	\begin{align*}
		A_L(\sigma_{k_n}) & = \int_0^{t_{k_n}} L(t,\sigma_{k_n}(t), \dot{\sigma}_{k_n}(t)) \; dt  \geq \int_{s_n}^{s_n+1} L(t,\sigma_{k_n}(t), \dot{\sigma}_{k_n}(t)) \; dt
	\end{align*}
	However, since $\sigma_{k_n}$ is a minimizing curve, the variational theory claims that it follows Euler-Lagrange flow as mentioned in Theorem \ref{TonelliTheorem}. This implies that for all $t\in [s_n,s_n+1]$, 
	\begin{align*}	
		(t,\sigma_{k_n}(t),\dot{\sigma}_{k_n}(t)) = \Phi_L^{t-s_n}(s_n,\sigma_{k_n}(s_n),\dot{\sigma}_{k_n}(s_n)) \in \Phi_L^{[0,1]}(\bar{U}) 
	\end{align*}
	Therefore we get from the inclusion (\ref{inclusion}) that
	\begin{equation*}
		A_L(\sigma_{k_n}) \geq \int_{s_n}^{s_n+1} L(t,\sigma_{k_n}(t),\dot{\sigma}_{k_n}(t)) \; dt \geq \nu
	\end{equation*}
	which contradicts the definition of the definition of $N$.
\end{proof}

\begin{proof}[Proof of Proposition \ref{chrec0}]
	Fix $\varepsilon >0$. Let $\delta >0$ be a real number. Using the lemma, we get a minimizing cruve $\sigma : [0,T] \to M$ from $x$ to $y$ with $T> \tau$ such that for all $t \in [0,T]$, $\Vert \dot{\sigma}(t) - X_t(\sigma(t)) \Vert \leq \delta$.\\
	Let $\alpha \in \mathbb{N}$ and $0 \leq \beta < \tau$ be such that $T = \alpha \tau + \beta$. For $k \in \{ 1,..,\alpha -1 \}$, We define the curves $\gamma_0 : [0, \tau + \beta] \to M$ and $\gamma_k : [k\tau + \beta,(k+1)\tau + \beta] \to M$ as follows
	\begin{equation}
		\begin{cases}
			\dot{\gamma}_k(t) = X_t( \gamma_k(t) ) \\
			\gamma_k(S_k) = \sigma(S_k)
		\end{cases}
	\end{equation}
	where we denoted by $[S_k,T_k]$ the domains of $\gamma_k$. Thus we have for all $t \in [S_k,T_k]$,
	\begin{align*}
		\Vert \sigma(t) - \gamma_k(t) \Vert & = \left\Vert \int_{S_k}^t \dot{\sigma}(s) - X_s( \gamma(s) ) \; ds \right\Vert \\
		& \leq  \int_{S_k}^t \Vert \dot{\sigma}(s) - X_s( \sigma(s) ) \Vert \; ds + \int_{S_k}^t \Vert X_s( \sigma(s) ) - X_s( \gamma(s) ) \Vert \; ds \\
		& \leq (T_k - S_k)\delta + \int_{S_k}^t M. \Vert \sigma(s) - \gamma_k(s) \Vert \; ds\\
		& \leq 2\tau\delta + \int_{S_k}^t M. \Vert \sigma(s) - \gamma_k(s) \Vert \; ds
	\end{align*}
	where $M$ is a Lipschitz constant for the vector field $X_t$.\\
	And using the Gr\"onwall lemma, we get the upper bound
	\begin{align*}
		\forall t \in [S_k,T_k], \quad \Vert \sigma(t) - \gamma_k(t) \Vert \leq 2\tau\delta.e^{M(t-S_k)} \leq 2\tau\delta.e^{M2 \tau}
	\end{align*}
	Taking $\delta < \varepsilon.\frac{e^{-M2\tau}}{2\tau}$ gives the desired estimations. 
\end{proof}

A specialization of the previous proposition to the $(n,k)$-Peierls Barriers $h^{n\infty+k}$  (see Definition \ref{nPeierlsDefi}) leads to the variation below

\begin{cor} \label{gchrec} 
	Let $n \in \mathbb{N}^*$ and $0 \leq k \leq n-1$. Let $x$ and $y$ be two points of $M$ such that $h^{n\infty+k}(x,y)=0$. Then for every $\varepsilon >0$, there exists a finite family of curves $\big(\gamma_i : [0,T_i] \to M\big)_{0 \leq i \leq m}$ with integer times $T_i$ such that
	\begin{enumerate}[label=\roman*.]
		\item $\gamma_0(0) = x$.
		\item $T_0=k$ and for all $i>0$ $T_i$ is a multiple of $n$.
		\item For all $0 \leq i \leq m$, $\gamma_i(t)= f_t(\gamma_i(0))$.
		\item For all $0 \leq i \leq m-1$, $d\big(\gamma_i(T_i),\gamma_{i+1}(0)\big) < \varepsilon$, and $d\big(\gamma_{m}(T_{m}),y\big)< \varepsilon$.
	\end{enumerate}
\end{cor}

\subsubsection{Quantitative Properties of the Peierls Barriers $h^{k\infty}$}

We now present a lemma that compiles essential properties needed for our analysis, particularly in the study of the non-wandering set $\Omega(\mathcal{T})$ and in facilitating the selection of a smooth element within it.\\

We define the integer-time $\alpha$-limit set $\alpha_k(\gamma)$ and the integer-time $\omega$-limit set $\omega_k(\gamma)$ of a curve $\gamma : \mathbb{R} \to M$ as
		\begin{equation} \label{AlphaOmega1}
			 	\begin{split}
					\alpha_k(\gamma) &= \{ z \in M \; | \; \exists (q_n)_n \in \mathbb{N}^\mathbb{N}; \; \lim_n q_n = +\infty, \; \lim_n \gamma(-k.q_n) = z \} \\
					\omega_k(\gamma) &= \{ z \in M \; | \; \exists (q_n)_n \in \mathbb{N}^\mathbb{N}; \; \lim_n q_n = +\infty, \; \lim_n \gamma(k.q_n) = z \}
				\end{split}
		\end{equation}
		
Recall from \eqref{VectorFields} the vector fields $Y_t$ and $Z$ used to define $X_t$.

\begin{lem} \label{SameClass}
 	\begin{enumerate}
 		\item \label{SameClass0} Fix an integer $k \geq 1$. Let $x$ be a point of $M$ and let $y$ and $z$ be two respective points of $\alpha_k(f_t(x))$ and $\omega_k(f_t(x))$. Then 
 		\begin{equation}
 			h^{k\infty}(y,x) = h^{k\infty} (x,z) = h^{k\infty} (y,z) = 0
 		\end{equation}
 		\item \label{SameClass1} Let  $x$ be a $1$-periodic point under the flow $f_t$. Then for all integer $k \geq 1$ and $i \geq 0$, $h^{k \infty+i}(x,\cdot) = h^{ \infty}(x,\cdot)$ and $h^{k \infty+i}(\cdot,x) = h^{ \infty}(\cdot,x)$.
 		\item \label{SameClass2} Let $F$ be an arc-wise connected subset of $M$ such that the vector fields $X_t$ and $Y_t$ coincide on $f_t(F)$ and that there exist an integer $m$ such that $\mathcal{R}_m(F)=F$. Then, for all pair of points $x$ and $y$ in $F$, we have
		\begin{equation}
			\overline{h}(x,y) := \limsup_{k \to \infty} h^k(x,y) = 0
		\end{equation}
		In particular, for all integers $k \geq 1$ and $i \geq 0$,
		\begin{enumerate}[label=\roman*.]
			\item $h^{k \infty +i}(x,y) = 0$
			\item For any point $z \in M$, $h^{k \infty+i}(x,z) = h^{k \infty+i}(y,z)$ and $h^{k \infty+i}(z,x) = h^{k \infty+i}(z,y)$. 
			
			We will sometimes denote these quantities respectively by $h^{k \infty+i}(F,z)$ and $h^{k \infty+i}(z,F)$.
		\end{enumerate}
		\item \label{SeparateClass} Let $F$ be a subset defined as above and assume moreover that it is a closed $f_t$-invariant subset of $M$. Assume that the rotation $f_t = \mathcal{R}_t$ is $k$-periodic on $F$. Then, for all pair of points $x$ and $y$ in $M \setminus F$ which are separated by $F$, we have
		\begin{equation}
			h^{k\infty}(x, y) = h^{k\infty}(x, F) + h^{k\infty}(F, y)
		\end{equation}		 
 	\end{enumerate}
\end{lem}

\begin{proof}
	\ref{SameClass0}. This is an immediate implication of the Liminf property (\ref{PeierlsLiminf}) of the Peierls Barrier. More precisely,
	\begin{align*}
		0 \leq h^{k\infty}(y,x) \leq \liminf_i A_L( f_t(x)_{|[-ki,0]}) = 0
	\end{align*}		
	The other cases are analogous.\\

	\ref{SameClass1}. Let $x$ be a $1$-periodic point under the flow $f_t$. We fix two integers $k \geq 1$ and $i \geq 0$. We already know by definition of the Peierls barriers that
	\begin{align*}
		h^\infty(x,y) & = \liminf_n h^n(x,y) \leq \liminf_n h^{kn+i}(x,y) \leq h^{k \infty+i}(x,y) 
	\end{align*}
	Now let $\gamma_n : [0, k_n] \to M$ be a sequence of curves linking $x$ to $y$ with increasing integer times $k_n$ and such that $h^\infty(x,y) = \lim_n A_L(\gamma_n)$. We left-concatenate the curves $\gamma_n$ with the loops $f_t(x) : [0, (k-1).k_n+i] \to M$ based on the $1$-time periodic point $x$. We obtain new curves $\tilde{\gamma}_n : [0,k.k_n+i] \to M$ still linking $x$ to $y$. Since $f_t$ is of null action by the Mañé Lagrangian $L$, we get
	\begin{align*}
		h^\infty(x,y) &= \lim_n A_L(\gamma_n) = \lim_n A_L(\tilde{\gamma}_n) \geq \liminf_n h^{kn+i}(x,y) = h^{k \infty+i}(x,y) 
	\end{align*}
	The equality follows. The other equality is analogous. \\
	
	\ref{SameClass2}. Fix two points $x$ and $y$ in $F$. We distinguish two cases. The case where $y$ is periodic under the rotation $\mathcal{R}_t$ and the case where $\mathcal{R}_t$ is an irrational rotation at $y$.
	
	First assume that $y$ is periodic under $\mathcal{R}_t$ with integer period $\rho \geq 1$. We can consider for all integer $0 \leq i < \rho$, the points $y_i =\mathcal{R}_{i}^{-1}(y)$ and the curves $\gamma^i : [0,1] \to F$ connecting $x$ to $y_i$. The latter exist thanks to the arc-wise connectedness of $F$. Now for all integer time $k \geq 1$, we define the curve $ \gamma_k : [0,k] \to M$ by 
	\begin{equation*}
		\gamma_k (t)= \mathcal{R}_t \circ \gamma^i\left( \frac{t}{k} \right)
	\end{equation*}
	where $0 \leq i <\rho$ is such that $k  \equiv i \pmod {\rho}$. Then, its velocity is given by
	\begin{align*}
		\dot{\gamma}_k(t) &= \frac{d\mathcal{R}_t}{dt} \circ \gamma^i \left( \frac{t}{k} \right) + d\mathcal{R}_t. \frac{\dot{\gamma}^i}{k} \left( \frac{t}{k} \right) 
		= Y_t (\gamma_k(t)) + \frac{1}{k} d\mathcal{R}_t. \dot{\gamma}^i\left( \frac{t}{k} \right) 
	\end{align*}		
	Using the fact that $X_t = Y_t$ on $f_t(F)= \mathcal{R}_t(F)$, we get
	\begin{align*}
		A_L(\gamma_k) &= \int_0^k L(\tau, \gamma_k(\tau), \dot{\gamma}_k(\tau)) \; d\tau = \int_0^k \frac{1}{2} \Vert \dot{\gamma}_k(\tau) - Y_\tau(\gamma_k(\tau))\Vert^2 d\tau \\
		&= \int_0^k \frac{1}{2k^2} \left|\left| d\mathcal{R}_\tau. \dot{\gamma}^i\left( \frac{\tau}{k} \right) \right|\right|^2 d\tau \leq \Vert d\mathcal{R}_t\Vert^2_\infty. \max_{0 \leq j < \rho}\Vert\dot{\gamma}^j\Vert_\infty^2.\frac{1}{2k} \to 0 \quad \text{as } k \to \infty
	\end{align*}
	Therefore, we deduce that 
	\begin{equation}
		0 \leq \overline{h}(x,y) = \limsup_{k \to \infty} h^k(x,y) \leq \limsup_{k \to \infty} A_L(\gamma_k) =0
	\end{equation} 
	
	Now assume that $\mathcal{R}_t$ is an irrational rotation at $y$. Fix $\varepsilon >0$ and consider a large $n$ such that the points $y_i = \mathcal{R}_{i}^{-1}(y)$, $0\leq i <n$ are $\varepsilon$-dense in the orbit of $y$. For all integer $k \geq 1$ and $0\leq i <n$, we consider curves $\gamma^i: [0,1- \varepsilon /k ] \to F$ linking $x$ to $y_i$ with a uniform $C^1$-bound over $k$. Now fix $k \geq 1$ and let $0 \leq i < n$ be such that $d(y_k,y_i) < \varepsilon$ where $y_k = \mathcal{R}_{k}^{-1}(y) $. Consider a curve $\tilde{\gamma}^i : [0,\varepsilon] \to M$ linking $y_i$ to $y_k$ and such that $\Vert \dot{\tilde{\gamma}}^i \Vert_\infty \leq \varepsilon$. Now define the curve $\gamma_k : [0,k] \to M$ by
	\begin{equation*}
		\gamma_k(t) =
		\begin{cases}
			\mathcal{R}_t \circ \gamma^i\left( \frac{t}{k} \right) & \text{if } t \in [0, k-\varepsilon] \\
			\tilde{\gamma}^i (t - k + \varepsilon) & \text{if } t \in [k-\varepsilon, k]
		\end{cases}
	\end{equation*}
	Then, doing the same computations as in the periodic case, we obtain
	\begin{align*}
		A_L(\gamma_k) &= \int_0^{k-\varepsilon} L(\tau, \gamma_k(\tau), \dot{\gamma}_k(\tau)) \; d\tau + \int_{k-\varepsilon}^k L(\tau, \gamma_k(\tau), \dot{\gamma}_k(\tau)) \; d\tau \\
		& \leq \int_0^{k-\varepsilon} \frac{1}{2k^2} \left|\left| d\mathcal{R}_\tau. \dot{\gamma}^i\left( \frac{\tau}{k} \right) \right|\right|^2 d\tau + \int_{k-\varepsilon}^k L(\tau, \gamma_k(\tau), \dot{\gamma}_k(\tau)) \; d\tau \\
		& \leq \Vert d\mathcal{R}_t\Vert^2_\infty. \max_{0 \leq j < n}\Vert\dot{\gamma}^j\Vert_\infty^2.\frac{1}{2k} + \left(\sup_{\substack{ (\tau, z) \in \mathbb{T}^1 \times M  \\ \Vert v\Vert \leq \varepsilon}} L(\tau,z,v) \right). \varepsilon
	\end{align*}
	Taking $k \geq \frac{1}{\varepsilon}$, we conclude that $\lim_k A_L(\gamma_k) =0$ and that $\overline{h}(x,y) =0$. \\
	
	The remaining claimed identities follow from 
	\begin{align*}
		0 \leq h^{k \infty+i}(x,y) \leq \overline{h}(x,y) =0
	\end{align*}
	and from the triangular inequality (\ref{TriangInegPeierls}) of Proposition \ref{PeierlsProp} applied twice as below
	\begin{align*}
		h^{k \infty+i}(z,x) &\leq h^{k \infty+i}(z,y) + h^{k \infty}(y,x) = h^{k \infty+i}(z,y) \\
		&\leq h^{k \infty+i}(z,x) + h^{k \infty}(x,y) = h^{k \infty+i}(z,x)
	\end{align*}
	
	\ref{SeparateClass}. Let $x$ and $y$ be two points of $M \setminus F$ separated by $F$. We know from the triangular inequality (\ref{TriangInegPeierls}) that
	\begin{equation}
		h^{k\infty}(x, y) \leq h^{k\infty}(x, F) + h^{k\infty}(F, y)
	\end{equation}
	We need to prove the inverse inequality. Let $\gamma_i : [0, k.n_i] \to M$ be curves linking $x$ to $y$ such that $h^{k\infty}(x, y) = \lim_i A_L(\gamma_i)$. Since $F$ separates $x$ and $y$, there exist times $t_i \in (0, k.n_i)$ such that the points $z_i := \gamma_i(t_i)$ belong to $F$. 
	
	We concatenate the curve $\gamma_{i|[0,t_i]}$ with the curve $f_{t_i,t}(z_i): \big[t_i, k \lceil \frac{t_i}{k} \rceil \big] \to M$ to get a first curve $\gamma^1_i:\big[0, k \lceil \frac{t_i}{k} \rceil \big] \to M$ linking $x$ to $z_i^1:= f_{t_i,k \lceil \frac{t_i}{k} \rceil}(z_i)$. And since the flow $f_t$ of $X_t$ is of null action, we have $A_L(\gamma_{i|[0,t_i]}) = A_L(\gamma_i^1)$.
	
	Similarly, we concatenate the curve $f_{t_i,t}(z_i): \big[ k \lfloor \frac{t_i}{k} \rfloor, t_i \big] \to M$ with the curve $\gamma_{i|[t_i, kn_i]}$ to get a curve $\gamma_i^2 : \big[ k \lfloor \frac{t_i}{k} \rfloor,kn_i \big] \to M$ linking $z_i^2 := f_{t_i, k \lfloor \frac{t_i}{k} \rfloor}$ to $y$. And we still have $A_L(\gamma_{i|[t_i, kn_i]}) = A_L(\gamma_i^2)$.
	
	For $j = 1,2$, the points $z_i^j$ do belong to the $f_t$-invariant set $F$. And by compactness, we can assume that they converge up to extraction to $z^j \in F$. We get
	\begin{align*}
		h^{k\infty}(x, y) &= \lim_i A_L(\gamma_i) = \lim_i A_L(\gamma_i^1) + A_L(\gamma_i^2) \\
		& \geq h^{k \infty}(x,z^1) + h^{k \infty}(z^2,y) \\
		& =h^{k \infty}(x,F) + h^{k \infty}(F,y)
	\end{align*}
	where we used the Liminf property (\ref{PeierlsLiminf}) in the second line. This gives the wanted inequality.
\end{proof}

\begin{rem}
	These lemmas will be applied in various ways in the different sets (\ref{Sets}) that served the construction of $f_t$.
	
	The last two properties of the lemma can be applied to the connected components of $D$ or to the set $B_n \setminus \bigcup_i B_n^i$ which satisfy all the assumptions on $F$. 
	
	As an example, we present an evaluation of $h^{\rho_n\infty}(x^i_n,D)$ in dimension $d \geq 3$.  This quantity is well defined due to the Point \ref{SameClass2} of the lemma and the fact that $D$ is connected in high dimensions. Since $F = \partial B_n$ separates $x^i_n$ and $D$, we have
	\begin{align*}
		h^{\rho_n\infty}(x^i_n,D) = h^{\rho_n\infty}(x^i_n,\partial B_n) + h^{\rho_n\infty}(\partial B_n, D)
	\end{align*}
	The dynamics of $f_t$ in $A_n$ is such that for all $x \in A_n$, $\alpha_{\rho_n}(f_t(x)) \subset \partial B_n$ and $\omega_{\rho_n}(f_t(x)) \subset \partial C_n \subset D$ so that Property \ref{SameClass0} implies $h^{\rho_n\infty}(\partial B_n, D) = 0$. Thus, we get the equality
	\begin{align*}
		h^{\rho_n\infty}(x^i_n,D) = h^{\rho_n\infty}(x^i_n,\partial B_n)
	\end{align*}
\end{rem}

\section{Construction of a Non-Periodic Recurrent Viscosity Solution} \label{SectionRecurrence}

We now proceed with constructing a non-wandering viscosity solution that is not periodic, thus partially proving Theorem \ref{MainC1}. We will select a solution $u(t,x)$ such that $u(t,x_n)$ is periodic with a minimal period $\rho_n$. And since $\rho_n$ diverges to infinity, $u(t,x)$ cannot be periodic. However, the regularity conditions will not yet be met, and a suitable choice of a smooth solution will be deferred to Section \ref{RegSection}.\\

Additionally, we will examine the dynamics of $\mathcal{T}$ when restricted to the $\omega$-limit set $\omega(u)$, and more broadly to $\omega(v)$ for any non-wandering viscosity solution $v\in \Omega(\mathcal{T})$. We will observe that $\Omega(u)$ forms a Cantor set within $\mathcal{C}(M, \mathbb{R})$, where $\mathcal{T}$ behaves as an odometer. And in the general case, we will see through a proof of Theorem \ref{MainOmega} and Proposition \ref{MainPropOmega} that $\mathcal{T}_{|\omega(v)}$ is a factor of an odometer.\\

\subsection{Choice of the Non-Periodic Recurrent Initial Data} \label{SectionNCCofT}

Now that the framework has been established, our objective is to identify a suitable scalar map $u$ with a non-periodic $\omega$-limit set $\omega(u)$ under the action of the Lax-Oleinik semi-group $\mathcal{T}$. To achieve this, we start by constructing a $\rho_n$-periodic viscosity solution at every $\rho_n$-orbit $\{x^i_n\}$ of the Lagrangian flow $\phi_L$.\\

Recall from Proposition \ref{kvisc} that the barriers $h^{\rho_n\infty}(x^0_n, \cdot)$ are $\rho_n$-periodic viscosity solutions. Consequently, a fitting candidate for an initial data $u: M \to \mathbb{R}$ would be
\begin{equation} \label{ic}
	u(x) = \inf_{n \geq 0} \{ h^{\rho_n\infty}(x_n, x) \}
\end{equation}
where $x_n = x^0_n$ were defined in (\ref{Points}).\\

\begin{theo} \label{RecurrenceTheo}
	The viscosity solution with initial data $u$ defined in (\ref{ic}) is a recurrent, non-periodic viscosity solution of the Hamilton-Jacobi equation (\ref{HJ}).
\end{theo}

\begin{proof}
	The proof is segmented into two main parts. One for non-periodicity and another for recurrence. However, before delving into these steps, we need to determine the action of the Lax-Oleinik semigroup $\mathcal{T}$ on the initial data $u$ of (\ref{ic}).

\subsubsection{Evaluation of $\mathcal{T}^ku$}

We evaluate the action of the Lax-Oleinik operator $\mathcal{T}$ on the chosen scalar map $u$.

\begin{prop} \label{TkuProp}
	For all integer $k \geq 0$, we have 
	\begin{equation} \label{TkuFormula}
		\mathcal{T}^k u(x) = \inf_{n \geq 0} \{ h^{\rho_n\infty+k}(x_n, x) \}
	\end{equation}
\end{prop}

\begin{proof} 
	We use the following lemma
	\begin{lem} \label{ViscosityInf}
		Let $v_n \in \mathcal{C}(\mathbb{R} \times M)$ be a sequence of viscosity solutions and let $u(t,x) = \inf_n \{v_n(t,x) \}$. Then $u$ is a viscosity solution.
	\end{lem} 
	\begin{proof}
		Fix two real times $s<t$. Then, for all $x \in M$ the following computation holds
		\begin{align*}
			\mathcal{T}^{s,t}u(s,x) & = \inf_{y \in M} \big\{ u(s, y ) + h^{s,t}(y,x) \big\} \\
			&= \inf_{y \in M} \big\{ \inf_{n \geq 0} \{ v_n(s, y ) \} + h^{s,t}(y,x) \big\} \\
			&= \inf_{n \geq 0} \; \inf_{y \in M} \Big\{ v_n(s, y ) + h^{s,t}(y,x) \Big\} \\
			&= \inf_{n \geq 0}  \big\{ \mathcal{T}^{s,t}v_n(s,x) \big\} = \inf_{n \geq 0}  \{ v_n(t,x) \} =  u(t,x)
		\end{align*}
	\end{proof}
	We know from Proposition \ref{kvisc} that the maps $h^{\rho_n\infty}(\cdot,x_n,\cdot)$ are viscosity solutions. Hence, the lemma immediately implies that
	\begin{align*}
		\mathcal{T}^k u(x) = \inf_ {n\geq 0} \{ \mathcal{T}^k h^{\rho_n\infty}(x_n, x) \} = \inf_{n\geq 0} \{  h^{\rho_n\infty+k}(x_n, x) \}
	\end{align*}
\end{proof}

In order to determine the exact periodicity of $u$ around every orbit $\{x^i_n\}$, we need to take a closer look on the behaviour of the $\rho_n$-barrier $h^{\rho_n\infty}$ for the studied \Mane Lagrangian $L$.

\begin{prop} \label{hper}
	For all integers $n \geq 0$, the map $h^{\rho_n\infty}(x_n, \cdot)$ is a periodic viscosity solution with minimal period $\rho_n$.
\end{prop}

\begin{proof}
	The Proposition \ref{kvisc} already tells that the maps $h^{\rho_n\infty}(x_n, \cdot)$ are $\rho_n$-periodic viscosity solutions of the Hamilton-Jacobi equation (\ref{HJ}). However, the minimality of the periods $\rho_n$ requires extra effort. We claim that
	\begin{enumerate}[label=(\roman*)]
		\item $h^{\rho_n\infty}(x_n,x_n) =0$. \label{=0}
		\item For all $k \in \{1,..,\rho_n-1\}$, $\mathcal{T}^k h^{\rho_n\infty}(x_n,x_n) = h^{\rho_n\infty +k}(x_n,x_n) \not=0$. \label{n=0}
	\end{enumerate}	
	
	\paragraph{\ref{=0}} For this case, One needs to follow the flow $f_t$ starting at $x_n$. Let $\gamma : \mathbb{R} \to M$ be the curve $\gamma(t) = f_t(x_n)$ which, by construction of the point $x_n$, is $\rho_n$-periodic. We have by definition of the $\rho_n$-barrier $h^{\rho_n\infty}$ that
	\begin{align*}
		h^{\rho_n\infty}(x_n,x_n) = \liminf_k h^{\rho_nk}(x_n,x_n) &\leq \liminf_k A_L( \gamma_{|[0,\rho_nk]}) = \liminf_{k} \int_0^{\rho_nk} L(\tau,\gamma(\tau), \dot{\gamma}(\tau)) \; d\tau \\ 
		&= \liminf_{k} \int_0^{\rho_nk} \frac{1}{2}  \Vert \dot{\gamma}(\tau) -X_\tau(\gamma(\tau))\Vert^2 d\tau \\
		&=  \liminf_{k} \int_0^{\rho_nk} \frac{1}{2}  \Vert \partial_\tau f_\tau(x_n) -X_\tau(f_\tau(x_n))\Vert^2 d\tau = 0
	\end{align*}
	where the last nullity is due to the fact that $f_t$ is the flow of $X_t$. Combining this inequality with the non-negativity of the Mañé Lagrangian barriers noted in Corollary \ref{hpos}, we deduce that $h^{\rho_n\infty}(x_n,x_n) = 0$.
	
	\paragraph{\ref{n=0}} Let $k \in \{1,..,\rho_n-1\}$. It suffices to see that the kind of chain transitivity claimed by Corollary \ref{gchrec} does not hold in this case. We act by contradiction and suppose that $h^{\rho_n\infty +k}(x_n,x_n)=0$. 
	
	 Recall that the ball $B^k_n$ centered at $x^k_n$ with radius $\delta_n$ is the basin of attraction of $x^k_n$ under the map $f_{\rho_n}$. Fix a radius $0<\delta< \frac{\delta_n}{2}$. We know that $f_{\rho_n}$ sends the open ball $B_\delta = B(x^k_n, \delta)$ compactly into itself so that $f_{\rho_n} (\overline{B_\delta}) \subset B_\delta$. Let $\varepsilon$ be defined as $\varepsilon = d\big(f_{\rho_n} (\overline{B_\delta}), \partial B_\delta\big)$. Then we have $\varepsilon < \delta$ and $f_{\rho_n} (B_\delta) \subset B_{\delta- \varepsilon}$.
	
	Since we assumed $h^{\rho_n\infty +k}(x_n,x_n)=0$, we can now apply the Corollary \ref{gchrec} to this $\varepsilon$. We use the same notations for the obtained curves. Given that $\gamma_0(t) = f_t(x_n)$, it follows that $\gamma_0(k) = x^k_n$. This enables to localize $\gamma_1(0)$ as below 
	\begin{equation*}
		\gamma_1(0) \in B(\gamma_0(k), \varepsilon) = B(x^k_n, \varepsilon) \subset B(x^k_n, \delta)= B_\delta
	\end{equation*}
	Using the definition of $\varepsilon$ and the fact $T_1$ is a multiple of $\rho_n$, we deduce that $\gamma_1(T_1)$ belongs to $f_{\rho_n} (B_\delta) \subset B_{\delta- \varepsilon}$. And as a result, the localization of the next initial point $\gamma_2(0)$ gives $\gamma_2(0) \in B_\delta$.
	
	A simple induction leads to the inclusion $x_n \in B_\delta \subset B^k_n = B_{\delta_n}$ which contradicts the fact the $x^k_n$-centred ball $B^k_n$ and the $x_n$-centred ball $B^0_n$ are disjoint.
\end{proof}

\begin{rem} \label{hstrictpos}
	Note that the proof of point \ref{n=0} yields the following result : For every integers $n>0$ and $0\leq k < \rho_n$, and every point $x \neq x^k_n$, we have $h^{\rho_n \infty +k}(x_n, x) > 0$.
	
	This proof provides the necessary details to verify the hypothesis of the contrapositive version of Proposition \ref{chrec0}, which establishes the strict positivity of $\liminf_{t \to 0} h^t(x,y) >0$ or $\liminf_{n \to 0} h^{t_n}(x,y) >0$.
	
	From now on, we will omit such detailed explanations and only assert the non-existence of pseudo-orbits linking two studied point based on dynamics of the flow $f_t$ between these.
\end{rem}

\subsubsection{Non-Periodicity of $u$.} \label{NonPeriodicitySection}

	We proceed with the proof of the non-periodicity of $u$. Arguing by contradiction, suppose that there exists a positive integer $q>0$ such that $\mathcal{T}^qu = u$. Given that the sequence of periods $\rho_n$ diverges to infinity, we can fix an integer $n \geq 0$ such that $\rho_n > q$. We claim that 
	\begin{enumerate}[label=(\roman*)]
		\item For $k=0,\rho_n$, $\mathcal{T}^k u(x_n) = 0$. \label{==0}
		\item For $k=1,..,\rho_n-1$, $\mathcal{T}^k u(x_n) \not= 0$. \label{not=0}
	\end{enumerate}
	which contradicts the $q$-periodicity of $(\mathcal{T}^ku(x_n))_{k \geq 0}$ with $q < \rho_n$.\\
	
	The point \ref{==0} is a consequence of Proposition \ref{TkuProp} combined with the proof of Proposition \ref{hper}. When gathered, we get for $k=0,\rho_n$
	\begin{equation*}
		0 \leq \mathcal{T}^ku(x_n) \leq h^{\rho_n \infty+k}(x_n,x_n) = h^{\rho_n \infty}(x_n,x_n) = 0
	\end{equation*}		
		
	The second point \ref{not=0} is more subtle.  Fix an integer $1 \leq k < \rho_n$. We know from the proof of Proposition \ref{hper} that $h^{\rho_n \infty+k}(x_n,x_n) > 0$. However, we need a uniform positive lower bound on $h^{\rho_m \infty +k}(x_m,x_n)$ for $m \neq n$ in order to deduce that $\mathcal{T}^k u(x_n) = \inf_{m \geq 0} \{ h^{\rho_m\infty+k}(x_m, x_n) \} > 0$.
	
	Fix an integer $m \neq n$. We use the dynamics of the flow $f_t$. We know from the construction that the flow is directed from $\partial B_n$ towards $\partial C_n$ (see Figure \ref{FigureConstruction}). We will use this fact to apply the contrapositive part of Proposition \ref{chrec0}. 
	
	But before that, let us take a deeper look on the minimizing curves that realize $h^{\rho_m\infty+k}(x_m, x_n)$. There exists a sequence of minimizing curves $\gamma_i : [0, \rho_m q_i+k] \to M$ from $x_m$ to $x_n$ such that $h^{\rho_m\infty+k}(x_m, x_n) = \lim_i A_L(\gamma_i)$. Since $\partial C_n$ and $\partial B_n$ separate successively $x_m$ from $x_n$, there exist two real times $0 \leq s_i < t_i \leq \rho_nq_i+k$ such that $\gamma_i(s_i) \in \partial C_n$ and $\gamma_i(t_i) \in \partial B_n$. However, we know from the à priori compactness Theorem \ref{APrioriCompactness} that there is a uniform Lipschitz value $M>0$ for all minimizing curves. Hence, we have
	\begin{align*}
		d(\partial C_n, \partial B_n) \leq d(\gamma_i(s_i), \gamma_i(t_i)) \leq  M.(t_i-s_i)
	\end{align*}
	
	Set $\tau = \frac{d(\partial C_n, \partial B_n)}{M}$ which is independent of $m$. The following holds.
	\begin{lem}
		There exists a real number $\varepsilon_n >0$ depending only on $\tau$ such that 
		\begin{equation}
			\inf \{h^{s,t}(x,y)\; | \; 0\leq s < t, \; t-s \geq \tau, \; x \in \partial C_n, \; y \in \partial B_n \} \geq \varepsilon_n
		\end{equation}
	\end{lem}	
	\begin{proof}
		We argue by contradiction. Assume that there exists a sequence of curves $\sigma_i : [s_i,t_i] \to M$ such that  $t_i - s_i \geq \tau$, $x_i:= \sigma_i(s_i) \in \partial C_n$, $y_i:= \sigma_i(t_i) \in \partial B_n$ and $h^{s_i,t_i}(x_i,y_i) = A_L(\sigma_i) \to 0$ as $i \to +\infty$. The time periodicity of the Lagrangian allows, up to time translation, to assume that $s_i \in \mathbb{T}^1$ for all $i \geq 0$. And by compactness of $\mathbb{T}^1 \times \partial C_n \times \partial B_n$,we can further assume that, up to extraction, $(s_i,x_i,y_i)$ converges to $(s,x,y) \in \mathbb{T}^1 \times \partial C_n \times \partial B_n$. 
		
		If the times $t_i$ are bounded, we can also assume that $t_i$ converges to $t \geq s+ \tau$. The continuity of the potential $h$ results in the equality
		\begin{equation*}
			h^{s,t}(x,y) = \lim\limits_i h^{s_i,t_i}(x_i,y_i) = 0
		\end{equation*}
		This gives rise to a minimizing curve $\gamma : [s,t] \to M$ linking $x$ to $y$ such that
		\begin{equation}
			A_L(\gamma) = \int_s^t \frac{1}{2} \Vert\dot{\gamma}(\zeta) - X_{\zeta}(\gamma(\zeta) \Vert^2 d\zeta =0
		\end{equation}
		Thus, $\dot{\gamma}(\zeta) = X_\zeta(\gamma(\zeta))$ and the curve follows the flow $f_t$. However, we know from the construction that the flow $f_t$ is directed from $\partial B_n$ toward $\partial C_n$, which contradicts the existence of such a curve $\gamma$.

		We showed that the times $t_i$ are unbounded, diverging to $+ \infty$. Using the lipschitz-regularity of the potentials $h$ stated in Proposition \ref{Regularity}, We have for all integer $i \geq 0$
		\begin{equation*}
			| h^{s,t_i}(x,y) - h^{s_i,t_i}(x_i,y_i) | \leq \kappa_\tau. d\big((s,x,y),(s_i,x_i,y_i) \big) \to 0 \quad \text{as } i \to \infty
		\end{equation*}
		Hence, we get $\liminf_i h^{s,t_i}(x,y) = 0$, which contradicts the contrapositive claim of Proposition \ref{chrec0}.
	\end{proof}
	The lemma above established that $A_L(\gamma_{i|[s_i,t_i]}) \geq \varepsilon_n$, providing the uniform lower bound
	\begin{align*}
		h^{\rho_m\infty+k}(x_m, x_n) = \lim_i A_L(\gamma_i) \geq \liminf_i A_L(\gamma_{i|[s_i,t_i]}) \geq \varepsilon_n
	\end{align*}
	Therefore
	\begin{equation}
		\mathcal{T}^k(u)(x_n) \geq \inf_{m \geq 0} \{ h^{\rho_m\infty+k}(x_m, x_n) \} \geq \min\{ \varepsilon_n, h^{\rho_n\infty+k}(x_n,x_n) \} >0
	\end{equation}
	This concludes the proof of point \ref{not=0} and the proof of the non-periodicity of $u$.
	
	\begin{rem}
		The proof for point \ref{not=0} could have been simplified by exploiting the autonomous nature of the radial flow $g_t$ from $\partial B_n$ to $\partial C_n$. We made a deliberate choice not to take advantage of this feature.
	\end{rem}

\subsubsection{Recurrence of $u$} \label{RecurrenceSection}
	Set for all $n \geq 0$, the integer $p_n = \prod\limits_{k=0}^n\rho_k$. The $\rho_n$-periodicity of the $\rho_n$-barriers $h^{\rho_n \infty}$ tells that
	\begin{equation} \label{Tpku}
		\begin{split}
			\mathcal{T}^{p_k}u(x) & = \inf_{n \geq 0} \{ h^{\rho_n\infty+p_k}(x_n, x) \}\\
			&= \min \big\{ \inf_{0\leq n \leq k} \{h^{\rho_n\infty}(x_n, x)\} , \inf_{n > k} \{h^{\rho_n\infty + p_k}(x_n, x)\} \big\}
		\end{split}
	\end{equation}
	We see that the $k$ first elements of the infimum defining $\mathcal{T}^{p_k}u$ coincide with those in the infimum defining $u$. Our expectation is that the difference between these two maps diminishes as $k$ approaches infinity.\\
	
	We need to compare the two maps $h^{\rho_n\infty}(x_n, x)$ and $h^{\rho_n\infty+i}(x_n, x)$ for some $0 \leq i < \rho_n$. A lemma in this direction is the following

	\begin{lem} \label{Tk+iu}
		For all integers $n \geq 0$ and $0 \leq i < \rho_n$, 	
		\begin{equation}
			h^{\rho_n\infty+i}(x_n, x) = h^{\rho_n\infty}(x^i_n, x)
		\end{equation}
	\end{lem}
	\begin{proof}
		We know that $h^{i}(x_n, x^i_n)=0$ and $h^{\rho_n-i}(x^i_n, x_n)=0$ as they are respectively realized by the curves $f_t(x_n)$ and $f_t(x^i_n)$. Hence, applying the triangular inequality (\ref{TriangIneg}), we obtain
		\begin{multline*}
			h^{\rho_n\infty+i}(x_n, x) = \liminf_j h^{\rho_nj+i}(x_n, x) \\ 
			\leq \liminf_j \big[ h^{i}(x_n, x^i_n)+ h^{\rho_nj}(x^i_n, x) \big] = \liminf_j h^{\rho_nj}(x^i_n, x) =  h^{\rho_n\infty}(x^i_n, x)
		\end{multline*}
		and
		\begin{multline*}
			h^{\rho_n\infty}(x^i_n, x) = h^{\rho_n\infty+\rho_n}(x^i_n, x) = \liminf_j h^{\rho_nj+\rho_n}(x^i_n, x) \\ 
			\leq \liminf_j \big[ h^{\rho_n-i}(x^i_n, x_n)+ h^{\rho_nj+i}(x_n, x)\big]  = \liminf_j h^{\rho_nj+i}(x_n, x) =  h^{\rho_n\infty+i}(x_n, x)
		\end{multline*}
	\end{proof}

	We proceed to the comparison. Recall from the definition of the $n$-Peierls barriers and from Proposition \ref{PeierlsProp} that $h^{\rho_n\infty}$ is $\kappa_1$-Lipschitz. Thus, we get for all integer $n>k$ and all point $x$ in $M$,
	\begin{align*}
		|h^{\rho_n\infty+p_k}(x_n, x) - h^{\rho_n\infty}(x_n, x)| & = |h^{\rho_n\infty}(x^{p_k}_n, x) - h^{\rho_n\infty}(x_n, x)| \\
		& \leq 2\kappa_1.d(x^{p_k}_n,x_n) \leq 2\kappa_1. 2r_n \leq 4\kappa_1. r_k =: \varepsilon_k
	\end{align*}
	Since we have equality $h^{\rho_n\infty + p_k} = h^{\rho_n\infty}$ for all $n\leq k$, it follows that for all integer $n \geq 0$ and all $x \in M$
	\begin{equation}
		|h^{\rho_n\infty+p_k}(x_n, x) - h^{\rho_n\infty}(x_n, x)| \leq \varepsilon_k
	\end{equation}		
	In particular, 
	\begin{equation*}
		h^{\rho_n\infty}(x_n, x) - \varepsilon_k  \leq h^{\rho_n\infty+p_k}(x_n, x) \leq h^{\rho_n\infty}(x_n, x) + \varepsilon_k 
	\end{equation*}
	Taking the infimum over $n$ and using (\ref{Tpku}), we get
	\begin{equation*}
		u(x) -\varepsilon_k \leq \mathcal{T}^{p_k}u(x) \leq u(x) +\varepsilon_k
	\end{equation*}
	or more precisely
	\begin{equation*}
		\Vert\mathcal{T}^{p_k}u(x) - u(x) \Vert_\infty \leq \varepsilon_k \to 0 \quad \text{as } k \to +\infty
	\end{equation*}
	This means that $\mathcal{T}^{p_k}u$ converges to $u$ in $\mathcal{C}(M, \mathbb{R})$, indicating that $u$ belongs to its own $\omega$-limit set $\omega(u)$ and is a recurrent viscosity solution. 
\end{proof}

\begin{rem} Some remarks are to be made on this result.
	\begin{enumerate}
		\item As mentioned in the introduction, this Theorem \ref{RecurrenceTheo} implies that, at the difference of non-autonomous \cite{MR1650261} or the one-dimensional case \cite{MR2041603}, there is no convergence of the Lax-Oleinik operator to a periodic orbit.
		\item Recall from Corollary \ref{Minimality} and Remark \ref{rect} the Minimality of the Lax-Oleinik operator $\mathcal{T}$ on $\omega(u)$. This means that all the elements $v$ of $\omega(u)$ are non-periodic, recurrent viscosity solutions. A deeper study of the $\omega$-limit set of the constructed $u$ will be done in the next Subsection \ref{OmegaUSection}.
		\item The arguments of the proof do work for a large variety of examples. It seems uncomplicated to construct various Hamiltonians that admit non-periodic recurrent viscosity solutions. An example where the Mather set contains a Cantor set would be as follows. $f = \tau \circ g$ where $g$ is represented in black arrows in the Figure \ref{FigureLemniscate} below and $\tau$ is the map that exchanges the interior components of the lemniscates and acts as an adding machine (See Subsection \ref{OmegaUSection} for a definition) on the Cantor set formed by the intersection of these different components. The latter map is represented in pink in Figure \ref{FigureLemniscate}.
		\begin{figure}[h!]
			\centering
			\includegraphics[width=0.8\linewidth]{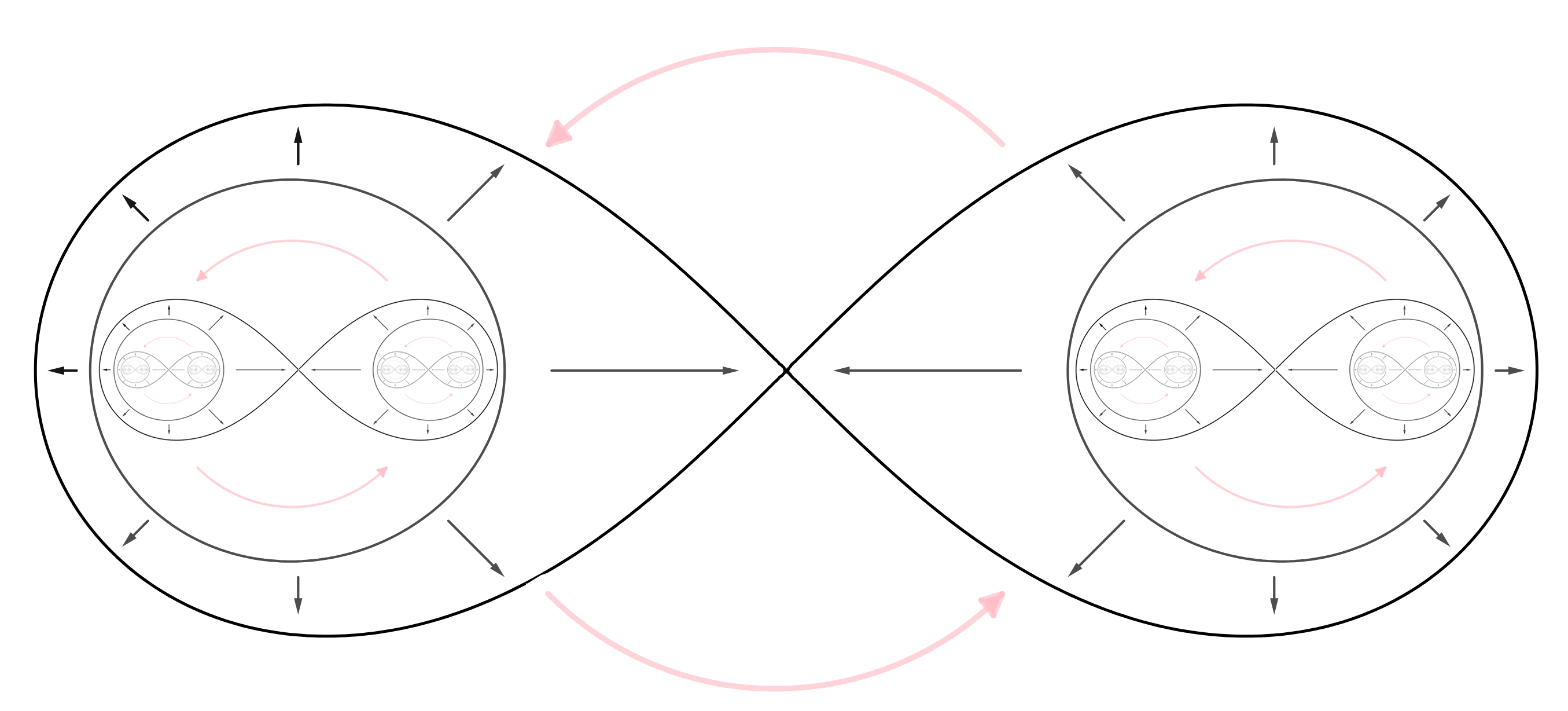}
			\caption{Dynamics of $f$.}
			\label{FigureLemniscate}
		\end{figure}
	\end{enumerate}
\end{rem}

\subsection{Action of the Lax-Oleinik Operator $\mathcal{T}$ on the Non-Wandering Set $\Omega(\mathcal{T})$} \label{OmegaUSection}

In this subsection, we outline the construction of all recurrent viscosity solutions associated with the Mané Lagrangian $L$ defined in (\ref{LagMane}). Furthermore, we provide insights into the dynamics of $\mathcal{T}$ within their $\omega$-limit sets.\\

To begin, we focus on examining the $\omega$-limit set $\omega(u)$ of the previously established recurrent viscosity solution $u$ defined in (\ref{ic}). There are slight topological differences between the cases where the manifold $M$ is of dimension $d=2$ or $d \geq 3$. And since the proofs are analogous, we will address only the high-dimensional case, noting the differences with the 2D case in the remarks.\\

\subsubsection{The $\omega$-Limit Set $\omega(u)$ and Adding Machines} \label{SectionAddingMachine}

First and foremost, it is essential to establish a means of identifying each element within the $\omega$-limit set $\omega(u)$. To do so, we start by giving a convenient expression of $\mathcal{T}^ku$.

\begin{prop}\label{TkuOmegaProp}
	If $d \geq 0$, For all integers $k \in \mathbb{N}$, we have
	\begin{equation*}
		\mathcal{T}^ku(x) = 
		\begin{cases}
			\min \big( h^\infty(z_\infty,x), h^{\rho_n\infty}(x^k_n, x) \big) & \text{if } x \in C_n \\
			0 & \text{if } x \in \overline{D}
		\end{cases}
	\end{equation*}
\end{prop}

\begin{proof}
	Fix an integer $n\in \mathbb{N}$. Let $u^k_n$ be the map defined by
	\begin{equation}
		u^k_n(x) = \inf_{m \neq n} \{ h^{\rho_m\infty}(x^k_m,x) \}
	\end{equation}
	We know from Proposition \ref{TkuProp} and Lemma \ref{Tk+iu} that
	\begin{equation} \label{TknuOmega}
		\begin{split}
			\mathcal{T}^ku(x) &= \inf_{m \geq 0}\{h^{\rho_m\infty+k}(x_m,x)\} = \inf_{m \geq 0}\{h^{\rho_m\infty}(x^k_m,x)\} \\
			&= \min \big( u_n^k(x), h^{\rho_n\infty}(x^k_n, x) \big)
		\end{split}
	\end{equation}
	
	We show that for all $x \in C_n$, $u^k_n(x) =  h^\infty(z_\infty,x)$. Let $x$ be a fixed point of $C_n$ and fix an integer $m \neq n$. Let us evaluate $h^{\rho_m\infty}(x^k_m,x)$. The set $\overline{D'}$ separates $x$ and $x^k_m$. Then, applying Property \ref{SeparateClass} of Lemma \ref{SameClass} applied to $F = \overline{D'}$ followed by an application Property \ref{SameClass1} to the $1$-periodic point $z_\infty$, yields
	\begin{equation} \label{TknuOmegaDem1}
		\begin{split}
			h^{\rho_m\infty}(x^k_m, x) &= h^{\rho_m \infty} (x^k_m,D') +  h^{\rho_m \infty} (D',x) \\
			&=h^{\rho_m \infty} (x^k_m,z_\infty) +  h^{\rho_m \infty} (z_\infty,x) \\
			&= h^{\rho_m \infty} (x^k_m,z_\infty) +  h^{\infty} (z_\infty,x)
		\end{split}
	\end{equation}
	Moreover, the regularity of the barriers, as stated in Proposition \ref{pPeierlsProp}, added with the fact that $h^{\rho_m\infty}(x^k_m,x^k_m) =0$, provide
	\begin{align*}
		0 \leq h^{\rho_m\infty}(x^k_m,z_\infty) & = h^{\rho_m\infty}(x^k_m,z_\infty) - h^{\rho_m\infty}(x^k_m,x^k_m) \leq \kappa_1. d(z_\infty, x^k_m) \leq \kappa_1. r_m \to 0 \quad \text{as } m \to \infty
	\end{align*}
	Thus, $\inf_{m \neq n} \{ h^{\rho_m\infty}(x^k_m,z_\infty) \} =0$ and using \eqref{TknuOmegaDem1}, we obtain
	\begin{align*}
		u_n^k(x)&= \inf_{m \neq n} \{h^{\rho_m\infty}(x^k_m,x) \} \\
		&=\inf_{m \neq n} \{h^{\rho_m\infty}(x^k_m,z_\infty) +h^{\infty}(z_\infty,x) \} \\
		&= \inf_{m \neq n} \{ h^{\rho_m\infty}(x^k_m,z_\infty) \} +h^{\infty}(z_\infty,x) \\
		&= h^{\infty}(z_\infty,x)
	\end{align*}
	
	Furthermore, the Property \ref{SameClass2} of Lemma \ref{SameClass} shows that for all $x \in \overline{D}$ and all integer $m$, $h^{\rho_m\infty}(x^k_m,x) = h^{\rho_m\infty}(x^k_m,z_\infty)$ which, in the same fashion as above, yields for all $x \in \overline{D}$, $u^k_n(x) = h^\infty(z_\infty,x) = 0$.
	We showed that 
	\begin{align*}
		u^k_n(x) = 
		\begin{cases}
			h^\infty(z_\infty,x) & \text{if } x \in C_n \\
			0 & \text{if } x \in \overline{D}
		\end{cases}
	\end{align*}
	This, added to identity (\ref{TknuOmega}), concludes the proof.
\end{proof}

\begin{rem}
	In the two dimensional case, the proposition doesn't hold. However, by the same arguments, we still have for all $x \in C_n$ that $u_n^k(x) = u_n^0(x)$, and if we consider the sets 
	\begin{equation}
		\begin{split}
			D'_0 &:= \{r \geq r_{1} + 2 \delta_{1}\} \supset C_0 \\
			D'_n &:= \{ r_{n+1} + 2 \delta_{n+1} \leq r \leq r_{n-1} - 2 \delta_{n-1} \} \supset C_n \quad \text{if } n \geq 1
		\end{split}
	\end{equation}
	then 
	\begin{equation} \label{Tku2Dcase}
		\mathcal{T}^ku(x) = \min \big(u^0_n(x), h^{\rho_n\infty}(x^k_n, x) \big) \quad \text{if } x \in D'_n
	\end{equation}
\end{rem}   

The subsequent proposition allows the desired identification.

\begin{prop} \label{UniquenessSequence}
	Every element $v$ of $\omega(u)$ is characterized by a unique infinite sequence $\underline{k}(v)=(k_n(v))_{n \geq 0}$ in $Z_{\underline{\rho}}:= \prod_{n=0}^\infty \mathbb{Z}/\rho_n\mathbb{Z}$ such that
	\begin{equation}
		v(x) = \inf_{n \geq 0} \{ h^{\rho_n\infty+k_n(v)}(x_n,x) \}
	\end{equation}
\end{prop}

\begin{proof}
	Proposition \ref{TkuOmegaProp} show that 
	\begin{align*}
		\mathcal{T}^ku(x) &= \inf_{m \geq 0}\{h^{\rho_m\infty+k}(x_m,x)\} \\
		&= \begin{cases}
			\min \big( h^\infty(z_\infty,x), h^{\rho_n\infty}(x^k_n, x) \big) & \text{if } x \in C_n \\
			0 & \text{if } x \in D
		\end{cases}
	\end{align*}
	Given that the finiteness of the sets $\{x_n^k\}_{k \geq 0}$, it follows that for any $v \in \omega(u)$ and for all $n \geq 0$, there must exist an integer $k_n(v) \in \{0, 1, \ldots, \rho_n-1\}$ such that
	\begin{equation}
		v(x)= 
		\begin{cases}
			\min \big( h^\infty(z_\infty,x), h^{\rho_n\infty}(x^{k_n(v)}_n, x) \big) & \text{if } x \in C_n \\
			0 & \text{if } x \in D
		\end{cases} 
	\end{equation}
	Let $\tilde{v} : M \to \mathbb{R}$ be the map defined by
	\begin{equation}
		\tilde{v}(x) = \inf_{n \geq 0} \{ h^{\rho_n\infty+k_n(v)}(x_n, x) \}
	\end{equation}
	The exact same proof as for Proposition \ref{TkuOmegaProp} applied to $\tilde{v}$ shows that
	\begin{align*}
		\tilde{v}(x) &= \inf_{n \geq 0} \{ h^{\rho_n\infty+k_n(v)}(x_n, x) \} \\
		&= \begin{cases}
			\min \big( h^\infty(z_\infty,x), h^{\rho_n\infty}(x^{k_n(v)}_n, x) \big) & \text{if } x \in C_n \\
			0 & \text{if } x \in D
		\end{cases} \\
		&= v(x)
	\end{align*}
	
	\textit{Unicity of $k_n(v)$.} An application of the quantitative Property \ref{SameClass0} of Lemma \ref{SameClass} gives for all $i \in \{0,..,\rho_n-1\}$, 
	\begin{align*}
		h^\infty (z_ \infty , x^i_n) = h^\infty (z_\infty , y_n) > 0 \quad \text{and} \quad h^{\rho_n\infty}(x^{k_n(v)}_n, x^i_n)=
		\begin{cases}
			h^\infty(x^{k_n(v)}_n,y_n)>0 & \text{if } i \neq k_n(v) \\
			0 & \text{if } i = k_n(v)
		\end{cases}
	\end{align*}
	This implies that $v(x^i_n) =0$ if and only if $i =k_n(v)$. In other words, for all integer $n \geq 0$, $k_n(v)$ is the index $i$ of the unique point $x^i_n$ such that $v(x^i_n) =0$. This results in the uniqueness of $k_n(x)$.
\end{proof} 

\begin{rem}
	\begin{enumerate}
		\item In the proof, it is evident that each element $v$ in $\omega(u)$ is uniquely determined by the images of the points $\{x^i_n \; ; \; 0 \leq i < \rho_n , n \geq 0 \} \subset \mathcal{M}_0$. This follows from a broader uniqueness theorem proved in proposition 3.2 of \cite{MR2041603}, which asserts that all recurrent viscosity solutions are uniquely characterized by their images on the projected Mather set $\mathcal{M}_0$.
		\item The identity \eqref{Tku2Dcase} suffices to prove the proposition in the 2D case.
	\end{enumerate}
\end{rem}

We have shifted the analysis of $\omega(u)$ from the space $\mathcal{C}(M,\mathbb{R})$ to $Z_{\underline{\rho}}=\prod_{n=0}^\infty \mathbb{Z}/\rho_n\mathbb{Z}$. This enables to compare the dynamics of $\mathcal{T}_{|\omega(u)}$ to other dynamical systems on the new space $Z_{\underline{\rho}}$.\\

Let us define a topology in the space $Z_{\underline{\rho}}= \prod_{n=0}^\infty \mathbb{Z}/\rho_n\mathbb{Z}$ by endowing it with a metric $d$ defined as follows
\begin{equation} \label{TopologyCoeff}
	d\left(\underline{q}, \underline{p}\right) = \frac{1}{2^{\nu \left(\underline{q}, \underline{p}\right)}} \quad \text{with} \quad \nu \left(\underline{q}, \underline{p}\right) = \min \{n \geq 0 \; | \; q_n \neq p_n \}
\end{equation}

\begin{defi}
	The \textit{odometer map} also called the \textit{adding machine} on $Z_{\underline{\rho}}= \prod_{n=0}^\infty \mathbb{Z}/\rho_n\mathbb{Z}$ is the homeomorphism $\tau$ defined by
	\begin{equation}
		\begin{matrix}
			\tau : & \prod\limits_{n=0}^\infty \mathbb{Z}/\rho_n\mathbb{Z} &\longrightarrow &\prod\limits_{n=0}^\infty \mathbb{Z}/\rho_n\mathbb{Z}\\ 
			&\underline{q}=(q_2,q_3,...) &\longmapsto &\underline{q} + \underline{1} = (q_2 + 1 ,q_3 + 1 ,...)
		\end{matrix}
	\end{equation}
\end{defi}

\begin{rem} \label{OdometersRem1} 
	For more on adding machines, a standard reference is \cite{MR0551496}. A survey from a dynamical perspective is available in \cite{MR2180227}. And for a recent and concise introduction, we refer to \cite{MR4492832}.
	\begin{enumerate}
		\item In the standard definition, adding machines are minimal maps on their domains which is not necessarily the case of the map $\tau$. However, if $\underline{q}$ is a non-periodic point of $\tau$, then the restriction of $\tau$ to the $\omega$-limit set $\omega\left(\underline{q}\right)$ behaves as an adding machine in the usual sense.
	
	In this paper, we chose to include periodic maps in what we refer as odometers or adding machines.  
		\item Note that the odometer $\tau$ is an isometry of the space $Z_{\underline{\rho}}$. Indeed, if the first $n$-terms of the sequences $\underline{q}$ and $\underline{p}$ coincide, the same holds for $\underline{q}+1$ and $\underline{p}+1$ so that $d(\tau(\underline{q}), \tau(\underline{p}))= d(\underline{q}, \underline{p})$.
	\end{enumerate}
\end{rem}

We now state the theorem that completes the understanding of the asymptotic behaviour of the constructed recurrent viscosity solution $u$ introduced in (\ref{ic}).

\begin{theo} \label{OmegaOdometer}
	The restriction of the Lax-Oleinik semigroup $\mathcal{T}$ to the $\omega$-limit set $\omega(u)$ is a factor of the odometer map $\tau$. More precisely, there exists a continuous injective map $\varphi : \omega(u) \hookrightarrow Z_{\underline{\rho}}$ such that the following commutative diagram holds
	\begin{equation}
		\begin{tikzcd}
			\omega(u) \arrow[hookrightarrow]{d}{\varphi} \arrow{r}{\mathcal{T}}  & \omega(u) \arrow[hookrightarrow]{d}{\varphi} \\
			Z_{\underline{\rho}} \arrow{r}{\tau}  & Z_{\underline{\rho}}
		\end{tikzcd}	
	\end{equation}
\end{theo}

\begin{proof}
	The choice of the map $\varphi$ comes from the unicity Proposition \ref{UniquenessSequence}. For all $v$ in $\omega(u)$, we set $\varphi(v) = \underline{k}(v)$. Then the injectivity of the map $\varphi$ follows immediately. We only need to prove that it is continuous.
	
	For all $n \geq 0$, we set $v_n^i : M \to \mathbb{R}$ to be the map
	\begin{equation*}
		v^i_n(x) = \min \big( h^\infty(z_\infty,x), h^{\rho_n\infty}(x^i_n, x) \big)
	\end{equation*}
	and the constants $\lambda_n$ to be defined as
	\begin{equation*}
		\lambda_n = \min_{0 \leq i \neq j < \rho_n} \left\{ \left\Vert v_{n|C_n}^i - v_{n|C_n}^j \right\Vert_\infty \right\}
	\end{equation*}
	We have $\lambda_n >0$. Indeed, if we take an integer $0 \leq i < \rho_n$, then we know from Proposition \ref{chrec0} that
	\begin{align*}
		v_n^i(x_n^i)  = h^{\rho_n\infty}(x^i_n, x^i_n) = 0 \; , \quad \quad h^\infty(z_\infty,x^i_n) > 0 \; , \quad \text{and} \quad  \forall j \neq i, \; h^{\rho_n\infty}(x^j_n, x^i_n)>0
	\end{align*}
	This yields $\big\Vert v_{n|C_n}^i - v_{n|C_n}^j \big\Vert_\infty >0$, and taking the infimum on $i\neq j$, we obtain the positivity $\lambda_n >0$.

	We can show that the sequence $\lambda_n$ is decreasing. But instead of proving it, we consider the positive non-increasing sequence of $\tilde{\lambda}_n$ defined by
	\begin{equation*}
		\tilde{\lambda}_n = \min_{2 \leq i \leq n} \lambda_i
	\end{equation*}
	Let $v$ and $w$ be two elements of $\omega(u)$ such that $\Vert v - w \Vert_\infty < \tilde{\lambda}_n$. We know from Proposition \ref{UniquenessSequence} that for all $n \geq 0$, 
	
	\begin{equation*}
		v_{|C_n} = v_n^{k_n(v)} \quad \text{and} \quad w_{|C_n} = v_n^{k_n(w)}
	\end{equation*}		
	Hence, from the definition of $\tilde{\lambda}_n$ and the hypothesis $\Vert v - w \Vert_\infty < \tilde{\lambda}_n$, we deduce that for all $2 \leq i \leq n$, $k_i(v) = k_i(w)$ which means that
	\begin{equation*}
		d(\varphi(v), \varphi(w)) \leq \frac{1}{2^{n+1}}
	\end{equation*}
	This proves the (uniform) continuity of $\varphi$ and concludes the proof of the theorem.
\end{proof}

\begin{cor} \label{Cantor}
	The $\omega$-limit set $\omega(u)$ is a Cantor space.
\end{cor}

\begin{proof}
	The topology induced by the metric $d$ on the space $\prod_{n=0}^\infty \mathbb{Z}/\rho_n\mathbb{Z}$ is the cylinder topology which is totally disconnected. The Theorem \ref{OmegaOdometer} asserts that $\omega(u)$ is embedded in $\prod_{n=0}^\infty \mathbb{Z}/\rho_n\mathbb{Z}$. Thus, it is totally disconnected. Moreover, we know that $\omega(u)$ is a compact metric space and Corollary \ref{Minimality} and Remark \ref{rect} indicate that it has no periodic points and hence no isolated points since every element is neared by its own orbit. These properties collectively imply that $\omega(u)$ is homeomorphic to the Cantor set.
\end{proof}

It is possible to provide a precise description of the minimal odometer $\tau : \varphi(\omega(u)) \to \varphi(\omega(u))$. The classification of these (minimal) odometers is a well-known subject, and we will make use of a classification theorem to be stated below. But first, let us introduce a more classical definition of odometers.

\begin{defi}
	For all sequence of integers $\underline{k} = (k_n)_{n \geq 0}$, we define the \textit{$\underline{k}$-odometer} $\tau_{\underline{k}} : Z_{\underline{k}} \to Z_{\underline{k}}$, defined on the space $Z_{\underline{k}}= \prod_{n=0}^\infty \mathbb{Z}/k_n\mathbb{Z}$ endowed with the metric (\ref{TopologyCoeff}), as follows 
	\begin{equation}
		\tau_{\underline{k}} (\underline{q}) = 
		\begin{cases}
			(\smash[b]{\underbrace{0,0,\cdots,0\,}_\text{$n$ times}}, q_n+1, q_{n+1}, q_{n+2}, ...) & \text{if } n= \min \{ m \geq 0 \; ; \; q_{m} \neq k_m-1 \} < +\infty\\
			\\
			\underline{0} = (0,0,0,...) & \text{if } n= + \infty
		\end{cases}
	\end{equation}
\end{defi}

\begin{rem}
	These are the classical maps that we refer to as odometers. Examining their behaviour, it becomes clearer to understand why they are named adding machines. As mentioned in Remark \ref{OdometersRem1}, they are transitive and even minimal. And there is an odometer for every positive sequence $\underline{k}$.
\end{rem}

\begin{defi}
	For all sequence of integers $\underline{k} = (k_n)_{n \geq 0}$, we define its \textit{multiplicity function} $m_{\underline{k}}$ by
	\begin{equation}
		\begin{matrix}
			m_{\underline{k}} : & \mathbb{P} & \longrightarrow & \mathbb{N} \cup \{\infty\} \\
			& p & \longmapsto & \sum_{n \geq 0} \nu_p (k_n)
		\end{matrix}
	\end{equation}
	where $\mathbb{P}$ is the set of prime numbers and $\nu_p$ is the $p$-adic valuation $\nu_p(k) = \max \{ i \in \mathbb{N} \; | \; p^i  \text{ divides } k\}$.
\end{defi}

\begin{theo}(Classification of adding machines \cite{MR1332404} \cite{MR4492832}) \label{OdometerClassification}
	For any two positive sequences of integers $\underline{k}$ and $\underline{q}$, the classical odometers $\tau_{\underline{k}}$ and $\tau_{\underline{q}}$ are topologically conjugate if and only if the sequences $\underline{k}$ and $\underline{q}$ have the same multiplicity function $m_{\underline{k}}= m_{\underline{q}}$.
\end{theo}

\begin{cor} \label{Odom}
	Let $m : \mathbb{P} \to \mathbb{N} \cup \{ \infty \}$ be the map defined by
	\begin{equation}
		m(p) = \sup \{ \nu_p(\rho_n) \; ; \; n \geq 0 \}
	\end{equation}
	Then the odometer $\mathcal{T}_{|\omega(u)}$ is topologically conjugate to any $\underline{\rho}'$-odometer $\tau_{\underline{\rho}'}$ with multiplicity function $m_{\underline{\rho}'} = m$.
\end{cor}

\begin{proof}
	The first step is to understand the $\tau$-orbit of $\underline{0}$ in $Z_{\underline{\rho}}$ which has been proven in Theorem \ref{OmegaOdometer} to be homeomorphic to the $\mathcal{T}$-orbit of $u$ in $\mathcal{C}(M,\mathbb{R})$. For all integer $n \geq 0$, consider the following clopen $n$-cylinder 
	\begin{equation} \label{TopologyCylinder}
		\begin{split}
			C(\underline{0},n) &= \left\{ \underline{k} \in Z_{\underline{\rho}} \; \left| \; d(\underline{0},\underline{k}) \leq \frac{1}{2^{n+1}} \right.\right\} \\
			& = \left\{ \underline{k} \in Z_{\underline{\rho}} \; | \; k_i= 0 , \; i= 0 , .. , n \right\}
		\end{split}
	\end{equation}
	As noticed in the second point of Remark \ref{OdometersRem1}, the odometer $\tau$ is an isometry on $Z_{\underline{\rho}}$ so that for all integer $k \in \mathbb{Z}$, 
	\begin{equation*}
		\tau^k(C(\underline{0},n)) = C(\tau^k(\underline{0}),n)
	\end{equation*}
	We would like to associate each of these cylinders to the $(n+1)$-tuples of $Z_{\underline{\rho},n} = \prod_{i=0}^n \mathbb{Z}/\rho_i\mathbb{Z}$ formed by the first $n$ terms of the sequences $\tau^k(\underline{0})$. For that purpose, we define the set $C_{\underline{\rho},n}$ of $(n+1)$-cylinders of $Z_{\underline{\rho}}$ and the conjugacy map
	\begin{equation} \label{ConjugacyCylinders}
		\begin{matrix}
			\psi_{\underline{\rho},n} : & C_{\underline{\rho},n} & \longrightarrow & Z_{\underline{\rho},n} \\
			&C( \underline{q},n) & \longmapsto  & \underline{q}_n = (q_0, ..., q_n)
		\end{matrix}
	\end{equation}
	This conjugacy allows to identify the cylinders to $(n+1)$-tuples of $Z_{\underline{\rho},n}$.  
	
	 If we denote by $\orb(\tau, C(\underline{0},n))$ the $\tau$-orbit of the cylinder $C(\underline{0},n)$, we get the following diagram
	\begin{equation} \label{DiagramOrbit1}
		\begin{tikzcd}
			\orb(\tau,C(\underline{0},n)) \arrow[d, hook, "\psi_{\underline{\rho},n}"] \arrow{r}{\tau}  & \orb(\tau,C(\underline{0},n)) \arrow[d, hook, "\psi_{\underline{\rho},n}"] \\
			Z_{\underline{\rho},n}  \arrow{r}{\tau_n}  & Z_{\underline{\rho},n} 
		\end{tikzcd}	
	\end{equation}
	where $\tau_n$ is the odometer $\tau_n(\underline{q}) = \underline{q} + \underline{1}$ on $Z_{\underline{\rho},n}$. As a result, the $\tau$-orbit of $C(\underline{0},n)$ corresponds to the $\tau_n$-orbit of $\underline{0}_n = (0,..,0)$ in $Z_{\underline{\rho},n}$.\\
	
	\textit{Understanding $\orb(\tau_n,\underline{0}_n)$.} The set $\orb(\tau_n,\underline{0}_n)$ corresponds to the subgroup of the additive group $(\prod_{i=0}^n \mathbb{Z}/\rho_i\mathbb{Z}, +)$ generated by $\underline{1}_n = (1,..,1)$. Moreover, the theory of finite abelian groups shows that the order of this element $\underline{1}_n$ in $\prod_{i=0}^n \mathbb{Z}/\rho_i\mathbb{Z}$ is equal to the exponent of the entire group, which is $\lcm(\rho_i \; ; \; i=0 \; .. \;n)$ where lcm stands for least common multiple.\\
	
	\textit{Construction of a $\underline{\rho}'$-odometer $\tau_{\underline{\rho}'}$.} We would like to rewrite the diagram (\ref{DiagramOrbit1}) with isometric columns onto a new space $Z_{\underline{\rho}',n} = \prod_{i=0}^n \mathbb{Z}/\rho'_i\mathbb{Z}$ endowed with its usual metric (\ref{TopologyCoeff}). We also aim to replace $\tau_n$ with a map $\tau_{\underline{\rho}',n}$ derived from the action of a minimal $\underline{\rho}'$-odometer $\tau_{\underline{\rho}'}$ on the cylinder $C(\underline{0},n)$ of $Z_{\underline{\rho}'}$ as in the diagram \eqref{DiagramOrbit1}. To achieve this, note that by the discussion above, we have determined the precise cardinality of each orbit $\orb(\tau,C(\underline{0},n))$. Building on this, we set for all integer $n \geq 1$
	\begin{equation} \label{ConjugacySequence}
		\rho'_0 = \rho_0 \quad \text{and} \quad \rho'_n = \frac{\lcm(\rho_i \; ; \; i=0 \; .. \;n)}{\lcm(\rho_i \; ; \; i=0 \; .. \;n-1)}
	\end{equation}
	and consider the space $Z_{\underline{\rho}'} = \prod_{n=0}^\infty \mathbb{Z}/\rho'_n\mathbb{Z}$ and the corresponding $\underline{\rho}'$-odometer $\tau_{\underline{\rho}'}$. If we denote by $C'(\underline{0}', n)$ the $n$-cylinder around $\underline{0}'=(0,0,0,...)$ in the $Z_{\underline{\rho}'}$, then by definition of the minimal $\underline{\rho}'$-odometer $\tau_{\underline{\rho}'}$, we have
	\begin{equation*}
		\orb(\tau_{\underline{\rho}'}, C'(\underline{0}', n)) \simeq Z_{\underline{\rho}',n} := \prod_{i=0}^n \mathbb{Z}/\rho'_i\mathbb{Z}
	\end{equation*}
	and since 
	\begin{align*}
		\# Z_{\underline{\rho}',n}  = \prod_{i=0}^n \rho'_i = \rho_0.\prod_{i=1}^n \frac{\lcm(\rho_j \; ; \; j=0 \; .. \;i)}{\lcm(\rho_j \; ; \; j=0 \; .. \;i-1)} = \lcm(\rho_i \; ; \; i=0 \; .. \;n) = \# \orb(\tau,C(\underline{0},n))
	\end{align*}	
	This allows to define, up to conjugacy by the map $\psi_{\underline{\rho}',n}$ analogous to \eqref{ConjugacyCylinders}, a bijective map $\varphi_n : \orb (\tau,C(\underline{0}, n)) \to \orb(\tau_{\underline{\rho}'},C'(\underline{0}', n))$ as
	\begin{equation}
		\varphi_n( \tau_n^k(\underline{0}_n)) = \tau_{\underline{\rho}',n}^k (\underline{0}'_n)
	\end{equation}
	Thus, we get the following diagram
	\begin{equation} \label{DiagramOrbit2}
		\begin{tikzcd}
			\orb(\tau,C(\underline{0},n)) \arrow[d, "\varphi_n"] \arrow{r}{\tau_n}  & \orb(\tau,C(\underline{0},n)) \arrow[d, "\varphi_n"] \\
			Z_{\underline{\rho}',n} \simeq \orb(\tau_{\underline{\rho}'}, C'(\underline{0}', n)) \arrow{r}{\tau_{\underline{\rho}',n}}  & Z_{\underline{\rho}',n} \simeq \orb(\tau_{\underline{\rho}'}, C'(\underline{0}', n)) 
		\end{tikzcd}	
	\end{equation}
	where the columns are bijections.\\
	
	\textit{Construction of the conjugacy between $\tau_{|\omega(\underline{0})}$ and $\tau_{\underline{\rho}'}$.} Let us define the bijective map $\varphi : \orb(\tau,\underline{0}) \to \orb(\tau_{\underline{\rho}'},\underline{0}')$ by
	\begin{equation} \label{ConjugacyDefinition}
		\varphi( \tau^k(\underline{0})) = \tau_{\underline{\rho}'}^k (\underline{0}')
	\end{equation}	
	which is well defined since the two orbits are infinite. We show that it is an isometry. Let $\underline{q}$ and $\underline{p}$ be two elements of $\orb(\tau,\underline{0})$ and let $\underline{q}'=\varphi(\underline{q})$ and $\underline{p}'=\varphi(\underline{p})$ be their respective images. We have
	\begin{equation*}
		d(\underline{q}',\underline{p}') = \frac{1}{2^{\nu \left(\underline{q}', \underline{p}'\right)}} 
	\end{equation*}
	with
	\begin{align*}
		\nu \left(\underline{q}', \underline{p}'\right) & = \min \{n \geq 0 \; | \; q'_n \neq p'_n \} \\
		&= \min \{n \geq 0 \; | \; \underline{q}'_n \notin C'(\underline{p}'_n,n) \} \\
		&= \min \big\{n \geq 0 \; | \; \underline{q}'_n \notin \varphi_n \big(C(\underline{p}_n,n)\big) \big\}
	\end{align*}
	where in the second line we used the definition (\ref{TopologyCylinder}) of the cylinders $C'(\underline{p}'_n,n)$. If we consider two integers $k_q$ and $k_p$ such that $\underline{q} = \tau^{k_q}(\underline{0})$ and $\underline{p} = \tau^{k_p}(\underline{0})$. Then, we deduce from the diagram (\ref{DiagramOrbit2}) that 
	\begin{equation*}
		\underline{q}'_n = \tau_{\underline{\rho}', n}^{k_q}(\underline{0}_n) \notin \varphi_n \big(C(\tau_{\underline{\rho}', n}^{k_p}(\underline{0}_n),n)\big) \quad \text{if and only if} \quad \underline{q}_n = \tau_{n}^{k_q}(\underline{0}_n) \notin C(\tau_{n}^{k_p}(\underline{0}_n),n)
	\end{equation*}
	Thus, we get
	\begin{align*}
		\nu \left(\underline{q}', \underline{p}'\right) = \min \{n \geq 0 \; | \; \underline{q}_n \notin C(\underline{p}_n,n) \} = \nu \left(\underline{q}, \underline{p} \right)
	\end{align*}
	and 
	\begin{equation*}
		d(\underline{q}',\underline{p}') = \frac{1}{2^{\nu \left(\underline{q}', \underline{p}'\right)}} = \frac{1}{2^{\nu \left(\underline{q}, \underline{p}\right)}} = d(\underline{q},\underline{p})
	\end{equation*}
	Therefore, $\varphi : \orb(\tau,\underline{0}) \to \orb(\tau_{\underline{\rho}'},\underline{0}') \subset Z_{\underline{\rho}'}$ is an isometry. In particular, it is a uniformly continuous isometry from $\orb(\tau,\underline{0})$ to a subset of the complete compact space $Z_{\underline{\rho}'}$. Consequently,, it is possible to extend it to the closure $\omega(\underline{0})$ of $\orb(\tau,\underline{0})$, providing an isometric injection $\varphi : \omega(\underline{0}) \to Z_{\underline{\rho}'}$. Additionally, we know from the minimality of the $\underline{\rho}'$-odometer $\tau_{\underline{\rho}'}$ that $\omega(\underline{0}') = Z_{\underline{\rho}'}$. Thus, the conjugacy relation (\ref{ConjugacyDefinition}) and the continuity of $\varphi$ induce the inclusion
	\begin{equation*}
		Z_{\underline{\rho}'} = \omega(\underline{0}') \subset \varphi( \omega(\underline{0}))
	\end{equation*}
	This implies that the map $\varphi : \omega(\underline{0}) \to Z_{\underline{\rho}'}$ is a bijective isometry. Furthermore, the continuity of the odometers $\tau$ and $\tau_{\underline{\rho}'}$ allow to extend the conjugacy relation to $\omega(\underline{0})$, resulting in the desired commutative diagram
	\begin{equation} 
		\begin{tikzcd}
			\omega(\underline{0}) \arrow[d, "\varphi"] \arrow{r}{\tau} & \omega(\underline{0}) \arrow[d, "\varphi"] \\
			Z_{\underline{\rho}'} \arrow{r}{\tau_{\underline{\rho}'}}  & Z_{\underline{\rho}'} 
		\end{tikzcd}	
	\end{equation}
	where the columns are bijective isometries.\\
	
	\textit{Conclusion.} We showed that the odometer $\tau_{|\omega(u)}$ is topologically conjugate to the $\underline{\rho}'$-odometer $\tau_{\underline{\rho}'}$ for $\underline{\rho}'$ defined in (\ref{ConjugacySequence}). Let us compute its multiplicity map. For any prime number $p$, we have
	\begin{align*}
		m_{\underline{\rho}'}(p)& = \sum_{n = 0}^\infty \nu_p(\rho'_n) = \nu_p(\rho_0) + \sum_{n = 1}^\infty \nu_p(\lcm(\rho_i \; ; \; i=0 \; .. \;n))- \nu_p(\lcm(\rho_i \; ; \; i=0 \; .. \;n-1)) \\
		&= \lim\limits_n \nu_p(\lcm(\rho_i \; ; \; i=0 \; .. \;n)) = \lim\limits_n \sup_{2 \leq i \leq n} \nu_p(\rho_i) = \sup_{n \geq 0} \nu_p(\rho_n) = m(p)
	\end{align*}
	We have shown that the sequence $\underline{\rho}'$ has as a multiplicity function the map $m$ mentioned in the statement. An application of the classification Theorem \ref{OdometerClassification} concludes the proof of the corollary.
\end{proof}

\subsubsection{Description of the Non-Wandering Set $\Omega(\mathcal{T})$} \label{SectionNWLOManeLagrangian}

In this section, we extend our analysis beyond the $\omega$-limit of the specific viscosity solution $u$. We will describe the $\omega$-limits of any recurrent viscosity solution $v$ associated with the constructed Lagrangian $L$ defined in (\ref{LagMane}). Our approach involves applying a generalized representation formula on $\Omega(\mathcal{T})$ following \cite{Representation}. \\

But first, we need to identify the Mather set $\mathcal{M}_0$ expressed in the following proposition.
\begin{prop} \label{MatherConstructedProp}
	For the Mañé Lagrangian $L$ defined in (\ref{LagMane}), the time-zero projected Mather set $\mathcal{M}_0$ is given by 
	\begin{equation} \label{MatherConstructed}
		\begin{split}
		\mathcal{M}_0 = \mathcal{A}_0 &= M \setminus  \bigcup_{n \geq 0} \bigg( A_n \cup \bigcup_{0 \leq i < \rho_n} \big( B^i_n \setminus \{x^i_n\} \big) \bigg) \\
		&= \overline{D} \cup  \bigcup_{n\geq 0} \bigg( \overline{B_n} \setminus \bigcup_{0 \leq i < \rho_n} B^i_n \bigg) \cup  \bigcup\limits_{\substack{ n \geq 0 \\ 0 \leq i < \rho_n}}  \{x^i_n\}  
		\end{split}
	\end{equation}
	and the projected Mather set $\mathcal{M}$ is given by
	\begin{equation}
		\mathcal{M} = \{ (t, \mathcal{R}_t(x)) \; | \; x \in \mathcal{M}_0, t\in \mathbb{R} \}
	\end{equation}
\end{prop}

\begin{proof}
	We know from the inclusion $\mathcal{M}_0 \subset \mathcal{A}_0$ of Proposition \ref{MatherPeierls} that a point $x$ of $\mathcal{M}_0$ verifies $h^\infty(x,x) =0$. Thus, we can drop all the points that have a positive Peierls Barrier $h^\infty(x,x)>0$. According to Proposition \ref{chrec0}, all non-chain-recurrent points by the flow $f_t$ are to be dropped. This eliminated set $E$ of non-chain-recurrent points, equals by construction to 
	\begin{equation*}
		E := \bigcup_{n=0}^\infty \bigg( A_n \cup \bigcup_{0 \leq i < \rho_n} \big( B^i_n \setminus \{x^i_n\} \big) \bigg)
	\end{equation*}
	We deduce the inclusion
	\begin{align*}
		\mathcal{M}_0 \subset M \setminus E &= \overline{D} \cup  \bigcup_{n\geq 0} \bigg( \overline{B_n} \setminus \bigcup_{0 \leq i < \rho_n} B^i_n \bigg) \cup  \bigcup\limits_{\substack{ n \geq 0 \\ 0 \leq i < \rho_n}}  \{x^i_n\}  \\
		&= \overline{D'} \cup \bigcup\limits_{ n \geq 0 } \overline{D_n} \bigcup_{n\geq 0} \bigg( \overline{B_n} \setminus \bigcup_{0 \leq i < \rho_n} B^i_n \bigg) \cup  \bigcup\limits_{\substack{ n \geq 0 \\ 0 \leq i < \rho_n}}  \{x^i_n\}
	\end{align*}
	
	Let us verify the inverse inclusion. For any point $x$ of $\overline{D'}$, and for $v = X_0(x)= X_t(x)=0$, the measure $\mu = dt \otimes \delta_{(x,0)}$, where $\delta_{(x,0)}$ is the Dirac at $(x,v)$, is an invariant probability measure with
	\begin{equation*}
		\int_{\mathbb{T}^1 \times TM} L \; d\mu = \int_0^1 \frac{1}{2}\Vert v- X_t(x)\Vert^2 \; dt = 0 = -\alpha_0
	\end{equation*}
	Hence, $\mu$ is a minimizing measure and $x$ belongs to $\pi(\supp(\mu)) \cap \{t=0\} \subset \mathcal{M}_0$.
	
	Let $x$ be one of the $x^i_n$ or be a point of $\overline{B_n} \setminus \bigcup_{0 \leq i < \rho_n} B^i_n$. The point $x$ is $\rho_n$-periodic by the flow $f_t$ and so is $(x,v) = (x, X_0(x))$ by the Lagrangian flow $\phi_L^t$. Let $\gamma: \mathbb{R} \to M$ be the $\rho_n$-periodic loop, projection on $M$ of the loop $\phi^t_L(x,v) = \big(\gamma(t), X_t(\gamma(t))\big)$. And let $\mu$ be the uniform measure on the graph of $(t,\gamma(t)) \in \mathbb{T}^1 \times TM$. Once again, $\mu$ is a minimizing measure with support $\supp(\mu)$ being equal to the graph of $(t,\gamma(t))$. Hence, $x = \pi(\gamma(0))$ belongs to $\mathcal{M}_0$. 
	
	Now let $x$ be a point of $\overline{D_n}$. If it is periodic by $X_t(x) = Y_t(x)$, then it is in the Mather set $\mathcal{M}_0$ by the preceding case. The rotation number of the orbit $\mathcal{R}_t(x)$ is $\frac{1}{\rho_n}\eta_n(x)$ which only depends on $d(x, O_n)$. And since $\eta_n(x)$ is continuous and increasing with $d(x,O_n)$, we obtain a dense set $Q$ in $[2\delta_n, 3\delta_n]$ such that $\eta_n(Q) \subset \mathbb{Q}$. The set $\{x \in M \; | \; d(x,O_n) \in Q \}$ is dense in $\overline{D_n}$ and all its elements are periodic, and hence do belong to the Mather set $\mathcal{M}_0$. Since $\mathcal{M}_0$ is closed, we deduce that $\overline{D_n} \subset \mathcal{M}_0$. This terminates the proof of the identity (\ref{MatherConstructed}).
	
	For the projected Mather set $\mathcal{M}$, it suffices to observe that on the determined $\mathcal{M}_0$, we have $f_t = \mathcal{R}_t$. Given that the Mather set is $f_t$-invariant, the result becomes clear.
\end{proof}

We note that if we set $p_n = \prod\limits_{k=0}^n \rho_k$, then $\phi_{L}^{p_n}$ converges to the identity on $\mathcal{M}_0 setminus \bigcup_n D_n$. Moreover, a direct application of the properties of the Lemma \ref{SameClass} shows that the sets $D_n$ for all $x \in M$ and $y, z \in D_n$, then $\overline{h}(x,y) = \overline{h}(x,z)$ so that the sets $D_n$ are not detected by the Peierls barrier, and in particular by the generalized Peierls barrier $\underline{k}$ used in \cite{Representation} to describe the non-wandering set $\Omega(\mathcal{T})$. Consequently, our framework will meet the assumptions of the uniformly $\underline{p}$-recurrent case studied in Section 6.3.3 of \cite{Representation}.\\

In order to state the corresponding result, we need to introduce a final set of definitions used in the representation formula of $\Omega(\mathcal{T})$.

\begin{defi}
	\begin{enumerate}
		\item We define the \textit{$\underline{p}$-Peierls Barrier} $h^{\underline{p}}: M \times M \to \mathbb{R}$ by
		\begin{equation}
			h^{\underline{p}}(x,y) = \liminf_n h^{p_n}(x,y)
		\end{equation}
		with the corresponding time-depending $\underline{p}$-barrier
		\begin{equation}
			h^{\underline{p}+t}(x,y)= h^{\underline{p}}(t,x,y) = \liminf_n h^{t+p_n}(x,y)
		\end{equation}
		where $h^t$ is the potential introduced in (\ref{Potential}).
		
		\item We define the map $d_{\underline{p}} : M \times M \to \mathbb{R}_{\geq 0}$ by
		\begin{equation}
			d_{\underline{p}}(x,y) = h^{\underline{p}}(x, y) + h^{\underline{p}}(y, x)
		\end{equation}
	\end{enumerate}
\end{defi}

\begin{rem}
	As seen for the classical Peierls barrier in Proposition \ref{PeierlsProp}, the $\underline{p}$-barrier $h^{\underline{p}}(x, \cdot)$ is a viscosity solution for all point $x$ in $M$. Moreover, for any $\underline{p}$-recurrent point $x$ of the Mather set $\mathcal{M}_0$ with a $\underline{p}$-recurrent lift $\tilde{x} \in \tilde{\mathcal{M}}_0$ under the Lagrangian flow $\phi_L$, then the map $h^{\underline{p}}(x, \cdot)$ is a $\underline{p}$-recurrent viscosity solution. Its recurrence speed is controlled by $dist\big(\tilde{x}, \phi_L^{p_n}(\tilde{x})\big)$ or equivalently by $dist(x, f_{p_n}(x))$. (See \cite{Representation} for a proof of these non-trivial facts).
	
	In our case, if $x$ belongs to $D_n$, then we can take $x'$ be in the boundary $\partial D_n$ of $D_n$. Hence, the point $x'$ is $\underline{p}$-recurrent under the projected Lagrangian flow $f_t$, and by Property \ref{SameClass2} of Lemma \ref{SameClass}, we have $h^{\underline{p}}(x,\cdot) = h^{\underline{p}}(x',\cdot)$ which is a $\underline{p}$-recurrent viscosity solution with recurrence speed controlled by $dist\big( x', f_{p_n}(x')\big)$. 
	
	Consequently, we deduce that in our case, for all $x \in \mathcal{M}_0$, the map $h^{\underline{p}}(x,\cdot)$ is a $\underline{p}$-recurrent viscosity solution with recurrence speed uniformly controlled by $dist \big(f_{p_n|(\mathcal{M}_0 \setminus \bigcup_n D_n)}, Id \big)$. 
\end{rem}

\begin{prop}
	The map $d_{\underline{p}} : \mathcal{M}_0 \times \mathcal{M}_0 \to \mathbb{R}_{\geq 0}$ is a pseudometric on $\mathcal{M}_0$.
\end{prop}

\begin{proof}
	The symmetry is clear.\\
	\textit{Reflexiveness.} Let $x$ be a point of $\mathcal{M}_0$. If $x$ belongs to $D_n$ for some integer $n$, then Property \ref{SameClass2} of Lemma \ref{SameClass} asserts that $h^{\underline{p}}(x,x) = 0$. If not, then $x$ is $\underline{p}$-recurrent under the projected Lagrangian flow. We consider the $x(t)$ of the Mather set $\mathcal{M}$ starting at $x$. Then, Proposition \ref{CalibMather} asserts that it is calibrated by any weak-KAM solution $u$. We get from the liminf property (\ref{PeierlsLiminf}) of the Barrier that 
	\begin{align*}
		h^{\underline{p}}(x,x) = \liminf_n h^{p_n}(x,x) = \liminf_n h^{p_n}(x,x(p_n)) = \lim_n u(x(p_n)) - u(x) =0
	\end{align*}
	\textit{Triangular Inequality.} Using the first remark above and the triangular inequality (\ref{TriangIneg}) we get
	\begin{align*}
		h^{\underline{p}}(x,z) &=  \lim_n h^{\underline{p}+p_n}(x,z) = \lim_n \liminf_k h^{p_k+p_n}(x,z) \\
		&\leq \liminf_n \liminf_k h^{p_k}(x,y) + h^{p_n}(y,z) \\
		&= \liminf_k h^{p_k}(x,y) + \liminf_n  h^{p_n}(y,z) \\
		&= h^{\underline{p}}(x,y) + h^{\underline{p}}(y,z)
	\end{align*}
\end{proof}
		
This proposition allows to define the notion of $\underline{p}$-static classes as follows
\begin{defi}
	\begin{enumerate}
		\item We set $\sim$ to be the equivalence relation in $\mathcal{M}_0$ given by
		\begin{equation}
			x \sim y \Longleftrightarrow d_{\underline{p}}(x,y) = 0
		\end{equation}
	
		\item The \textit{$\underline{p}$-static classes} are the equivalence classes of the equivalence relation $\sim$. We denote by $\mathbb{M}_{\underline{p}}$ the set of $\underline{p}$-static classes represented by elements of $\mathcal{M}_0$ so that we have the inclusion $\mathbb{M}_{\underline{p}} \subset \mathcal{M}_0 \setminus \bigcup_n D_n$. 
		\item A map $\psi : \mathbb{M}_{\underline{p}} \to \mathbb{R}$ is said \textit{$\underline{p}$-dominated} if for all $x$ and $y$ in $\mathbb{M}_{\underline{p}}$, we have
		\begin{equation}
			\psi(y) - \psi(x) \leq  h^{\underline{p}}(x, y)
		\end{equation}
		We denote by $Dom_{\underline{p}}(\mathbb{M}_{\underline{p}})$ the set of dominated maps $\psi$ on $\mathbb{M}_{\underline{p}}$.
	\end{enumerate}
\end{defi}

The generalized representation formula is stated as follows

\begin{theo} \label{OmegaBijectionTheorem} (\cite{Representation})
	If $\phi_{L|\mathbb{M}_{\underline{p}}}^{p_n}$ converges uniformly to the identity on $\mathbb{M}_{\underline{p}}$, then we have a bijection
	\begin{equation} \label{OmegaBijection}
		\begin{matrix}
			Dom_{\underline{p}}(\mathbb{M}_{\underline{p}}) & \longrightarrow & \Omega(\mathcal{T}) \\
			\psi & \longmapsto & \inf\limits_{y \in \mathbb{M}_{\underline{p}}} \{ \psi(y) + h^{\underline{p}}(y, \cdot ) \}
		\end{matrix}
	\end{equation}
	Moreover, every element $v$ of $\Omega(\mathcal{T})$ is $\underline{p}$-recurrent with uniform recurrence over $\Omega(\mathcal{T})$.
\end{theo}

We will simply apply this theorem to the constructed \Mane Lagrangian $L$ defined in (\ref{LagMane}). To do so, we first need to identify the set $\mathbb{M}_{\underline{p}}$ of $\underline{p}$-static classes and then, to better understand the $\underline{p}$-barrier $h^{\underline{p}}$. This is done by an immediate application of the properties of Proposition \ref{chrec0} and Lemma \ref{SameClass}. Recall the notation of the different points introduced in (\ref{Points}).

\begin{prop} \label{pStaticClasses}
	We have 
	\begin{enumerate}[label= \roman*.]
		\item In dimension $d =2$,
		\begin{equation}
			\mathbb{M}_{\underline{p}} = \left( \bigcup_{\substack{n \geq 0 \\ 0\leq i < \rho_n}} \{x^i_n\} \right) \cup \left( \bigcup_{n \geq 0} \{y_n,z^+_n\} \right) \cup \{z_\infty\}
		\end{equation}
		\item In dimension $d \geq 3$, 
		\begin{equation}
			\mathbb{M}_{\underline{p}} = \left( \bigcup_{\substack{n \geq 0 \\ 0\leq i < \rho_n}} \{x^i_n\} \right) \cup \left( \bigcup_{n \geq 0} \{y_n\} \right) \cup \{z_\infty\}
		\end{equation}
	\end{enumerate}
\end{prop}

We define the set $\mathbb{M}'_{\underline{p}} = \mathbb{M}_{\underline{p}} \setminus \{z_\infty\}$. Since $\mathbb{M}'_{\underline{p}}$ is dense in $\mathbb{M}_{\underline{p}}$, and due to the continuity of $\underline{p}$-dominated maps $\psi : \mathbb{M}_{\underline{p}} \to \mathbb{R}$, these maps are determined by their images in $\mathbb{M}'_{\underline{p}}$.\\

Notice that all points $x$ of $\mathcal{M}_0$ and $\mathbb{M}_{\underline{p}}$ are periodic, with integer periods, under the flow $f_t$. Consider the map $\rho : \mathcal{M}_0 \to \mathbb{N}_{\geq 1}$ which associates to each points, its (positive) integer period in $\mathbb{M}_{\underline{p}}$. More precisely,
\begin{equation} \label{PeriodMather}
	\rho(x) = \rho_x =
	\begin{cases}
		\rho_n & \text{if } x = x^i_n, \; i=0,..,\rho_n-1 \\
		1 & \text{otherwise}
	\end{cases}
\end{equation}

\begin{rem} \label{RemarkPeriodMather}
	There is a subtlety about the points $y_n$. Note that all element $x$ of $\mathcal{M}_0 \setminus \{y_n\}_{n \geq 0}$ are $\rho_x$-periodic under the flow $f_t$. However, this does not hold for the points $y_n$. Indeed, we have $\rho_{y_n}=1$ and the point $y_n$ is not $1$-periodic under the flow $f_t$. However, its $\underline{p}$-static class $\bar{y}_n$ is a fixed point of $f_{1\left|\mathbb{M}_{\underline{p}}\right.}$ i.e. the points $y_n$ and $f_1(y_n)$ belong to the same $\underline{p}$-static class $\bar{y}_n$.
	
	Furthermore, we can prove using Lemma \ref{SameClass} that
	\begin{align*}
		h^{\rho_n \infty + k}(y_n, \cdot) = h^{\rho_n \infty }(f_k(y_n), \cdot) = h^{\rho_n \infty }(y_n, \cdot)
	\end{align*}
	which implies
	\begin{align*}
		h^{\rho_n\infty}(y_n, \cdot) = h^{\infty} (y_n, \cdot) = h^{\rho_{y_n}\infty} (y_n, \cdot)
	\end{align*}
	This justifies the reason of taking $\rho(y_n) = 1$.
\end{rem}

\begin{prop} \label{hcircles}
	For every point $x$ in $\mathcal{M}_0$, the $\underline{p}$-barrier $h^{\underline{p}}(\cdot, x,\cdot)$ and the $\rho_x$-barrier $h^{\rho_x\infty}(\cdot, x, \cdot)$ do coincide on $\mathbb{R} \times M$. 
\end{prop}

\begin{proof}
	Fix an element $(t,x,y)$ in $\mathbb{R} \times \mathcal{M}_0 \times M$. Observe from Remark \ref{RemarkPeriodMather} that if $x = y_n$, then taking $\rho_{y_n} =1$ or $\rho_n$ makes no difference. Thus, we can assume that $x$ is $\rho_x$-periodic under the flow $f_s$.\\
	 For $n \geq 0$ sufficiently large, the definition of $p_n$ allows us to consider an integer $p'_n$ such that $p_n = \rho_x.p'_n$. Then, we have
	\begin{align*}
		h^{\underline{p}+t}(x ,y ) & = \liminf_n h^{p_n + t}(x,y) = \liminf_n h^{\rho_x.p'_n + t}(x,y) \\
		& \geq \liminf_n h^{\rho_x.n + t}(x,y) = h^{\rho_x\infty + t}(x,y)
	\end{align*}
	For the inverse inequality, fix an integer $m \geq 1$ and let $n$ be such that $p_n \geq \rho_x. m$. The Tonelli Theorem \ref{TonelliTheorem} guarantees the existence of a curve $\gamma_1 : [0,\rho_x.m +t] \to M$ from $x$ to $y$ such that $h^{\rho_x.m +t}(x,y) = A_L(\gamma_1)$. Now let $\gamma_2 : [ 0 , p_n -\rho_x.m]$ be the curve defined by $\gamma_2(s) =f_s(x)$ so that it has null action $A_L(\gamma_2)=0$. 
	Since $x$ is $\rho_x$-periodic by the flow $f_t$, and the integer $\rho_x$ divides $p_n -\rho_x.m$, we deduce that $\gamma_2(0) = \gamma_2 (p_n -\rho_x.m) = x$ and that $\gamma_2$ is a loop. Then, we can concatenate the two curves $\gamma_2$ and $\gamma_1$ to obtain a third one $\gamma : [0, p_n +t] \to M$ connecting $x$ to $y$. Hence, we get
	\begin{align*}
		h^{\rho_x.m +t}(x,y) & = A_L(\gamma_1) = A_L(\gamma_1) + A_L(\gamma_2)  = A_L(\gamma) \geq h^{p_n +t}(x,y)
	\end{align*}
	Taking the liminf on $m$ and $k$, we derive the desired inequality
	\begin{equation*}
		h^{\rho_x\infty + t}(x,y) \geq h^{\underline{p}+t}(x ,y )
	\end{equation*}
	which concludes the proof.
\end{proof}

\begin{rem}
	Using this proposition, we can express the constructed recurrent viscosity solutions $u$ defined in (\ref{ic}) as $u(x) = \inf_{n \geq 0} \{ h^{\underline{p}}(x_n, x) \}$.
\end{rem}

The properties drawn from Propositions \ref{pStaticClasses} and \ref{hcircles} above finalize the application of Theorem \ref{OmegaBijectionTheorem} to the \Mane Lagrangian $L$, yielding

\begin{cor}
	We have a bijective map
	\begin{equation} \label{OmegaBijectionApplication}
		\begin{matrix}
			Dom_{\underline{p}}(\mathbb{M}_{\underline{p}}) & \longrightarrow & \Omega(\mathcal{T}) \\
			\psi & \longmapsto & \inf\limits_{y \in \mathbb{M}'_{\underline{p}}} \{ \psi(y) + h^{\rho_y \infty}(y, \cdot ) \}
		\end{matrix}
	\end{equation}
	where the structure of $\mathbb{M}'_{\underline{p}} = \mathbb{M}_{\underline{p}} \setminus \{z_\infty\}$ is detailed in Proposition \ref{pStaticClasses}.
\end{cor}

We were able in the previous Subsection \ref{SectionAddingMachine} to describe the dynamics of the Lax-Oleinik operator $\mathcal{T}$ on $\omega(u)$ for the initial date $u$ defined in (\ref{ic}). The following is a partial extension to the whole non-wandering set $\Omega(\mathcal{T})$.

\begin{prop} \label{MainPropOmega}
	For all recurrent viscosity solution $v$ in $\Omega(\mathcal{T})$, the Lax-Oleinik operator $\mathcal{T}$ restricted to the $\omega$-limit set $\omega(v)$ is a factor of an odometer map $\tau$.	
\end{prop}

\begin{proof}
	Let $v$ be an element of $\Omega(\mathcal{T})$ represented by the $\underline{p}$-dominated map $\psi : \mathbb{M}_{\underline{p}} \to \mathbb{R}$. Following the proof of Proposition \ref{TkuProp}, we get that for all non-negative integer $k \geq 0$, 
	\begin{equation} \label{Tkv2}
		\mathcal{T}^kv(x) =  \inf\limits_{y \in \mathbb{M}_{\underline{p}}} \{ \psi(y) + h^{\rho_y \infty + k}(y, \cdot ) \} \in \Omega(\mathcal{T})
	\end{equation}	 
	As in the proof of Theorem \ref{OmegaOdometer}, we will characterize every element of $\orb(\mathcal{T},v)$ by the coefficients $k$ involved in the $(\rho_y, k)$-barriers $h^{\rho_y \infty + k}$ intervening in (\ref{Tkv2}). For all $y$ in $\mathbb{M}_{\underline{p}}$ such that $\rho_y = 1$, we set $k_y = 0$. The other points $y \in \mathbb{M}_{\underline{p}}$ being in the $\{x^i_n \; ; \;  0 \leq i \leq \rho_n-1 , \; n \geq 0, \}$, we deduce that the remaining coefficients $k_y$ belong to the space $Z := \prod_{n \geq 0} (\mathbb{Z} / \rho_n \mathbb{Z})^{\rho_n}$. In other words, every element of the orbit $\orb(\mathcal{T},v)$ under $\mathcal{T}$ can be seen as an element of $Z$.
	
	We endow the space $Z$ with a metric analogously to (\ref{TopologyCoeff}). And we consider the continuous map $\varphi$ defined by
	\begin{equation}
		\begin{matrix}
			\varphi: & Z = \prod_{n \geq 0} (\mathbb{Z} / \rho_n \mathbb{Z})^{\rho_n} & \longrightarrow & \Omega (\mathcal{T}) \\
			& \underline{k} = (k_{x^i_n})_{i,n} & \longmapsto & \inf\limits_{y \in \mathbb{M}_{\underline{p}}} \{ \psi(y) + h^{\rho_y \infty + k_y}(y, \cdot ) \}
		\end{matrix}
	\end{equation}
	Let $\tau$ be the odometer map on $Z$ given by
	\begin{equation}
		\begin{matrix}
			\tau : & Z & \longrightarrow & Z \\
			& \underline{k}= (k_{x^i_n})_{i,n} &\longmapsto &\underline{k} + \underline{1} = (k_{x^i_n}+1)_{i,n}
		\end{matrix}
	\end{equation}
	Then, the identity (\ref{Tkv2}) gives rise to the commutative diagram
	\begin{equation}
		\begin{tikzcd}
			Z \arrow[d, "\varphi"] \arrow{r}{\tau}  & Z \arrow[d, " \varphi"] \\
			\Omega(\mathcal{T}) \arrow{r}{\mathcal{T}}  & \Omega(\mathcal{T})
		\end{tikzcd}	
	\end{equation}
	Note that $\varphi(\underline{0}) = v$. Hence, by continuity of the maps involved in the diagram, their restrictions to the set $Z_0 = \overline{\{ \tau^n(\underline{0}) \; ; \; n \in \mathbb{N} \}}$ yields the new diagram
	\begin{equation}
		\begin{tikzcd}
			Z_0 \arrow[d, "\varphi"] \arrow{r}{\tau}  & Z_0 \arrow[d, " \varphi"] \\
			\omega(v) \arrow{r}{\mathcal{T}}  & \omega(v)
		\end{tikzcd}	
	\end{equation}
	where the map $\varphi : Z_0 \to \omega(v)$ is onto. Therefore, $\mathcal{T}_{|\omega(v)}$ is a factor of an odometer map $\tau : Z_0 \to Z_0$.
\end{proof}

\begin{rem}
	\begin{enumerate}
		\item The Theorem \ref{MainOmega} is a direct consequence of the proof above. 
		\item This proof shows that the behaviour of $\mathcal{T}_{|\omega(v)}$ cannot be more complicated than that of $\tau_{|Z_0}$. Moreover, it is possible to obtain the entire dynamics of the odometer $\tau_{|Z_0}$ as seen in Theorem \ref{OmegaOdometer} and Corollary \ref{Odom} where we got $Z_0 \simeq \omega(u)$ and $\tau_{|Z_0}$ is conjugate to $\mathcal{T}_{|\omega(u)}$.
		\item When comparing the outcomes of the last two subsections, we can ask a natural question : Is it possible to determine the exact odometer acting on $\omega(v)$ ? This proves more intricate than in the case of $u$ due to the non-injectivity of the the map $\varphi$ constructed in the proof above. Even the fact that $\omega(v)$ is Cantor space isn't that clear.
		\item However, for a generic $\underline{p}$-dominated map $\psi$ where the domination is everywhere strict, and if we impose a non symmetry condition between the values of $\{\psi(x^i_n)\}_i$, then it becomes possible to ensure the injectivity of the constructed map $\varphi$. Consequently, the map $\varphi$ is a conjugacy between $\mathcal{T}_{|\omega(u)}$ and the odometer $\tau_{|Z_0}$ and we deduce that $\omega(v)$ is a Cantor set homeomorphic to $Z_0$.
	\end{enumerate}
\end{rem}

\section{Construction of a Smooth Non-Periodic Recurrent Viscosity Solution} \label{RegSection}

In Section \ref{SectionRecurrence}, we constructed a Lipschitz viscosity solution $u(t,x)$ of the Hamilton-Jacobi equation (\ref{HJ}) that is recurrent and non-periodic. In this section, we undertake the proof of Theorem \ref{MainC1} and refine our choice of $u$ to achieve local $C^\infty$ regularity.

\subsection{An Informal Preliminary Discussion.} \label{SubsectionInformal}

The initial data $u$ chosen in \eqref{ic} to obtain a recurrent, non-periodic viscosity solutions was 
\begin{equation}
	u(x) = \inf_{n \geq 0} \{ h^{\rho_n\infty}(x_n, x) \}
\end{equation}
where $x_n = x^0_n$ were defined in (\ref{Points}). However, there was no control over the sets of realization of the infimums and their boundaries where non-differentiability is very likely. Moreover, it has been established in Section \ref{SectionNWLOManeLagrangian} that all the recurrent viscosity solutions for the studied \Mane Lagrangian $L$ defined in (\ref{LagMane}) are of the form 
\begin{equation*}
	u_c(x) = \inf_{y \in \mathbb{M}} \{ c_y+ h^{\rho_y\infty}(y,x) \}
\end{equation*}
where $\mathbb{M}$ is a well chosen subset of the Mather set $\mathcal{M}_0$ and $\rho_y$ are positive integers associated with the points $y$. Therefore, the potential regularization of $u$ should take this form.\\

For simplicity, let us fix an integer $n \geq 0$ and focus on a viscosity solution, simpler than \eqref{ic}, with initial data $v$ given by
\begin{equation}
	v(x) = \inf \{ h^\infty (x_n, x) , h^\infty (z_n,x) \}
\end{equation}
where $z_n$ is an arbitrary point of $\partial C_n$. Let us look at its possible regularity in $C_n$.
\begin{enumerate}
	\item There is a significant risk of irregularity at the boundary between the two domains of the infimum, where $h^\infty (x_n, x) = h^\infty (z_n,x)$. However, this risk vanishes if the equality occurs on $\partial B_n$, which belongs to the Mather set where differentiability has been guaranteed by Remark \ref{MatherRegularity}. This concern can be dealt with by taking a modification of $v$ of the form 
	\begin{equation}
		v_c(x) = \inf \{ c_n + h^\infty (x_n, x) , h^\infty (z_n,x) \}
	\end{equation}
	where the constants $c_n$ are to be well chosen.
	\item Let us look at the regularity of $h^\infty (x_n, x)$ in $B^0_n$. Suppose that $h^\infty (x_n, x)$ is "realized" or calibrates two distinct curves $\gamma_1, \gamma_2 : (-\infty,0] \to M$ linking $x_n$ at $t \to -\infty$ to $x$ at $t=0$. 
	
	Being minimizing, these curves do follow the Lagrangian flow, implying $\dot{\gamma}_1(0) \neq \dot{\gamma}_1(0)$. However, if $h^\infty (x_n, x)$ is regular at $x$, the regularity Theorem \ref{CalibRegularity} imposes that $d_x h^\infty (x_n, x) = \partial_v L(x,\dot{\gamma}_1(0))$ and  $d_x h^\infty (x_n, x) = \partial_v L(x,\dot{\gamma}_2(0))$ which are distinct due to the convexity of $L$. 
	
	As a result, a condition for regularity is to have one and only one minimizing curve that goes from $x_n$ to $x$, and this is the main reason why we imposed the various symmetries listed in Remarks \ref{Symmetriesf} and \ref{SymmetryX} during the construction. The proofs will heavily rely on these.
\end{enumerate}

Let's now delve into the proof. We begin by addressing the choice of the constants $c_n$.

\subsection{Choice of the initial data $u$}

As discussed in Subsection \ref{SubsectionInformal}, we will proceed to a good choice of constants $c_n$, $n \geq 0$, ensuring that the viscosity solution $u_c$ defined by 
\begin{equation}
	u_c(x) = \inf_{n \geq 0} \{ c_n + h^{\rho_n\infty}(x_n, x) \}
\end{equation}
remains non-periodic, recurrent, and becomes regular. Additionally, we introduce the map $u^n_c$ defined by

\begin{equation} \label{uc}
	u_c^n(x) = \inf_{k \neq n} \{ c_k+ h^{\rho_k\infty}(x_k,x) \}
\end{equation}

\begin{prop} \label{Cn}
	There exist a sequence of real constants $c_n$ such that for the the associated solution $u_c$ we have
	\begin{enumerate}[label=\roman*.]
		\item In the sets $B^0_n$, $u_c(x) = c_n + h^{\rho_n\infty}(x_n, x)$.
		\item In the sets $A_n$, 
		\begin{itemize}
			\item (2D case) Either $u_c(x)$ is locally constant. 
			\item (3D case and above) Or $u_c(x)= u^n_c (x) =  h^{\infty}(z_\infty, x)$, where the point $z_\infty$ has been introduced in (\ref{Points}). 
		\end{itemize}
		\item In the sets $B_n \setminus B^0_n$, $u_c(x)$ is constant.
		\item In the set $D$, $u_c(x)$ is locally constant.
	\end{enumerate}
\end{prop}

\begin{prop} \label{Cnt}
	For the same real constants $c_n$ of Proposition \ref{Cn}, we have for all time $t \in \mathbb{R}$,
	\begin{enumerate}[label=\roman*.]
		\item In the sets $\mathcal{R}_t(B^0_n)$, $u_c(t,x) = c_n + h^{\rho_n\infty+t}(x_n, x)$.
		\item In the sets $A_n$, 
		\begin{itemize}
			\item (2D case) Either $u_c(t,x)$ is locally constant in $(t,x)$. 
			\item (3D case and above) Or $u_c(t,x)= u^n_c (t,x) =  h^{\infty+t}(z_\infty, x)$. 
		\end{itemize}
		\item In the sets $B_n \setminus \mathcal{R}_t(B^0_n)$, $u_c(t,x)$ is constant in $(t,x)$.
		\item In the set $D$, $u_c(t,x)$ is locally constant.
	\end{enumerate}
\end{prop}

In these propositions, locally constant can be replaced by constant on the connected components of the considered sets.\\

Recall from Remark \ref{RemarkDimension} that the 2D case and the higher-dimensional case present some topological differences. This adds a bit of intricacy to the selection of the constants $c_n$ which can be overtaken by separating the cases.

\subsubsection{Dimension $d \geq 3$ case} \label{3DSection}

We prove Proposition \ref{Cn} in the case of dimension $d \geq 3$. Recall from (\ref{Points}) that we fixed a point $z_\infty$ in $\overline{D}$ and points $y_n$ in $\partial B_n$. We set
\begin{equation} \label{Cn3DFormula}
	c_n = h^\infty (z_\infty,y_n) - h^{\rho_n \infty}(x_n,y_n) 
\end{equation}
and we will show that these constants are convenient.\\

But first, note that the connectedness of the set $D$ in this higher-dimensional case leads to the following lemma, which slightly simplifies the current scenario.

\begin{lem} \label{Ucn3D}
	 For all $x \in D$, $u_c^n(x) =0$ and for all $x \in C_n$, $u_c^n(x) = h^\infty (z_\infty,x)$.
\end{lem}

\begin{proof}
	Let $x$ be a fixed point of $C_n$ and fix an integer $k \neq n$. Let us evaluate $h^{\rho_k\infty}(x_k,x)$. The set $\overline{D'}$ separates $x$ and $x_k$. Then, applying Property \ref{SeparateClass} of Lemma \ref{SameClass} applied to $F = \overline{D'}$ followed by an application of the Property \ref{SameClass1} to the $1$-periodic point $z_\infty$, yields
	\begin{equation}
		\begin{split}
			h^{\rho_k\infty}(x_k, x) &= h^{\rho_k \infty} (x_k,z_\infty) +  h^{\rho_k \infty} (z_\infty,x) \\
			&= h^{\rho_k \infty} (x_k,z_\infty) +  h^{\infty} (z_\infty,x)
		\end{split}
	\end{equation}
	Additionally, the set $\partial B_k$ separates the points $x_k$ and $z_\infty$ and the application of the same Lemma \ref{SameClass} gives
	\begin{equation*} 
		h^{\rho_k \infty}(x_k,z_\infty) = h^{\rho_k \infty}(x_k,y_k) + h^{\infty}(y_k,z_\infty)
	\end{equation*}
	Moreover, applying Property \ref{SameClass2} of Lemma \ref{SameClass} to $F= \overline{D} \supset \partial C_n$, we get that $h^{\infty}(\cdot, \partial C_n) = h^{\infty}(\cdot, \overline{D})$ is well defined. Hence, the liminf property (\ref{PeierlsLiminf}) applied to any curve $f_t(y)$, $y \in A_k$, which $\alpha$ and $\omega$-limit sets respectively belong to $\partial B_k$ and $\partial C_k \subset \overline{D}$ yields
	\begin{align} \label{hnulltoward}
		h^{\infty}(y_k,z_\infty) = h^{\infty}(\partial B_n, \partial C_n) \leq \liminf_i A_L(f_t(y)_{|[-i,i]}) =0
	\end{align}
	This implies the equality $h^{\rho_k \infty}(x_k,z_\infty) = h^{\rho_k \infty}(x_k,y_k)$.\\
	Gathering the identities leads to
	\begin{align*}
		u_c^n(x)&= \inf_{k \neq n} \{ c_k+ h^{\rho_k\infty}(x_k,x) \} \\
		&=\inf_{k \neq n} \{ c_k+  h^{\rho_k\infty}(x_k,z_\infty) +h^{\infty}(z_\infty,x) \} \\
		&= u_c^n(z_\infty) +h^{\infty}(z_\infty,x) 
	\end{align*}
	with 
	\begin{align*}
		u_c^n(D) = u_c^n(z_\infty) &= \inf_{k \neq n} \{ c_k+  h^{\rho_k\infty}(x_k,z_\infty) \} \\
		&= \inf_{k \neq n} \{ c_k+  h^{\rho_k\infty}(x_k,y_k) \} \\
		&= \inf_{k \neq n} \{  h^{\rho_k\infty}(z_\infty,y_k) \} 
	\end{align*}
	
	Moreover, the regularity of the barriers, as stated in Proposition \ref{PeierlsProp}, and another application of Lemma \ref{SameClass} provide
	\begin{align*}
		0 \leq h^\infty (z_\infty,y_k) & = h^\infty (z_\infty,y_k) - h^\infty (z_\infty,\partial C_k) \leq \kappa_1. d(y_k, \partial C_k) = \kappa_1. \delta_k \to 0 \quad \text{as } k \to \infty
	\end{align*}
	Thus, $u_c^n(D) = \inf_{k \neq n} \{ h^\infty (z_\infty,y_k) \} =0$ and we obtain the wanted result $u_c^n(x) = h^{\infty}(z_\infty,x)$.
\end{proof}

We can now conclude the proof of Propositions \ref{Cn} \ref{Cnt} for dimensions greater than $3$.

\begin{prop} \label{Cn3D}
	The constants $c_n$ defined in (\ref{Cn3DFormula}) do answer the requirements of Proposition \ref{Cn}. More precisely,
	\begin{enumerate}[label=\roman*.]
		\item\label{Cn3D1} For all $x \in B^0_n$, $c_n + h^{\rho_n\infty}(x_n,x) \leq h^\infty (z_\infty,x)$.
		\item For all $x \in A_n$, $c_n + h^{\rho_n\infty}(x_n,x) \geq h^\infty (z_\infty,x)$.
		\item\label{Cn3D2} For all $x \in C_n \setminus (A_n \cup B^0_n)$, $c_n + h^{\rho_n\infty}(x_n,x) = h^\infty (z_\infty,x)= h^\infty (z_\infty,y_n)$.
		\item \label{Cn3D3} For all $x \in D$, $u_c(x) = 0$.
	\end{enumerate}
\end{prop}

\begin{proof}
	\textit{\ref{Cn3D3}} This is due to the Lemma \ref{Ucn3D} which states that for all $x \in D$, $u_c^n(x) = 0$ and $u_c(x) = \inf_{n \geq 0}{u_c^n(x)} =0$.\\
	
	\textit{\ref{Cn3D2}} This third point is the most straightforward. In fact, the proof of Lemma \ref{Ucn3D} and more precisely of (\ref{hnulltoward}) shows that $h^{\rho_n\infty}(x_n,\cdot)$ is constant equal to $h^{\rho_n\infty}(x_n,y_n)$ on $C_n \setminus (B^0_n)$. Similarly, $h^\infty (z_\infty,\cdot)$ is constant equal to $h^\infty (z_\infty,y_n)$ on $\overline{B_n}$. Thus, recalling the definition (\ref{Cn3DFormula}) of $c_n$, we get the equality between the two maps $c_n + h^{\rho_n\infty}(x_n,\cdot)$ and $h^\infty (z_\infty,\cdot)$ on $C_n \setminus (A_n \cup B^0_n) = \overline{B_n} \setminus B^0_n$.\\
	
	\textit{\ref{Cn3D1}} Now we prove the first point. Let $x$ be a point of $B^0_n$. Using the triangular inequality (\ref{TriangInegPeierls}) for the Peierls barrier, we get
	\begin{align*}
		h^{\rho_n\infty}(x_n,x) \leq h^{\rho_n\infty}(x_n,y_n) + h^{\rho_n\infty}(y_n,x)
	\end{align*}	 
	However a similar computation to (\ref{hnulltoward}) shows that $h^{\rho_n\infty}(y_n,x) = 0$. Moreover, we just saw from the third point that $h^{\infty}(z_\infty,x) = h^{\infty}(z_\infty,y_n)$. Hence, we deduce that
	\begin{align*}
		c_n + h^{\rho_n\infty}(x_n,x) \leq c_n + h^{\rho_n\infty}(x_n,y_n) = h^{\infty}(z_\infty,y_n) =  h^{\infty}(z_\infty,x)
	\end{align*}
	The second point on $A_n$ is proved analogously.
\end{proof}

\begin{prop} \label{Cnt3D}
	The constants $c_n$ defined in (\ref{Cn3DFormula}) do answer the requirements of Proposition \ref{Cnt}. More precisely, for all time $t \in \mathbb{R}$,
	\begin{enumerate}[label=\roman*.]
		\item\label{Cnt3D1} For all $x \in \mathcal{R}_t(B^0_n)$, $c_n + h^{\rho_n\infty+t}(x_n,x) \leq h^{\infty+t} (z_\infty,x)$.
		\item For all $x \in A_n$, $c_n + h^{\rho_n\infty+t}(x_n,x) \geq h^{\infty+t} (z_\infty,x)$
		\item\label{Cnt3D3} For all $x \in C_n \setminus (A_n \cup \mathcal{R}_t(B^0_n))$, $c_n + h^{\rho_n\infty+t}(x_n,x) = h^{\infty+t} (z_\infty,x)= h^{\infty+t} (z_\infty,y_n) = h^{\infty} (z_\infty,\partial B_n)$
		\item \label{Cnt3D3} For all $x \in D$, $u_c(t) = 0$.
	\end{enumerate}
\end{prop}     

\begin{proof}
	We prove that
	\begin{align*}
		h^{\infty+t}(z_\infty,\cdot) = h^{t,\infty+t}(z_\infty, \cdot) \quad \text{and} \quad h^{\rho_n\infty+t}(x_n, \cdot) = h^{t,\rho_n\infty+t}(\mathcal{R}_t(x_n), \cdot)
	\end{align*}
	Indeed, we use the triangular inequality \ref{TriangInegPeierls2} to obtain
	\begin{align*}
		h^{\rho_n\infty+t}(x_n, x) \leq h^{t}(x_n, f_t(x_n)) + h^{t,\rho_n\infty+t}(f_t(x_n), x) = h^{t,\rho_n\infty+t}(\mathcal{R}_t(x_n), x)
	\end{align*}
	and
	\begin{align*}
		h^{t,\rho_n\infty+t}(\mathcal{R}_t(x_n), x) \leq h^{t,\rho_n\infty}( \mathcal{R}_t(x_n), x_n) + h^{\rho_n\infty}(x_n,x) =  h^{\rho_n\infty}(x_n,x)
	\end{align*}
	where we used the fact that $f_t(x_n) = \mathcal{R}_t(x_n)$ is $\rho_n$-periodic and of null action. Similarly, we prove that for all $\rho_n$-periodic point $x$ in $C_n$,
	\begin{align*}
		h^{\infty+t}(z_\infty,\mathcal{R}_t x) = h^{\infty}(z_\infty,x)  \quad \text{and} \quad h^{\rho_n\infty+t}(x_n, \mathcal{R}_t x) = h^{\rho_n\infty}(x_n,x)
	\end{align*}
	Hence, we deduce that 
	\begin{align*}
		c_n &= h^\infty(z_\infty, y_n) - h^{\rho_n\infty}(x_n, y_n)\\
		&=h^\infty(z_\infty, \partial B_n) - h^{\rho_n\infty}(x_n, \partial B_n)\\
		&= h^{\infty+t}(z_\infty, \mathcal{R}_t(\partial B_n)) - h^{\rho_n\infty+t}(x_n, \mathcal{R}_t(\partial B_n))\\
		&= h^{\infty+t}(z_\infty, \partial B_n) - h^{\rho_n\infty+t}(x_n, \partial B_n)\\
		&= h^{t,\infty+t}(z_\infty, \partial B_n) - h^{t,\rho_n\infty+t}(x_n, \partial B_n)
	\end{align*}

	Note that the barrier $h^{t,\infty+t}$ is the Peierls barrier associated to the \Mane Lagrangian associated to the vector field $X_{t+\tau}$ with flow $f_{t,\tau}$, which restriction to $C_n$ can be represented by a rotation of Figure \ref{FigureConstruction} by $\mathcal{R}_t$. Therefore, the proof of Proposition \ref{Cnt3D} is analogous to that of \ref{Cn3D}.
	
	The fact that 
	\begin{align*}
		h^\infty(z_\infty, y_n) = h^\infty(z_\infty, \partial B_n) = h^{\infty+t}(z_\infty, \partial B_n) = h^{t,\infty+t}(z_\infty, \partial B_n)
	\end{align*}
	gives the constance in $t$ and $x$ in the statement of Proposition \ref{Cnt}.
\end{proof}

\subsubsection{Dimension $d = 2$ case}

We now turn our attention to the dimension $2$ case. In the forthcoming lemmas and results, we will present concise proofs, omitting redundant details that closely resemble those in the higher-dimensional case. However, we will highlight and elaborate on the distinctions pertinent to the 2D scenario.\\

As mentioned in Remark \ref{RemarkDimension}, the main difference between the 2D and the higher dimensional case is the disconnectedness of the sets $D$, $D_n$ and $A_n$. These invalidate the Lemma \ref{Ucn3D} as a curve that goes from $z_\infty$ to a point $x$ must cross all the sets $A_n$ and $B_n$ between them.\\

Recall from (\ref{Sets}) that the sets $A_n$ are divided into two connected components $A^\pm_n$, with $A^+_n$ possessing the larger $r$-coordinate. We also introduced in (\ref{Points}) the points $z^\pm_n \in \partial A^\pm_n \cap \partial C_n$ and the points $y_n \in \partial B_n$. 

\begin{lem} \label{2Dz+z-}
	We have the equalities
	\begin{equation}
		h^\infty(z^+_n,y_n) = h^\infty(z^-_n,y_n) = h^\infty(\partial C_n,y_n)
	\end{equation}
\end{lem}

This is due to the symmetries of $X_t$ in $A_n$ which were listed in Remark \ref{SymmetryX}. More precisely, $X_t$ is invariant by rotation in the $\theta$-coordinate and is symmetric with respect to the circle $O_n = \{r = r_n\}$. These symmetries allow for a reduction to the one-dimensional case. The proof of this lemma is postponed to Subsection \ref{SuitableCalib}, where explicit computations will be carried out using calibrated curves.

We set the constants $c_n$ to be
\begin{equation} \label{Cn2DFormula}
	c_n = - h^{\rho_n \infty}(x_n,y_n) 
\end{equation}

\begin{prop} \label{Cn2D}
	These constants $c_n$ do answer the requirements of Proposition \ref{Cn}. More precisely,
	\begin{enumerate}[label=\roman*.]
		
		\item \label{Cn2DDem1} For all $x \in C_0$, $u_c^n(x) = h^\infty(z^-_n,x)$.
		\item \label{Cn2DDem2} For all $n \geq 1$ and for all $x \in C_n$, $u_c^n(x) = \min \{h^\infty(z^\pm_n,x) \}$.
		\item \label{Cn2DDem3} For all $x \in C_n$, $c_n + h^{\rho_n\infty}(x_n,x) \leq u^n_c(x)$.
		\item \label{Cn2DDem4} For all $x \in C_n \setminus B^0_n$, $h^{\rho_n\infty}(x_n,x) = h^{\rho_n\infty}(x_n,y_n)$.
	\end{enumerate}		
\end{prop}

\begin{proof}
	\textit{\ref{Cn2DDem4}} The last statement is the most straightforward and is proved in the same fashion as the last point of Proposition \ref{Cn3D}. 
	
	\textit{\ref{Cn2DDem2}} We now direct our attention to the second point and fix an integer $n \geq 1$. Let $x$ be a point within $C_n$, and consider an integer $0 \leq k < n$. Given that the curves linking $y_k$ to $y_{k+1}$ must intersect $\partial A^-_k$ and $\partial A^+_{k+1}$, the application of Property \ref{SeparateClass} from Lemma \ref{SameClass} yields:
	\begin{align*}
		h^{\infty}(y_k,y_{k+1}) &=  h^{\infty}(y_k,z^-_k)+ h^{\infty}(z^-_k,z^+_{k+1}) + h^{\infty}(z^+_{k+1}, y_{k+1}) = h^{\infty}(z^+_{k+1}, y_{k+1})
	\end{align*}
	where all the null terms are deduced from the direction of the flow $f_t$ from one component $\partial A^{\pm}_i$ to another or from Property \ref{SameClass2} of Lemma \ref{SameClass}. Hence, since every curve linking $y_k$ to $z^+_n$ must the sets $\{r=r_i+\delta_i\} \subset \partial A^+_i$ for all $i=k+1,..,n-1$, we obtain
	\begin{align*}
		h^{\infty}(y_k,z^+_n) &= \sum_{j=k+1}^{n-1}  h^\infty(y_{j-1}, y_j)  +  h^\infty(y_{n-1}, z^+_n) \\
		&= \sum_{j=k+1}^{n-1} h^\infty(z^+_j, y_j)  + 0
	\end{align*}      
	Again, since every curve linking $x_k$ to $z_n^+$ must intersect $\{r=r_k+\delta_k\} \ni y_k$ and $\{r=r_n - 2\delta_n \} \ni z_n$, we get
	\begin{equation} \label{Cn2DDemo1}
		\begin{split}
			c_k + h^{\rho_k\infty}(x_k,x) &= c_k + h^{\rho_k\infty}(x_k,y_k) + h^{\infty}(y_k,z^+_n) + h^{\infty}(z^+_n,x) \\
			&= \sum_{j=k+1}^{n-1} h^{\infty}(z^+_j, y_j) + h^{\infty}(z^+_n,x) 
		\end{split}
	\end{equation}
	It follows that
	\begin{align*}
		u^{<n}_c(x) : =\inf_{k < n} \{ c_k + h^{\rho_k\infty} (x_k,x) \} = c_{n-1}+ h^{\rho_{n-1}\infty}(x_{n-1},x) = h^\infty( z^+_n,x)
	\end{align*}
	Similarly, we show that
	\begin{align*}
		u^{>n}_c(x) :=\inf_{k > n} \{ c_k + h^{\rho_k\infty} (x_k,x) \} = c_{n+1}+ h^{\rho_{n+1}\infty}(x_{n+1},x) = h^\infty( z^-_n,x)
	\end{align*}
	The result follows from the identity $u^n_c = \min \{ u^{<n}_c, u^{>n}_c(x) \}$.

	\textit{\ref{Cn2DDem1}} In the case of $C_0$, we have $u^0_c = u^{>n}_c(x)$.

	\textit{\ref{Cn2DDem3}} For this last inequality, one needs to notice that for all $x \in C_n$, 
	\begin{align*}
		c_n + h^{\rho_n\infty}(x_n,x) \leq c_n + h^{\rho_n\infty}(x_n,\partial C_n) = c_n + h^{\rho_n\infty}(x_n,y_n) =0 \leq u_c^n(x)
	\end{align*}
\end{proof}

Working with the \Mane Lagrangian associated to $X_{t+\tau}$, we obtain
\begin{prop} \label{Cnt2D}
	These constants $c_n$ do answer the requirements of Proposition \ref{Cn}. More precisely,
	\begin{enumerate}[label=\roman*.]
		\item \label{Cnt2DDem1} For all $x \in C_0$, $u_c^n(t,x) = h^{\infty+t}(z^-_n,x)$.
		\item \label{Cnt2DDem2} For all $n \geq 1$ and for all $x \in C_n$, $u_c^n(x) = \min \{h^{\infty+t}(z^\pm_n,x) \}$.
		\item \label{Cnt2DDem3} For all $x \in C_n$, $c_n + h^{\rho_n\infty+t}(x_n,x) \leq u^n_c(t,x)$.
		\item \label{Cnt2DDem4} For all $x \in C_n \setminus B^0_n$, $h^{\rho_n\infty+t}(x_n,x) = h^{\rho_n\infty+t}(x_n,y_n)$.
	\end{enumerate}		
\end{prop}

\subsubsection{Recurrence and Non-Periodicity of $u_c$}

We verify that the chosen initial data corresponds to a recurrent, non-periodic viscosity solution $u_c(t,x)$. We saw in Section \ref{SectionNWLOManeLagrangian} that the form (\ref{uc}) of the chosen initial data $u_c$ corresponds to a recurrent viscosity solution of the Hamilton-Jacobi equation. Hence, we only need to confirm the non-periodicity.

\paragraph{Non-Periodicity.}
We claim that 
\begin{enumerate}[label=(\roman*)]
	\item For $k=0,\rho_n$, $\mathcal{T}^ku_c(x_n) = c_n$.
	\item For $k=1,...,\rho_n-1$, $\mathcal{T}^ku_c(x_n) > c_n$.
\end{enumerate}

In fact, Proposition \ref{TkuProp} and Lemma \ref{Tk+iu} yield 
\begin{equation}	\label{Tkucormula}
	\mathcal{T}^k u_c(x) = \inf_{n \geq 0} \{c_n + h^{\rho_n\infty+k}(x_n, x) \} = \inf_{n \geq 0} \{c_n + h^{\rho_n\infty}(x^k_n, x) \}
\end{equation}
A symmetric version of Propositions \ref{Cn3D} and \ref{Cn2D} featuring $x^i_n$ in $B^i_n$ instead of $x_n$ in $B^0_n$ results in :
\begin{enumerate}[label=(\roman*)]
	\item For $k =0,\rho_n$,
	\begin{align*}
		\mathcal{T}^k u_c(x_n) = c_n + h^{\rho_n\infty}(x^k_n, x_n) = c_n + h^{\rho_n\infty}(x_n, x_n) = c_n
	\end{align*}
	\item For all $k=1,...,\rho_n-1$, 
	\begin{align*}
		\mathcal{T}^k u_c(x_n) = c_n + h^{\rho_n\infty}(x^k_n, x_n) = c_n + h^{\rho_n\infty}(x^k_n, y_n) > c_n
	\end{align*}
\end{enumerate}
Therefore, $(\mathcal{T}^ku_c)_{k\geq 0}$ cannot be periodic.

\subsection{Proof of the $C^\infty$ Regularity of $u_c$}

We begin the proof of Theorem \ref{MainC1} by establishing the $C^\infty$ regularity of the constructed recurrent, non-periodic viscosity solution $u_c$. This approach heavily relies on the differentiability along calibrated curves, as outlined in Theorem \ref{CalibRegularity}. Since the differential has a simple expression on these curves, higher regularity can be achieved through a well-behaved foliation of $M$ by calibrated curves. Thus, identifying these curves becomes crucial. The preliminary lemmas presented in the next subsection serve this purpose.

\subsubsection{Preliminary Lemmas : Identifying Suitable Calibrated Curves} \label{SuitableCalib}

We will take advantage of the autonomy of the flow $g_t$ used to define $f_t$. Let us introduce the autonomous \Mane Lagrangian $L_Z : TM \to \mathbb{R}$ associated to the vector field $Z$ defined by
\begin{equation} \label{LZ}
	L_Z(x,v) = \frac{1}{2} \Vert v - Z(x) \Vert^2
\end{equation}
And we denote by $h_Z$ and $m_Z$ its relative barriers and \Mane potentials.\\

Analogously to Propisition \ref{MatherConstructedProp}, we get
\begin{prop} \label{MatherZ}
	The Mather set $\tilde{\mathcal{M}}_Z$ and its projection $\mathcal{M}_L$ to $M$ associated to $L_Z$ are $\tilde{\mathcal{M}}_Z = \mathbb{T}^1 \times \tilde{\mathcal{M}}_0$ and $\mathcal{M}_Z = \mathbb{T}^1 \times \mathcal{M}_0$
\end{prop}

\begin{lem} \label{CalibZX}
	A curve $\gamma : \mathbb{R} \to C_n$ is calibrated by $h^{\rho_n\infty}(x_n, \cdot)$ (resp. $h^\infty( z_\infty , \cdot)$) if and only if the curve $\sigma (t)= \mathcal{R}_t^{-1} \circ \gamma(t)$ is calibrated by $h_Z^{\rho_n\infty}(x_n, \cdot)$ (resp. $h_Z^\infty( z_\infty , \cdot)$).
\end{lem}

\begin{proof}
	We prove the case of $h^{\rho_n \infty}(x_n,\cdot)$. We will show that for all times $t \in \mathbb{R}$ and all points $x \in C_n$,
	\begin{equation} \label{CalibZXDem2}
		h^{\rho_n\infty+t}(x_n,x) = h_Z^{\rho_n\infty+t}(x_n,\mathcal{R}_t^{-1} x)
	\end{equation}
	We establish double inequality. Let $n_k$ be an increasing sequence of integers and $\sigma_k :[0,\rho_nn_k+t] \to M$ be a sequence of curves linking $x_n$ to $\mathcal{R}_t^{-1} x$ and such that $h_Z^{\rho_n\infty}(x_n,\mathcal{R}_t^{-1} x) = \lim_k A_{L_Z}(\sigma_k)$.\\
	
	We aim to replace $\sigma_k$ by curves $\tilde{\sigma}_k$ with image in the closure $\overline{C}_n$ of $C_n$. We treat the case of dimension 3 or above. We consider the real numbers $s_k^\pm$ defined by
	\begin{equation*}
		s^-_k := \inf \{\tau \in [0,\rho_nn_k +t] \; | \; \sigma_k(\tau) \notin C_n \} \quad \text{and} \quad s^+_k := \sup \{\tau \in [0,\rho_nn_k +t ] \; | \; \sigma_k(\tau) \notin C_n \}
	\end{equation*}
	Fix a smooth map $\zeta_k : [0,1] \to \partial C_n$ linking $\sigma_k(s_k^-)$ to $\sigma_k(s_k^+)$. Such map exists since in dimension 3 or above, the set $\partial C_n$ is connected. Then, for all integer $l \geq 1$, we set the curve $\sigma_{k,l} : [0,\rho_n(n_k+l) +t] \to \overline{C}_n$ to be given by
	\begin{equation}
		\sigma_{k,l}(\tau) = 
		\begin{cases}
			\sigma_k(\tau) & \text{if } \tau \in [0,s_k^-] \\
			\zeta_k\left( \frac{\tau-s_k^-}{s_k^+ - s_k^- + \rho_nl} \right) & \text{if } \tau \in [s_k^-, s_k^+ + \rho_nl] \\
			\sigma_k(\tau - \rho_nl) & \text{if } \tau \in [s_k^+ + \rho_nl,\rho_n(n_k+l)+t]
		\end{cases}
	\end{equation}
	Hence, we have
	\begin{align*}
		A_{L_Z}(\sigma_{k,l}) &= \int_0^{\rho_n(n_k+l)+t} \frac{1}{2} \Vert \dot{\sigma}_{k,l} - Z(\sigma_{k,l}) \Vert^2 d\tau \\
		&= \int_{[0,s_k^-] \cup [s_k^+\rho_nn_k+t]} \frac{1}{2} \Vert \dot{\sigma}_k - Z(\sigma_k) \Vert^2 \; d\tau + \int_{s_k^-}^{s_k^+ + \rho_nl} \frac{1}{2} \Vert \dot{\sigma}_{k,l} - Z(\sigma_{k,l}) \Vert^2 d\tau  \\
		&= A_{L_Z}(\sigma_{k|[0,s_k^-] \cup [s_k^+\rho_nn_k]})  + \int_{s_k^-}^{s_k^+ + \rho_nl} \frac{1}{2({s_k^+ - s_k^- + \rho_nl})^2} \left\Vert \dot{\zeta}_k\left( \frac{\tau-s_k^-}{s_k^+ - s_k^- + \rho_nl} \right) \right\Vert^2 d\tau  \\
		&\leq A_{L_Z}(\sigma_k) + \frac{\Vert \dot{\zeta}_k\Vert^2}{2({s_k^+ - s_k^- + \rho_nl})}  \longrightarrow 0 \quad \text{as } l \to +\infty
	\end{align*}
	where we used in the third equality that $Z$ is null in $\partial C_n$. Hence, by extracting simultaneously two increasing subsequences $n_{k_i}$ and $l_i$ of $n_k$ and $l$ such that $A_{L_Z}(\sigma_{k_i,l_i})$ converges, we obtain
	\begin{equation*}
		\lim_i A_{L_Z}(\sigma_{k_i,l_i}) \leq \lim_i A_{L_Z}(\sigma_{k_i})
	\end{equation*}
	Thus, using the liminf property \eqref{PeierlsLiminf} of the Peierls barrier yields
	\begin{align*}
		h_Z^{\rho_n \infty+t}(x_n,\mathcal{R}_t^{-1} x) \leq \lim_i A_{L_Z}(\sigma_{k_i,l_i}) \leq \lim_i A_{L_Z}(\sigma_{k_i}) = h_Z^{\rho_n \infty+t}(x_n,\mathcal{R}_t^{-1} x)
	\end{align*}	 
	and we deduce the equality $h_Z^{\rho_n \infty+t}(x_n,\mathcal{R}_t^{-1} x) = \lim_i A_{L_Z}(\sigma_{k_i,l_i})$ for curves $\sigma_{k_i,l_i}$ with images in $C_n$.
	
	We get the same conclusion for the 2D case by setting $s_{k,j}^\pm$ being successive times of entering and exiting the connected components of $M \setminus C_n$.\\
	
	We can now show the identity \eqref{CalibZXDem2}. For simplicity, we change the notation $\sigma_i$ for $\sigma_{k_i,l_i}$. We set for all $i$ the curve $\gamma_i (\tau)= \mathcal{R}_\tau \sigma_i(\tau)$. The evaluation of its velocity gives
	\begin{align*}
		\dot{\gamma_i}(\tau) = \frac{d}{dt}(\mathcal{R}_\tau \sigma_i)(\tau) = \frac{d\mathcal{R}_\tau}{d\tau} \sigma_i(\tau) + d\mathcal{R}_\tau.\dot{\sigma}_i(\tau) = Y_\tau(\gamma_i(\tau)) + d\mathcal{R}_\tau.\dot{\sigma}_i(\tau)
	\end{align*}		
	and using the form of $X_\tau$ given by (\ref{XFormula}), we get
	\begin{align*}
		\Vert \dot{\gamma}_i(\tau) - X_\tau(\gamma_i(\tau)) \Vert &= \Vert d\mathcal{R}_\tau.\dot{\sigma}_i(\tau) - d\mathcal{R}_\tau Z (\sigma_u(\tau)) \Vert = \Vert \dot{\sigma}_i(\tau) - Z (\sigma_i(\tau)) \Vert
	\end{align*}
	and consequently, $A_L(\gamma_i) = A_{L_Z}(\sigma_i)$. Moreover, we have $\gamma_i(0) = \sigma_i(0) = x_n$ and by $\rho_n$-periodicity of $\mathcal{R}_\tau$ on $C_n$, 
	\begin{align*}
		\gamma_i(\rho_n(n_{k_i}+l_i)+t) = \mathcal{R}_{\rho_n(n_{k_i}+l_i)+t}\sigma_i(\rho_n(n_{k_i}+l_i)+t) = \mathcal{R}_t\mathcal{R}_t^{-1}x = x
	\end{align*}		
	Hence, we the liminf Property \eqref{PeierlsLiminf} leads to
	\begin{align*}
		h^{\rho_n\infty+t}(x_n,x) \leq \liminf_i A_L(\gamma_i) = \lim_i A_{L_Z}(\sigma_i) = h_Z^{\rho_n \infty+t}(x_n,\mathcal{R}_t^{-1} x)
	\end{align*}
	
	The inverse inequality is obtained analogously by showing that we can find $\gamma_i$ with image on $C_n$ such that $h^{\rho_n\infty+t}(x_n,x) = \lim_k A_L(\gamma_i)$. This ends the proof the identity \eqref{CalibZXDem2}. Application to calibrated curves $\gamma$ and $\sigma(t) = \mathcal{R}_t^{-1} \circ \gamma(t)$ yields
	\begin{align*}
		h_Z^{\rho_n\infty+t}(x_n,\sigma(t)) = h^{\rho_n\infty+t}(x_n,\gamma(t)), \quad h^{\infty+t}(z_\infty,\gamma(t)) =  h_Z^{\infty+t}(z_\infty,\sigma(t)) \quad \text{and} \quad   A_{L_Z}(\sigma) = A_L(\gamma)
	\end{align*}
	These imply the result.
\end{proof}

The following Lemmas focus on identifying the calibrated curves of $h^{\rho_n\infty}_Z(x_n,\cdot)$ and $h_Z^\infty(\partial C_n, \cdot)$ respectively on  $B^0_n$ and $A_n$.

\begin{lem} \label{Confined}
	Let $u$ be a periodic viscosity solution of the Hamilton-Jacobi equation (\ref{HJ}) associated to $L_Z$ and consider for all point $(t,x)$ in $\mathbb{R} \times M$ a $u$-calibrated curve $\sigma_{t,x} : (-\infty, t ] \to M$ with $\sigma_{t,x}(t) = x$. Then
	\begin{enumerate} [label= \roman*.]
		\item If $x$ belongs $B^i_n$ then for all time $s \leq t$, $\sigma_{t,x}(s) \in B^i_n$.
		\item If $x$ belongs $A_n$ then for all time $s \leq t$, $\sigma_{t,x}(s) \in A_n$.
	\end{enumerate}
\end{lem}

Note that periodic viscosity solutions for $L_Z$ are stationary weak-KAM solutions due to Fathi's Convergence Theorem \ref{FathiTh}.

\begin{proof}
	We only prove the case of $B^0_n$. The case of $A_n$ is done similarly. Let $x$ be a point of $B^0_n$. Arguing by contradiction, suppose that $\sigma_{t,x}$ exits the ball $B^0_n$. Then there must exist a time $s < t$ such that $\sigma_{t,x}(s) \in \partial B^0_n$. Since the curve $\sigma_{t,x}$ is calibrated by $u$, we deduce from Theorem \ref{CalibRegularity} that $u$ is differentiable at $(s,\sigma_{t,x}(s))$ and $\partial_x u (s,\sigma_{t,x}(s)) = \partial_v L (s,\sigma_{t,x}(s), \dot{\sigma}_{t,x}(s))$.
	
	Now recall from Proposition \ref{MatherZ} that $\partial B^0_n \subset \mathcal{M}_Z$. Thus, we get the inclusion $(s,\sigma_{t,x}(s)) \in \mathcal{M}_Z$ and there exists an element $v$ of $T_{\sigma_{t,x}(s)}M$ such that $(s,\sigma_{t,x}(s), v) \in \tilde{\mathcal{M}}_Z$. The Proposition \ref{CalibMather}, added to the regularity Theorem \ref{CalibRegularity} on calibrated curves, shows that $u$ must be differentiable at $(s,\sigma_{t,x}(s), v)$ and $\partial_x u (s,\sigma_{t,x}(s)) = \partial_v L_Z (\sigma_{t,x}(s), v)$. We get $\partial_v L_Z (\sigma_{t,x}(s), \dot{\sigma}_{t,x}(s)) = \partial_v L_Z (\sigma_{t,x}(s), v)$. However, the Legendre map $\mathcal{L} : (\tau,y,w) \to (\tau,y,\partial_v L_Z (y,w))$ is a bijective in the Tonelli case (see \cite{fathi2008weak}). Hence, we get $v = \dot{\sigma}_{t,x}(s)$ and by the $\phi_{L_Z}$-invariance of the Mather set, we also get that $(\tau,\sigma_{t,x}(\tau)) =\big(\tau, \pi \circ \phi_{L_Z}^{s,\tau}(\sigma_{t,x}(s), v)\big)$ belongs to $\tilde{\mathcal{M}}_Z$ for all times $\tau \leq t$. The fact that $\sigma_{t,x}(s) \in \partial B^0_n$ implies that $\sigma_{t,x}$ must belong in the connected component of $\mathcal{M}_Z$ that contains $\partial B^0_n$. This contradicts the fact that $x$ belongs to $B^0_n$.
\end{proof}

\begin{lem} \label{CalibAB}
	For all integer $n \geq 0$,
	\begin{enumerate}
		\item On $B^0_n$, consider the spherical coordinates $(\delta_B, \varphi^B_2, .. ,\varphi^B_d) = (\delta_B, \bar{\delta}_B)$.\\
		If $x \in B^0_n$ and $\sigma_x: (-\infty, t] \to M$ with $\sigma_x(t) = x$ is a curve calibrated by $h_Z^{\rho_n\infty}(x_n, \cdot)$, then $\sigma_x(t)$ has constant $\bar{\delta}_B$-coordinate, $\alpha_1(\sigma_x) = \{x_n\}$ and for all $s \leq t$
		\begin{equation}
			h_Z^{\rho_n\infty}(x_n, \sigma_x(s)) = h_Z^{\rho_n\infty+s}(x_n, \sigma_x(s)) = A_L( \sigma_{x|(-\infty,s]}) = \int_{-\infty}^s L(\tau, \sigma_x(\tau), \dot{\sigma}_x(\tau)) \; d\tau
		\end{equation}
		\item On $A_n$, we have $A_n \subset C_n = O_n \times B^{d-1}$ where is a $d-1$-dimensional ball centered on $x_n$ and of radius $2\delta$. Using the spherical coordinates on $B^{d-1}$, we define the coordinates $(\theta,\delta_A, \varphi^A_2, .. ,\varphi^A_d) = (\delta_A, \bar{\delta}_A)$ on $A_n$.\\
		If $x \in A_n$ and $\sigma_x: (-\infty, t] \to M$ with $\sigma_x(t) = x$ is a curve calibrated by $h_Z^{\infty}(z_\infty, \cdot)$, then $\sigma_x(t)$ has constant $\bar{\delta}_A$-coordinate, $\alpha_1(\sigma_x) \subset \partial C_n$ and for all $s \leq t$
		\begin{equation}
			h_Z^{\infty}(z_\infty, \sigma_x(s)) = h_Z^{\infty+s}(z_\infty, \sigma_x(s)) = A_L( \sigma_{x|(-\infty,s]}) = \int_{-\infty}^s L(\tau, \sigma_x(\tau), \dot{\sigma}_x(\tau)) \; d\tau
		\end{equation}
	\end{enumerate}
\end{lem}

\begin{proof}
	We only prove the first point, the second being analogous. Note that due to Fathi's theorem \ref{FathiTh}, we have $h_Z^{\rho_n\infty+s} = h_Z^{\rho_n\infty} = h_Z^\infty$ is independent on time. Hence, we deduce from Property \ref{pvisc} of Proposition \ref{PeierlsProp} that $h_Z^{\rho_n\infty}(x_n,\cdot)$ is a weak-KAM solution associated to $L_Z$. And since $x$ belongs to $B^0_n$, we deduce from Lemma \ref{Confined} that $\sigma_x(\tau) \in B^0_n$ for all time $\tau \leq t$.\\
	
	We first prove that $\alpha_1(\sigma_x) = \{x_n\}$. We saw that $\sigma_x$ is calibrated by a weak-KAM solution. Hence, Proposition \ref{SemiStaticCalib} asserts that $\sigma_x$ is semi-static, and according to Proposition \ref{SemiStaticLimitInclusions}, its $\alpha$-limit $\alpha_1(\sigma_x)$ belongs to a single static class of $\mathbb{M}_Z$. Knowing that $\sigma_x$ has its image in the ball $B^0_n$, we deduce that $\alpha_1(\sigma_x) = \bar{x}_n = \{x_n\}$, or $\alpha_1(\sigma_x) \subset \partial B^0_n$.
	
	Assume that $\alpha_1(\sigma_x) \subset \partial B^0_n$. There exists an increasing sequence of integers $k_i$ and a point $y \in \partial B^0_n$ such that $\sigma_x(-\rho_n k_i)$ converges to $y$. Taking the limit for $t= -\rho_n k_i$ in the calibration equation and using the positivity of the Lagrangian $L_Z$, we get
	\begin{align} \label{DemSymmetryh2}
		h_Z^{\rho_n \infty}(x_n, x) = h_Z^{\rho_n\infty}(x_n,y) + A_{L_Z}(\sigma_{x|(-\infty,0]}) > h_Z^{\rho_n\infty}(x_n,y)
	\end{align}
	the strict inequality is due to the strict positivity of $A_L(\sigma_{x|(-\infty,0]}) > 0$ since the flow $g_t$ of $Z$ is directed from $\partial B^0_n$ towards its center $x_n$, at the opposite of $\sigma_x$.
	However, an application of the triangular inequality (\ref{TriangInegPeierls}) of the Peierls barrier gives
	\begin{align*}
		h_Z^{\rho_n \infty}(x_n, x) \leq h_Z^{\rho_n \infty}(x_n, y) + h_Z^{\rho_n \infty}(y, x)
	\end{align*}
	We claim that $h_Z^{\rho_n \infty}(y, x) = 0$. Indeed, if $y'$ is a limit point of the sequence $g_{\rho_nk}^{-1}(x)$, we have $y' \in \partial B^0_n$ and
	\begin{align*}
		0 \leq h_Z^{\rho_n \infty}(y, x) \leq h_Z^{\rho_n \infty}(y, y') + h_Z^{\rho_n \infty}(y', x) \leq h_Z^{\rho_n \infty}(y, y') + \liminf_k h_Z^{\rho_nk}(f_{\rho_nk}^{-1}(x),x) = 0 + 0
	\end{align*}
	where we used Lemma \ref{SameClass} for the nullity of the first term and the fact that the flow $g_t$ has null action for the nullity of the second term. We finally get $h_Z^{\rho_n \infty}(x_n, x) \leq h_Z^{\rho_n \infty}(x_n, y)$ which contradicts (\ref{DemSymmetryh2}). Therefore, we deduce that $\alpha_1(\sigma_x) = \bar{x}_n = \{x_n\}$.\\
	
	 Let $k_i$ be an increasing integer sequence such that $\sigma_x(-\rho_nk_i)$ converges to $x_n$. Taking the limit in the calibration equation, we get
	\begin{align*}
		h_Z^{\rho_n \infty}(x_n, x) &= h_Z^{\rho_n\infty}(x_n, x_n) + \lim_i A_L(\sigma_{x|(-\rho_nk_i,0]}) = A_L(\sigma_{x|(-\infty,0]}) \\
	\end{align*}
	
	Let us show that $\sigma_x$ has constant $\bar{\delta}_B$-coordinate. We adopt the notation $\sigma_x = (\sigma_\delta, \sigma_{\bar{\delta}})$ in the spherical coordinates. Let $\bar{\delta}_t$ be the $\bar{\delta}$-coordinate of $x=(x_\delta, x_{\bar{\delta}})$ in the spherical coordinates, and consider the curve $\tilde{\sigma}_x$ given by $\tilde{\sigma}_x = (\sigma_\delta, \bar{\delta}_t) \in  \{\bar{\delta} = \bar{\delta}_t\}$. Recall from Remark \ref{SymmetryX} that, on $B^0_n$, the vector field $Z$ has null $\bar{\delta}$ coordinate and that its $\delta$-coordinate $Z_\delta(y)$ depends only on the $\delta$-coordinate of $y$, i.e $Z_\delta(\sigma_x(t)) = Z_\delta(\tilde{\sigma}_x(t))$. Hence, using the notation $\dot{\sigma}_x=  \dot{\sigma}_{\delta} \overset{\perp}{+} \dot{\sigma}_{\bar{\delta}}$, we have for all $\tau \leq t$, 
	\begin{align*}
		L_Z(\sigma_x(\tau), \dot{\sigma}_x(\tau)) &= \frac{1}{2} \Vert \dot{\sigma}_x(\tau) - Z(\sigma_x(\tau)) \Vert^2 = \frac{1}{2} | \dot{\sigma}_{\delta}(\tau) - Z_\delta(\sigma_x(\tau)) |^2 + \frac{1}{2} \Vert \dot{\sigma}_{\bar{\delta}}(\tau) \Vert^2 \\
		&= \frac{1}{2} \Vert \dot{\tilde{\sigma}}_x(\tau) - Z(\tilde{\sigma}_x(\tau)) \Vert^2 + \frac{1}{2} \Vert \dot{\sigma}_{\bar{\delta}}(\tau) \Vert^2 = L_Z(\tilde{\sigma}_x(\tau), \dot{\tilde{\sigma}}_x(\tau)) + \frac{1}{2} \Vert \dot{\sigma}_{\bar{\delta}}(\tau) \Vert^2 
	\end{align*}
	We have $\alpha_1(\tilde{\sigma}_x) = \{x_n\}$. Thus, noting that $x= \sigma_x(t) = \tilde{\sigma}_x(t)$, we get
	\begin{align*}
		h_Z^{\rho_n\infty+t}(x_n, \sigma_x(t)) & =  h_Z^{\rho_n\infty+t}(x_n, \tilde{\sigma}_x(t)) \\
		&\leq \int_{-\infty}^t L_Z(\tilde{\sigma}_x(\tau), \dot{\tilde{\sigma}}_x(\tau)) \; d\tau  \\
		&\leq \int_{-\infty}^t L_Z(\tilde{\sigma}(\tau), \dot{\tilde{\sigma}}(\tau)) + \frac{1}{2} \Vert \dot{\sigma}_{\bar{\delta}}(\tau) \Vert^2  d\tau \\
		&= \int_{-\infty}^t L_Z(\sigma(\tau), \dot{\sigma}(\tau))  \; d\tau = h_Z^{\rho_n\infty+t}(x_n, \sigma(t))
	\end{align*}
	Therefore, there is equality everywhere and $\int_{-\infty}^t \frac{1}{2} \Vert \dot{\sigma}_{\bar{\delta}}(\tau) \Vert^2  d\tau =0$. By continuity of $\dot{\sigma}_{\bar{\delta}}$, we deduce that it is null and that $\sigma_{\bar{\delta}}$ is constant.
\end{proof}

\begin{rem}
	By Property \ref{SeparateClass} of Lemma \ref{SameClass}, note that for all curve $\gamma$ in $C_n$, being calibrated by $h_Z^\infty( z_\infty, \cdot)$ is equivalent to being calibrated by $h_Z^\infty( \partial C_n , \cdot) = h_Z^\infty( z_\infty, \cdot)$ in dimension higher than 2 and by $h_Z^\infty( z_n^- , \cdot) = -h_Z^\infty( z_\infty, z_n^-) + h_Z^\infty( z_\infty, \cdot)$ in the 2D case. Moreover, the result remains valid for the calibration by $h_Z^\infty( z_n^+ , \cdot) $.  
\end{rem}

\begin{lem} \label{CalibZ}
	We denote by $\phi_{-Z}$ the flow associated to the vector field $-Z$.
	\begin{enumerate}
		\item \label{CalibZRadial} For all $x \in B^0_n$, the curve $\sigma_x : \mathbb{R} \to B^0_n$ defined by $\sigma_x(t) = \phi_{-Z}^t(x)$ is calibrated by $h_Z^{\rho_n\infty}(x_n, \cdot)$.

		\item \label{CalibZThetaCoord} For all $x \in A_n$, the curve $\sigma_x : \mathbb{R} \to A_n$ defined by $\sigma_x(t) = \phi_{-Z}^t(x)$ is calibrated by $h_Z^{\infty}(\partial C_n, \cdot)$.
	\end{enumerate}
\end{lem}

\begin{proof}
	We prove the first point, the second being analogous. Let $t>0$ be a positive time, and let $y = g_t(x) = \phi_{-Z}^t(x)$ be a point of $B^0_n$. Since $h^{\rho_n\infty}(x_n,\cdot)$ is a viscosity solution, we get from Proposition \ref{CalibExist} the existence of a calibrated curve $\sigma : (-\infty,t] \to M$. Hence, we infer from Lemma \ref{CalibAB} that $\alpha_1(\sigma) = \{x_n\}$, $\sigma$ has constant $\bar{\delta}$-coordinate, and that for all $s < t$
	\begin{equation*}
		h_Z^{\rho_n\infty}(x_n, \sigma(s))=  \int_{-\infty}^s \frac{1}{2} \Vert \dot{\sigma}(\tau) - Z(\sigma(\tau))\Vert^2 \; d\tau =  \int_{-\infty}^s \frac{1}{2} |\dot{\sigma}_\delta(\tau) - Z_\delta(\sigma(\tau))|^2 \; d\tau
	\end{equation*}
	Then,
	\begin{equation*}
		\frac{d}{ds} h_Z^{\rho_n\infty}(x_n, \sigma(s))= \frac{1}{2} \Vert \dot{\sigma}(s) - Z(\sigma(s))\Vert^2
	\end{equation*}
	Moreover, Theorem \ref{CalibRegularity} shows that $h_Z^{\rho_n\infty}(x_n, \cdot)$ is differentiable at $\sigma(s)$ for all $s<t$, yielding
	\begin{align*}
		\frac{d}{ds} h_Z^{\rho_n\infty}(x_n, \sigma(s))&= dh_Z^{\rho_n\infty}(x_n, \sigma(s)).\dot{\sigma}(s) = \partial_v L_Z(\sigma(s), \dot{\sigma}(s)). \dot{\sigma}(s) \\
		&= \langle \dot{\sigma}(s) - Z(\sigma(s)), \dot{\sigma}(s) \rangle
	\end{align*}
	where $\langle \cdot , \cdot \rangle$ is the scalar product associated to the norm $\Vert \cdot \Vert$. We deduce from these identities that
	\begin{align*}
		\langle \dot{\sigma}(s) - Z(\sigma(s)), \dot{\sigma}(s) \rangle &= \frac{1}{2} \Vert \dot{\sigma}(s) - Z(\sigma(s))\Vert^2 =  \left\langle \dot{\sigma}(s) - Z(\sigma(s)) , \frac{1}{2}\dot{\sigma}(s) - \frac{1}{2}Z(\sigma(s)) \right\rangle
	\end{align*}
	and
	\begin{align*}
		0 = \langle \dot{\sigma}(s) - Z(\sigma(s)), \dot{\sigma}(s) + Z(\sigma(s)) \rangle = \Vert  \dot{\sigma}(s) \Vert^2 -  \Vert  Z(\sigma(s)) \Vert^2
	\end{align*}
	We obtain
	\begin{equation*}
		|\dot{\sigma}_\delta(s) |= \Vert  \dot{\sigma}(s) \Vert = \Vert  Z(\sigma(s)) \Vert = |Z_\delta(\sigma(s))|
	\end{equation*}
	The vector field $Z$ is directed from $\partial B^0_n$ towards $x_n$, so that $Z_\delta$ is negative on $B^0_n \setminus \{x_n\}$. In particular, it is non-null and we deduce that either $\dot{\sigma}_\delta(s) = Z_\delta(\sigma(s))$ for all $s<t$, or $\dot{\sigma}_\delta(s) = -Z_\delta(\sigma(s))$ for all $s<t$. However, we have $\alpha_1(\sigma) = \{x_n\}$. Thus, for all $s<t$
	\begin{equation*}
		\dot{\sigma}_\delta(s) = - Z_\delta(\sigma(s)) \quad \text{and} \quad \dot{\sigma}(s) = - Z(\sigma(s))
	\end{equation*}
	which extends by continuity to $s=t$. Since $\sigma(t) = y$, we conclude that for all $s \leq t$, 
	\begin{align*}
		\sigma(s) = \phi_{-Z}^{s-t}(y) = \phi_{-Z}^{s-t}\circ \phi_{-Z}^t(x) = \phi_{-Z}^s(x) = \sigma_x(s)
	\end{align*}
	We proved that $\sigma_x$ is calibrated on $(-\infty,t]$ for arbitrary. Hence, it is calibrated on $\mathbb{R}$.
\end{proof}

We conclude this section with the...
\begin{proof}[Proof of Lemma \ref{2Dz+z-}]
	Set $x^\pm$ to be the points of $A_n$ written in the coordinates $(r,\theta,x_3,..,x_d)$ as $x^\pm = (r_n \pm \frac{3\delta_n}{2}, 0, .. , 0)$. We consider the curves 
	\begin{equation} \label{2Dz+z-Dem2}
		\sigma^\pm (t) = \phi_{-Z}^t(x^\pm) \quad \text{and} \quad \gamma^\pm(t) = \mathcal{R}_t \circ \sigma^\pm(t)
	\end{equation}
	From the dynamics of $-Z$, we infer that $\alpha(\sigma^\pm) = \{z_n^\pm\}$ and $\omega(\sigma^\pm) \subset \partial B_n$. And by Lemmas \ref{CalibZ} and \ref{CalibAB}, we deduce that $\sigma^\pm$ is calibrated by $h_Z^\infty(z_n^\pm,\cdot)$ and 
	\begin{equation} \label{2Dz+z-Dem1}
		\begin{split}
			h_Z^\infty(z_n^\pm,\partial B_n) &= \lim_{\substack{ s\to-\infty \\ t\to +\infty}} h_Z^{s,t}(\sigma^\pm(s), \sigma^\pm(t)) = \int_\mathbb{R} L_Z(\sigma^\pm(\tau), \dot{\sigma}^\pm(\tau)) \; d\tau  \\
			&= \int_\mathbb{R} \frac{1}{2} \Vert \dot{\sigma}^\pm(\tau) - Z(\sigma^\pm(\tau)) \Vert^2  \; d\tau = \int_\mathbb{R} \Vert Z(\sigma^\pm(\tau)) \Vert^2  \; d\tau 
		\end{split}
	\end{equation}
	Let us show that $\Vert Z(\sigma^\pm(\tau)) \Vert = \Vert \dot{\sigma}^\pm(\tau) \Vert$. We have by definition of the curves $\sigma^\pm$ that
	\begin{align*}
		\begin{cases}
			\dot{\sigma}^\pm (\tau) = -Z (\sigma^\pm(\tau)) \\
			\sigma^\pm(0) = x^\pm
		\end{cases}
	\end{align*}
	We adopt the notation of Lemma \ref{CalibAB} and we consider the coordinate $\delta$ defined by $x_\delta = d(O_n,x)$. The symmetries of the vector field $Z$ stated in Remark \ref{SymmetryX} show that $Z_\delta = \Vert Z\Vert$ and that it only depends on $x_\delta$ for $x \in A_n$. Hence, we deduce that $\sigma^\pm_\delta$ both verify the ODE
	\begin{align*}
		\begin{cases}
			\dot{\sigma}_\delta^\pm (\tau) = -Z_\delta (\sigma_\delta^\pm(\tau)) \\
			\sigma_\delta^\pm(0) = \frac{3\delta_n}{2}
		\end{cases}
	\end{align*}
	which has a unique solution. This yields the equalities $\sigma^+_\delta = \sigma^-_\delta$ and $\dot{\sigma}^+_\delta = \dot{\sigma}^-_\delta$. Going back to \eqref{2Dz+z-Dem1}, we conclude that
	\begin{align*}
		h_Z^\infty(z_n^+,\partial B_n) = \int_\mathbb{R} \Vert Z(\sigma^+(\tau)) \Vert^2  \; d\tau = \int_\mathbb{R} \Vert Z(\sigma^-(\tau)) \Vert^2  \; d\tau = h_Z^\infty(z_n^-,\partial B_n)
	\end{align*}
	
	In order to prove the equality for the Peierls barrier $h^\infty$ associated to $L$, we use Proposition \ref{CalibZX} which shows that the curves $\gamma^\pm$ defined in \eqref{2Dz+z-Dem2} are calibrated by $h^\infty(z_n^\pm,\cdot)$ and we use the analogous of identity \eqref{CalibZXDem2} to $h^\infty(z_n^\pm,\cdot)$ which gives
	\begin{align*}
		h_Z^\infty(z_n^\pm,\partial B_n) = \lim_{t\to +\infty} h_Z^{\infty+t}(z_n^\pm, \sigma^\pm(t)) = \lim_{t\to +\infty} h^{\infty+t}(z_n^\pm, \gamma^\pm(t)) = h^\infty(z_n^\pm, \partial B_n)
	\end{align*}
\end{proof}

\subsubsection{Proof of Theorem \ref{MainC1}} \label{SubsectionPfofReg}

We establish the $C^\infty$ regularity of $u_c(t,x)$. For that purpose, we define for all subset $F$ of $M$, the subset $RF$ of $\mathbb{R} \times M$ given by 
\begin{equation}
	RF := \cup_{s\in \mathbb{R}} \{s\} \times \mathcal{R}_z(F)
\end{equation}     

Hence, we obtain
\begin{equation}
	\begin{split}
		RB^0_n &= \bigcup_{s \in \mathbb{R}} \{s\} \times \mathcal{R}_s(B^0_n)\\
		R(B^0_n\setminus\{x_n\}) &= \bigcup_{s \in \mathbb{R}} \{s\} \times \mathcal{R}_s(B^0_n \setminus\{x_n\})\\
		RB_n &= \bigcup_{s \in \mathbb{R}} \{s\} \times \mathcal{R}_s(B_n) \\
		RA_n &= \mathbb{R} \times A_n \\
		RC_n &= \mathbb{R} \times C_n \\
		RD &= \mathbb{R} \times D
	\end{split}
\end{equation}

\begin{prop} \label{RegularityBn}
	The restriction of the solution $u_c$ to the closure of $RB^0_n$ is $C^\infty$ regular and all its derivatives are null on the boundary $\partial R(B^0_n \setminus \{x_n\})$.
\end{prop}      

\begin{proof}
	Fix an integer $n \geq 0$. We first focus on the set $R(B^0_n\setminus\{x_n\})$. The proof strategy involves constructing a foliation of $R(B^0_n\setminus\{x_n\})$ by calibrated curves. According to Proposition \ref{Cnt}, in this set, $u_c(t,x) =c_n+ h^{\rho_n\infty+t}(x_n, x)$. Thus, we will aim to prove the calibration for the barrier $h^{\rho_n\infty+t}(x_n, \cdot)$ on the set $\mathcal{R}_t(B^0_n)$. 
	
	Let $(t,x)$ be in $R(B^0_n\setminus\{x_n\})$. We set the curves $\sigma_y$ and $\gamma_{(t,x)} : \mathbb{R} \to M$ defined by
	\begin{equation}
		y=  \phi_{-Z}^{-t} \circ \mathcal{R}_t^{-1}(x), \quad \sigma_y(\tau) = \phi_{-Z}^\tau(y) \quad \text{and} \quad \gamma_{(t,x)}(\tau) = \mathcal{R}_\tau \circ \sigma_y(\tau)
	\end{equation}
	Lemma \ref{CalibZ} states that $\sigma_y$ is calibrated by $h_Z^{\rho_n\infty}(x_n,\cdot)$. Hence, Lemma \ref{CalibZX} states that $\gamma_{(t,x)}$ is calibrated by $h^{\rho_n\infty}(x_n,\cdot)$. Moreover, we have
	\begin{equation*}
		\gamma_{(t,x)}(t)= \mathcal{R}_t \circ \sigma_y(t) = \mathcal{R}_t  \circ \phi_{-Z}^t(y) = \mathcal{R}_t \circ \phi_{-Z}^t \circ\phi_{-Z}^{-t} \circ \mathcal{R}_t^{-1}(x) =x
	\end{equation*}
	Hence, we deduce from Theorem \ref{CalibRegularity} that $u_c$ is differentiable at $(t,x)$ and
	\begin{equation} \label{duc}
		\begin{split}
			du_c(t,x) &= (\partial_tu_c(t,x) , d_xu_c(t,x)) = \big(  -H (t,x,d_xu_c(t,x)) , \partial_vL (t,x, \dot{\gamma}_{(t,x)}(t) )  \big) \\
			&= \left(  -H (t,x,d_xu_c(t,x)) , \partial_vL \left(t,x, \left.\frac{d}{d\tau}\right|_{\tau =t}\big(\mathcal{R}_\tau \circ \phi_{-Z}^{\tau-t} \circ \mathcal{R}_t^{-1}(x)\big) \right)  \right) 
		\end{split}
	\end{equation}
	Therefore, $u_c$ is $C^\infty$ regular on $R(B^0_n\setminus\{x_n\})$.\\
	
	This formula extends to the closure of $R(B^0_n\setminus\{x_n\})$. Indeed, we have $Z_{|\partial  R(B^0_n\setminus\{x_n\})} = 0$ and for all $(t,x) \in \partial R(B^0_n\setminus\{x_n\})$, the curve $f_{t,\tau}(x) = \mathcal{R}_{t,\tau}(x)$ is calibrated and of null action. Hence, the restriction of $u_c$ to the closure of $R(B^0_n\setminus\{x_n\})$ is of $C^\infty$ regularity. And in particular, it is $C^\infty$ on $RB^0_n$ and in particular, at $(t,\mathcal{R}_t(x_n))$.
	
	Let us compute its derivative on the boundary $\partial RB^0_n$. For all $(t,x) \in RB^0_n$, if we set $y=  \phi_{-Z}^{-t} \circ \mathcal{R}_t^{-1}(x)$, we have
	\begin{align*}
		\left.\frac{d}{d\tau}\right|_{\tau =t}\big(\mathcal{R}_\tau \circ \phi_{-Z}^{\tau} (y)\big) &= \frac{d}{dt}\mathcal{R}_t \circ \phi_{-Z}^{t} (y) + d\mathcal{R}_t.\frac{d}{dt}\phi_{-Z}^{t} (y) \\
		&=\frac{d}{dt}\mathcal{R}_t \circ \mathcal{R}_t^{-1}(x) - d\mathcal{R}_t.Z  \mathcal{R}_t^{-1}(x)\\
		&= Y_t(x) - d\mathcal{R}_t.Z  \mathcal{R}_t^{-1}(x)
	\end{align*}
	where $Y_t$ is the vector field associated to the isotopy $\mathcal{R}_t$. Additionally, we know from \eqref{XFormula} and \eqref{LagMane} that
	\begin{align*}
		\partial_vL(t,x,v) = v- X_t(x) = v- Y_t(x) - d\mathcal{R}_t.Z  \mathcal{R}_t^{-1}(x)
	\end{align*}
	Thus, we obtain
	\begin{align*}
		 d_xu_c(t,x) &=  \partial_vL \left(t,x, \left.\frac{d}{d\tau}\right|_{\tau =t}\big(\mathcal{R}_\tau \circ \phi_{-Z}^{\tau-t} \circ \mathcal{R}_t^{-1}(x)\big) \right) \\
		 &= Y_t(x) - d\mathcal{R}_t.Z  \mathcal{R}_t^{-1}(x) - Y_t(x) - d\mathcal{R}_t.Z  \mathcal{R}_t^{-1}(x) = -2 d\mathcal{R}_t.Z  \mathcal{R}_t^{-1}(x)
	\end{align*}
	Moreover, we have the limits
	\begin{align*}
		\lim_{(t,x) \to \partial RB^0_n} d( B^0_n, \mathcal{R}_t^{-1}(x)) = 0 \quad \text{and} \quad \lim_{y \to B^0_n} Z =0
	\end{align*}
	where the last limit on $Z$ is in the $C^\infty$-topology. Therefore, we deduce that $d_xu_c$ converges to $0$ in the $C^\infty$ topology as $(t,x)$ converges to $\partial B^0_n$.
	
	We also compute $\partial_tu_c$ using the expression \eqref{ManeHamiltonianFormula} of the Hamiltonian $H$. This yields
	\begin{align*}
		\partial_tu_c(t,x) &=  -H (t,x,d_xu_c(t,x)) = - H( t,x,  -2 d\mathcal{R}_t.Z  \mathcal{R}_t^{-1}(x)) \\
		&= \frac{1}{2} \Vert -2 d\mathcal{R}_t.Z  \mathcal{R}_t^{-1}(x) + X_t(x) \Vert^2 - \frac{1}{2} \Vert X_t(x) \Vert^2 \\
		&= \frac{1}{2} \Vert Y_t(x) - d\mathcal{R}_t.Z  \mathcal{R}_t^{-1}(x) \Vert^2 - \frac{1}{2} \Vert Y_t(x) - d\mathcal{R}_t.Z  \mathcal{R}_t^{-1}(x) \Vert^2 \\
		&= -\langle Y_t(x), d\mathcal{R}_t.Z  \mathcal{R}_t^{-1}(x) \rangle
	\end{align*}
	which also converges to $0$ in the $C^\infty$ topology as $(t,x)$ goes to $\partial RB^0_n$. We proved that $du_c$ converges to $0$ in the $C^\infty$ topology as $(t,x) \in RB^0_n$ goes to the boundary $\partial RB^0_n$.	
\end{proof}
	
\begin{prop} \label{RegularityAn}
	The restriction of the solution $u_c$ to the closure of $RA_n := \mathbb{R} \times A_n$ is $C^\infty$ regular and all its derivatives are null on the boundary $\partial RB^0_n$.
\end{prop}  

\begin{proof}
	In the 2D case, we have shown in Proposition \ref{Cnt} that $u_c$ is constant on $RA_n$ which implies the result.
	
	In the 3D case, the same Proposition \ref{Cnt} asserts that for all $(t,x) \in RA_n$, $u(t,x) = h^{\infty+t}(z_\infty,x)$. Applying the second property of Lemma \ref{CalibZ}, we prove analogously to Proposition \ref{RegularityBn} that the curves $\gamma_{t,x}(\tau) = \mathcal{R}_\tau \circ \phi_{-Z}^{\tau-t} \circ \mathcal{R}_t^{-1}(x)$ are calibrated by $u_c$, and we obtain the same formula \eqref{duc} on $du_c$ which also converges to $0$ at the boundary $\partial RA_n$ of $RA_n$ in the $C^\infty$ topology.
\end{proof}

\begin{proof}[Proof of Theorem \ref{MainC1}]
	Gathering the results of Propositions \ref{Cnt}, \ref{RegularityBn} and \ref{RegularityAn}, we have shown that $u_c$ is $C^\infty$ regular on the set 
	\begin{align*}
		RD \cup \bigcup_{n\geq 0} \Big( RA_n \cup RB^0_n \cup (RB_n \setminus RB^0_n) \Big) = M \setminus \left( \bigcup_{n\geq 0} \partial RC_n \cup  \partial RB_n \cup \partial RB^0_n \right)
	\end{align*}
	In order to complete the proof of Theorem \ref{MainC1}, it suffices to prove $C^\infty$ regularity on $\partial RC_n$, $\partial RB_n$ and $\partial RB^0_n$.\\

	\textit{On $\partial RC_n$}. We have that $u_c$ is locally constant on $RD$ and all the derivatives of $du_{c|RA_n}$ converge to $0$ at the external boundary $\partial RC_n$ of $RA_n$. Hence, $u_c$ is $C^\infty$ at $\partial RC_n$ with null derivatives.

	\textit{On $\partial RB^0_n$ and $\partial RB_n$}. We have that $u_c$ is constant on $RB_n \setminus RB^0_n$, and all the derivatives of $du_{c|RA_n}$ and $du_{c|RB^0_n}$ converge to $0$ at their boundary, and in particular at $\partial RB^0_n$ and $\partial RB_n$. Hence, $u_c$ is $C^\infty$ at $\partial RB^0_n$ and $\partial RB_n$ with null derivatives.\\

	This concludes the proof of the main Theorem \ref{MainC1}.
\end{proof}

\begin{rem}
	\begin{enumerate}	
		\item Lemma \ref{CalibZ} can be understood as follows: The \Mane set $\tilde{\mathcal{N}}_Z$  corresponding to the autonomous component $Z$ of $X_t$ is symmetric with respect to the zero section of $TM$. The radial symmetry imposed on $Z$ yields a \Mane set that, when viewed radially, closely resembles what is observed in one-dimensional systems, specifically the phase portrait of a pendulum. A radial section of $\tilde{\mathcal{N}}_Z$ around $r_n$ is represented in Figure \ref{FigureMané}.
		
	\begin{figure}[!htbp]
		\centering
		\includegraphics[width=0.9\linewidth]{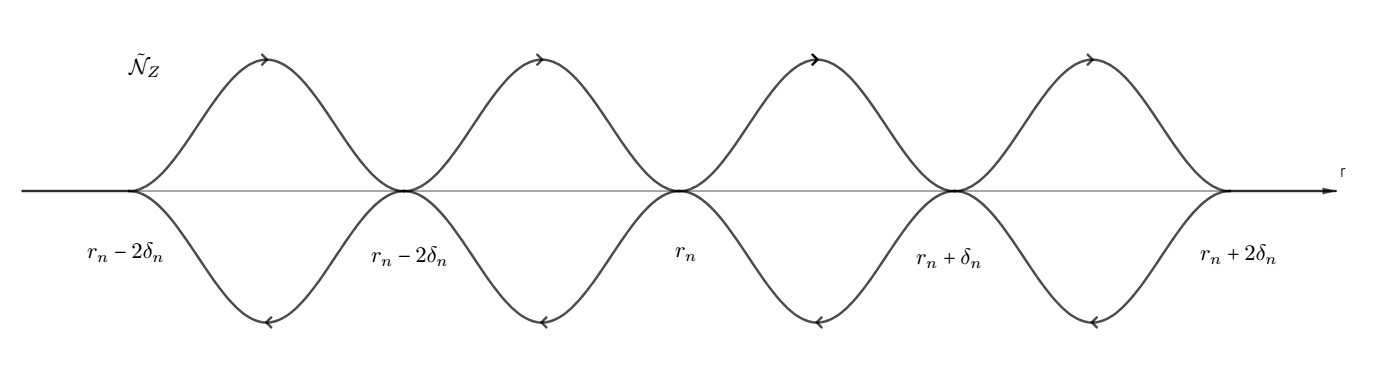}
		\caption{Radial Section at angle $\theta=0$ of the \Mane Set $\tilde{\mathcal{N}}_Z$.}
		\label{FigureMané}
	\end{figure}
		
		\item Note that the obtained $C^1$ regularity is highly unstable. A generic perturbation of the constants $c_n$ induces irregularities. Specifically, if the equality $c_n + h^{\rho_n\infty}(x_n,x) = h^\infty(z_\infty,x)$ occurs outside the Peierls set $\mathcal{A}_0$, then the loss of regularity is inevitable. In other words, the set of $C^1$ recurrent viscosity solutions is sparse within the non-wandering set $\Omega(\mathcal{T})$ of the Lax-Oleinik operator $\mathcal{T}$. 
		
		\item Similarly, any perturbation of the Hamiltonian that disrupts its symmetries may also result in a loss of regularity, and such a Hamiltonian may not possess any $C^1$ recurrent viscosity solution. There are examples of Hamiltonian systems with no regular elements in their non-wandering set $\Omega(\mathcal{T})$. The simplest example is the autonomous simple pendulum, where the unique element of $\Omega(\mathcal{T})$ is a Lipschitz weak-KAM solution. 
		
		\item Additionally, obtaining $C^{1,1}$ recurrent viscosity solutions is significantly easier than obtaining $C^2$ or more regular solutions. This is because if the infimum of two Peierls barriers occurs on the Peierls set, $C^{1,1}$ regularity is guaranteed. Achieving higher regularity requires modifying the Hamiltonian to ensure that convergence to such points occurs only via parabolic trajectories.
		
		\item From a symplectic perspective, the smooth recurrent locally viscosity solution $u_c$ gives rise to a Lagrangian submanifold, namely the graph of $du_c$ in the cotangent bundle $T^*M$, which is ($C^1$-) recurrent under the Hamiltonian flow $\phi_H$.
	\end{enumerate}
\end{rem}

\paragraph{Acknowledgement.} The author expresses deep gratitude to Marie-Claude Arnaud for her careful reading, and to Jaques Féjoz for very insightful discussions.

\bibliographystyle{alpha}
\addcontentsline{toc}{section}{References}
\bibliography{Biblio}

\begin{thebibliography}{CISM13}

\bibitem[Bar90]{MR1014729}
G.~Barles.
\newblock Regularity results for first order {H}amilton-{J}acobi equations.
\newblock {\em Differential Integral Equations}, 3(1):103--125, 1990.

\bibitem[Ber08]{MR2393423}
Patrick Bernard.
\newblock The dynamics of pseudographs in convex {H}amiltonian systems.
\newblock {\em J. Amer. Math. Soc.}, 21(3):615--669, 2008.

\bibitem[BR04]{MR2041603}
Patrick Bernard and Jean-Michel Roquejoffre.
\newblock Convergence to time-periodic solutions in time-periodic
  {H}amilton-{J}acobi equations on the circle.
\newblock {\em Comm. Partial Differential Equations}, 29(3-4):457--469, 2004.

\bibitem[BS95]{MR1332404}
Jorge Buescu and Ian Stewart.
\newblock Liapunov stability and adding machines.
\newblock {\em Ergodic Theory Dynam. Systems}, 15(2):271--290, 1995.

\bibitem[BS00]{MR1752423}
G.~Barles and Panagiotis~E. Souganidis.
\newblock On the large time behavior of solutions of {H}amilton-{J}acobi
  equations.
\newblock {\em SIAM J. Math. Anal.}, 31(4):925--939, 2000.

\bibitem[Cha25a]{Representation}
Skander Charfi.
\newblock Representation of global viscosity solutions for tonelli
  hamiltonians, 2025.

\bibitem[Cha25b]{charf2025}
Skander Charfi.
\newblock {\em Récurrences des sous-variétés lagrangiennes et des solutions
  de viscosité sous des actions symplectiques qui dévient la verticale}.
\newblock PhD thesis, 2025.
\newblock Thèse de doctorat dirigée par Arnaud, Marie-claude et Fejoz,
  Jacques Mathematiques Université Paris Cité 2025.

\bibitem[CI99]{MR1720372}
Gonzalo Contreras and Renato Iturriaga.
\newblock {\em Global minimizers of autonomous {L}agrangians}.
\newblock 22$^{\rm o}$ Col\'{o}quio Brasileiro de Matem\'{a}tica. [22nd
  Brazilian Mathematics Colloquium]. Instituto de Matem\'{a}tica Pura e
  Aplicada (IMPA), Rio de Janeiro, 1999.

\bibitem[CISM13]{CIS}
Gonzalo Contreras, Renato Iturriaga, and Hector Sanchez-Morgado.
\newblock Weak solutions of the hamilton-jacobi equation for time periodic
  lagrangians, 2013.

\bibitem[CL83]{MR0690039}
Michael~G. Crandall and Pierre-Louis Lions.
\newblock Viscosity solutions of {H}amilton-{J}acobi equations.
\newblock {\em Trans. Amer. Math. Soc.}, 277(1):1--42, 1983.

\bibitem[Dow05]{MR2180227}
Tomasz Downarowicz.
\newblock Survey of odometers and {T}oeplitz flows.
\newblock In {\em Algebraic and topological dynamics}, volume 385 of {\em
  Contemp. Math.}, pages 7--37. Amer. Math. Soc., Providence, RI, 2005.

\bibitem[Fat98]{MR1650261}
Albert Fathi.
\newblock Sur la convergence du semi-groupe de {L}ax-{O}leinik.
\newblock {\em C. R. Acad. Sci. Paris S\'{e}r. I Math.}, 327(3):267--270, 1998.

\bibitem[Fat08]{fathi2008weak}
A.~Fathi.
\newblock {\em The Weak KAM Theorem in Lagrangian Dynamics}.
\newblock (Book in Preparation). Cambridge Studies in Advanced Mathematics.
  Cambridge University Press, 2008.

\bibitem[Fle69]{MR0235269}
Wendell~H. Fleming.
\newblock The {C}auchy problem for a nonlinear first order partial differential
  equation.
\newblock {\em J. Differential Equations}, 5:515--530, 1969.

\bibitem[FM00]{MR1792479}
Albert Fathi and John~N. Mather.
\newblock Failure of convergence of the {L}ax-{O}leinik semi-group in the
  time-periodic case.
\newblock {\em Bull. Soc. Math. France}, 128(3):473--483, 2000.

\bibitem[HR79]{MR0551496}
Edwin Hewitt and Kenneth~A. Ross.
\newblock {\em Abstract harmonic analysis. {V}ol. {I}}, volume 115 of {\em
  Grundlehren der Mathematischen Wissenschaften}.
\newblock Springer-Verlag, Berlin-New York, second edition, 1979.
\newblock Structure of topological groups, integration theory, group
  representations.

\bibitem[HV22]{MR4492832}
Roman Hric and Miroslav V\'{y}bo\v{s}\v{t}ok.
\newblock Classification of odometers: a short elementary proof.
\newblock {\em Ann. Math. Sil.}, 36(2):184--192, 2022.

\bibitem[Lio82]{MR667669}
Pierre-Louis Lions.
\newblock {\em Generalized solutions of {H}amilton-{J}acobi equations},
  volume~69 of {\em Research Notes in Mathematics}.
\newblock Pitman (Advanced Publishing Program), Boston, Mass.-London, 1982.

\bibitem[Mat91]{MR1109661}
John~N. Mather.
\newblock Action minimizing invariant measures for positive definite
  {L}agrangian systems.
\newblock {\em Math. Z.}, 207(2):169--207, 1991.

\bibitem[Mat93]{MR1275203}
John~N. Mather.
\newblock Variational construction of connecting orbits.
\newblock {\em Ann. Inst. Fourier (Grenoble)}, 43(5):1349--1386, 1993.

\bibitem[Mn91]{mane1990global}
Ricardo Ma\~n\'e.
\newblock {\em Global variational methods in conservative dynamics}.
\newblock Instituto de matema\'atica pura e aplicada, 1991.

\bibitem[Mn92]{MR1166538}
Ricardo Ma\~{n}\'{e}.
\newblock On the minimizing measures of {L}agrangian dynamical systems.
\newblock {\em Nonlinearity}, 5(3):623--638, 1992.

\bibitem[Mn97]{MR1479499}
Ricardo Ma\~{n}\'{e}.
\newblock Lagrangian flows: the dynamics of globally minimizing orbits.
\newblock {\em Bol. Soc. Brasil. Mat. (N.S.)}, 28(2):141--153, 1997.

\bibitem[NR97a]{MR1457088}
Gawtum Namah and Jean-Michel Roquejoffre.
\newblock Comportement asymptotique des solutions d'une classe d'\'{e}quations
  paraboliques et de {H}amilton-{J}acobi.
\newblock {\em C. R. Acad. Sci. Paris S\'{e}r. I Math.}, 324(12):1367--1370,
  1997.

\bibitem[NR97b]{MR1485736}
Gawtum Namah and Jean-Michel Roquejoffre.
\newblock Convergence to periodic fronts in a class of semilinear parabolic
  equations.
\newblock {\em NoDEA Nonlinear Differential Equations Appl.}, 4(4):521--536,
  1997.

\bibitem[NR99]{MR1680905}
Gawtum Namah and Jean-Michel Roquejoffre.
\newblock Remarks on the long time behaviour of the solutions of
  {H}amilton-{J}acobi equations.
\newblock {\em Comm. Partial Differential Equations}, 24(5-6):883--893, 1999.

\bibitem[Roq98]{MR1646936}
Jean-Michel Roquejoffre.
\newblock Comportement asymptotique des solutions d'\'{e}quations de
  {H}amilton-{J}acobi monodimensionnelles.
\newblock {\em C. R. Acad. Sci. Paris S\'{e}r. I Math.}, 326(2):185--189, 1998.

\end{thebibliography}

\end{document}